\documentclass[12pt,reqno]{amsart}
\usepackage[margin=1in]{geometry}
\usepackage{amsmath,amssymb,amsthm,graphicx,amsxtra, setspace}
\usepackage[utf8]{inputenc}
\usepackage{mathrsfs}
\usepackage{hyperref}
\usepackage{upgreek}
\usepackage{xcolor}
\usepackage{mathtools}
\allowdisplaybreaks

\usepackage[pagewise]{lineno}

\newtheorem{theorem}{Theorem}[section]
\newtheorem{lemma}[theorem]{Lemma}
\newtheorem{proposition}[theorem]{Proposition}
\newtheorem{assumption}[theorem]{Assumption}
\newtheorem{corollary}[theorem]{Corollary}
\newtheorem{definition}[theorem]{Definition}

\newtheorem{remark}[theorem]{Remark}

\allowdisplaybreaks

\let\originalleft\left
\let\originalright\right
\renewcommand{\left}{\mathopen{}\mathclose\bgroup\originalleft}
\renewcommand{\right}{\aftergroup\egroup\originalright}

\renewcommand{\d}{\/\mathrm{d}\/}

\def\w{\textbf{W}^{\varepsilon}_{{\theta}^{\varepsilon}}}

\def\L{\mathbb{L}}
\def\A{\mathrm{A}}
\def\I{\mathrm{I}}

\def\C{\mathrm{C}}
\def\f{\boldsymbol{f}}
\def\J{\mathrm{J}}
\def\B{\mathrm{B}}
\def\D{\mathrm{D}}
\def\y{\boldsymbol{y}}

\def\E{\mathbb{E}}

\def\x{\boldsymbol{x}}

\def\g{\boldsymbol{g}}

\def\h{\boldsymbol{h}}
\def\z{\boldsymbol{z}}
\def\v{\boldsymbol{v}}
\def\V{\mathbb{v}}
\def\w{\boldsymbol{w}}
\def\W{\mathrm{W}}

\def\N{\mathbb{N}}
\def\r{\mathrm{r}}

\def\V{\mathbb{V}}
\def\wi{\widetilde}

\def\u{\mathrm{U}}
\def\P{\mathrm{P}}
\def\u{\boldsymbol{u}}
\def\H{\mathbb{H}}

\newcommand{\R}{\mathbb{R}}

\renewcommand{\d}{\/\mathrm{d}\/}

\newcommand{\Addresses}{{
		\footnote{
			
			\noindent \textsuperscript{1,2}Department of Mathematics, Indian Institute of Technology Roorkee-IIT Roorkee,
			Haridwar Highway, Roorkee, Uttarakhand 247667, INDIA.\par\nopagebreak
			\noindent  \textit{e-mail:} \texttt{Manil T. Mohan: maniltmohan@ma.iitr.ac.in, maniltmohan@gmail.com.}
			
			\textit{e-mail:} \texttt{Kush Kinra: kkinra@ma.iitr.ac.in.}
			
			\noindent \textsuperscript{*}Corresponding author.
			
			\textit{Key words:} Stochastic convective Brinkman-Forchheimer equations, unbounded domains, cylindrical Wiener process, random dynamical system, absorbing sets, asymptotically compactness, random attractors. 
			
			Mathematics Subject Classification (2020): Primary 35B41, 35Q35; Secondary 37L55, 37N10, 35R60.

}}}

\begin{document}
	
	\title[Random attractors for 2D SCBF equations]{$\H^1$-Random attractors for 2D stochastic convective Brinkman-Forchheimer equations in unbounded domains
		\Addresses}
	\author[K. Kinra and M. T. Mohan]
	{Kush Kinra\textsuperscript{1} and Manil T. Mohan\textsuperscript{2*}}

	\maketitle
	\begin{abstract}
		The asymptotic behavior of solutions of two dimensional stochastic convective Brinkman-Forchheimer (2D SCBF) equations in unbounded domains is discussed in this work (for example, Poincar\'e domains). We first prove the existence of $\H^1$-random attractors for the stochastic flow generated by 2D SCBF equations (for the absorption exponent $r\in[1,3]$) perturbed by an additive noise on Poincar\'e domains. Furthermore, we deduce the existence of a unique invariant measure in $\H^1$ for the 2D SCBF equations defined on Poincar\'e domains. In addition, a remark on the extension of these results to general unbounded domains is also discussed. Finally, for 2D SCBF equations forced by additive one-dimensional  Wiener noise, we prove the upper semicontinuity of the random attractors, when the domain changes from bounded to unbounded (Poincar\'e). 
	\end{abstract}

	\section{Introduction} \label{sec1}\setcounter{equation}{0}
	The analysis of long time behavior of infinite dimensional dynamical systems including   two-dimensional Navier-Stokes equations (2D NSE) is well investigated in \cite{R.Temam}. The existence of a global attractor for  2D NSE on some unbounded domains like Poincar\'e domains is obtained in \cite{Rosa}. Following the work \cite{Rosa}, the existence of global attractors for the two-dimensional convective Brinkman-Forchheimer (CBF) equations in $\H,$ for the absorption exponent $r\in[1,\infty)$ and the external forcing $\f\in\V'$ (from the Gelfand triple $\V\subset\H\subset\V'$, see section \ref{sec2} for functional framework) in Poincar\'e like domain is proved in \cite{Mohan2}. The upper semicontinuity of global attractor with respect to domain, that is, when domain changes from bounded to unbounded (Poincar\'e) domain is also established in \cite{Mohan2}. In \cite{NJu}, the author extended the work \cite{Rosa} and  showed that if the external forcing term is in the natural space $\H$, then the global attractor for  2D NSE is compact not only in the $\L^2$-norm but also in the $\H^1$-norm, and it attracts all bounded sets in $\H$ in the metric of $\V$.  Whenever the forcing $\f\in\H$, the existence of $\H^1$-global attractors for 2D  CBF equations in unbounded domains for the absorption exponent $r\in[1,3]$ is proved in \cite{MTM2}. It is well-investigated in the literature that a large class of stochastic partial differential equations (SPDEs) generate random dynamical systems (RDS) (cf. \cite{Arnold}).  The analysis of infinite dimensional RDS is also an essential branch in the study of qualitative properties of SPDEs (see \cite{BCF,Crauel,CF}, etc for more details).	The random dynamics of 2D stochastic Navier-Stokes equations (SNSE) has been carried out in the works \cite{BCF,BCLLLR,BGT,BL1,BL,Crauel1,CDF,CF}, etc and the references therein. 
	
	In this work, we are concerned about the long time behavior of 2D SCBF equations in Poincar\'e domains $\mathcal{O}\subset \mathbb{R}^2$ and general unbounded domains. Let $\mathcal{O}$ satisfy the following assumption:
	\begin{assumption}\label{assumpO}
		Let $\mathcal{O}$ be an open, connected and may be unbounded subset of $\R^2$, the boundary of which is uniformly of class $\mathrm{C}^3$ (cf. \cite{Heywood}). For the domain $\mathcal{O}$, we also assume that, there exists a positive constant $\lambda_1 $ such that the following Poincar\'e inequality  is satisfied:
		\begin{align}\label{2.1}
			\lambda_1\int_{\mathcal{O}} |\psi(x)|^2 \d x \leq \int_{\mathcal{O}} |\nabla \psi(x)|^2 \d x,  \ \text{ for all } \  \psi \in \H^{1}_0 (\mathcal{O}).
		\end{align}
		A domain in which Poincar\'e inequality is satisfied, we call it as a Poincar\'e domain and if  $\mathcal{O}$ is bounded in some direction, then the Poincar\'e inequality \eqref{2.1} holds. For example, one can consider  $\mathcal{O}=\R\times(-L,L),$ $L>0$ (see page 306, \cite{R.Temam}).
	\end{assumption}
	The convective Brinkman-Forchheimer (CBF) equations characterize the motion of incompressible viscous fluid  through a rigid, homogeneous, isotropic, porous medium.	The 2D SCBF equations perturbed by additive noise are given by 
	\begin{equation}\label{SCBF}
		\left\{
		\begin{aligned}
			\d\u&=[\mu \Delta\u-(\u\cdot\nabla)\u-\alpha\u-\beta|\u|^{r-1}\u-\nabla p+\boldsymbol{f}]\d t + \d\mathrm{W}(t),  \text{ in } \ \mathcal{O}\times(0,\infty), \\ \nabla\cdot\u&=0, \ \text{ in } \ \mathcal{O}\times(0,\infty), \\
			\u&=\boldsymbol{0},\ \ \text{ on } \ \partial\mathcal{O}\times[0,\infty), \\
			\u(0)&=\x, \ \text{ in } \ \mathcal{O},
		\end{aligned}
		\right.
	\end{equation} where $\W(\cdot)$  is a given Hilbert space valued Wiener process on some given filtered probability space  $ (\Omega, \mathscr{F}, \{\mathscr{F}_t\}_{t\in \R}, \mathbb{P})$ whose properties will be specified in section \ref{sec2}. Here $\u(x,t) \in \R^2$, $p(x,t)\in\R$ and $\f(x,t)\in\R^2$ represent the velocity field, the pressure field and an external forcing, respectively, at position $x$ and time $t$. The constant $\mu>0$ represents the Brinkman coefficient (effective viscosity), the positive constants $\alpha$ and $\beta$ represent the Darcy (permeability of porous medium) and Forchheimer coefficients, respectively. The exponent $r\in[1,\infty)$ is called the absorption exponent and  $r=3$ is known as the critical exponent. The exponents $r=1,2,3$ are physically relevant and they represent linear, quadratic and cubic growths, receptively. For $\alpha=\beta=0$, we obtain the classical 2D stochastic Navier-Stokes equations, and one can consider the system \eqref{SCBF} as a damped SNSE with the nonlinear damping $\alpha\u+\beta|\u|^{r-1}\u$. For the unique solvability of the deterministic two- and three-dimensional CBF equations in bounded and periodic domains, the interested readers are referred to see \cite{AO,FHR,HR,KT2,Mohan1}, etc and for its stochastic counterpart, see \cite{MTM}. 
	
	In most of the results available in the literature, the existence of random attractors is heavily based on random compact attracting set (using compact embedding). But in this work, we are not able to find random compact attracting set due to following two reasons:
	\begin{itemize}
		\item [(i)]For the stochastic 2D NSE and related models, for $\f\in\H$, the existence of random absorbing sets in more regular spaces than $\V$ is not available (for example in $\D(\A^{s})$, for $s>\frac{1}{2}$).
		\item [(ii)] The domain $\mathcal{O}$ is unbounded.
	\end{itemize}
	In the deterministic case, the above mentioned difficulties were resolved by different methods, cf.  \cite{Abergel, Ghidaglia,Rosa}, etc for the autonomous case and \cite{CLR1, CLR2}, etc for the non-autonomous case. For SPDEs, the methods available in the deterministic case have also been  generalized by several authors (see for example, \cite{BL, BLL,BLW, Wang}, etc).  The concept of an asymptotically compact cocycle was introduced in \cite{CLR1} and the authors proved the existence of attractors fornon-autonomous 2D NSE. Later, several authors used this method to prove the existence of random attractors in unbounded domains, see  for example \cite{BGT,BLL,BLW,FY,LYZ, LXS, KM,Mohan2, MTM2, Slavik, SLHZ,Wang} etc. The existence of a unique random attractor  for the 2D and 3D SCBF equations \eqref{SCBF} perturbed by additive rough noise in $\H$ is proved in \cite{KM}.
	
	Recently, for $r\in[1,3]$, the existence of random attractors in $\H$ and $\V$ for the stochastic flow generated by 2D SCBF equations in bounded domains perturbed by small additive smooth noise is obtained in \cite{KM1}. The authors in \cite{KM1} also proved the upper semicontinuity of the  random attractor, when the noise term goes to zero. In this work, first we prove the existence of random attractors for 2D SCBF equations for $r\in[1,3]$ in $\V$ on unbounded (Poincar\'e) domains. We use the concepts developed in \cite{BL1} to prove the desired results for the existence of random attractors. In \cite{CF}, authors proved that the existence of compact invariant random set is a sufficient condition for the existence of invariant measures. Since, random attractor itself is a compact invariant set, the existence of invariant measures is confirmed. In addition, we prove the uniqueness of  invariant measure for our model.
	
{Next, we prove the upper semicontinuity of the random attractors with respect to domain, that is, when domain changes from bounded to unbounded (Poincar\'e) domain. Due to some technical difficulty (related to the RKHS of Wiener process and the cut-off problem), we are not able to prove the upper semicontinuity of random attractors with respect to domain for the system \eqref{SCBF}. But, we are able to prove if the noise in \eqref{SCBF} replaced by a finite dimensional noise (see \eqref{SCBF_in} below). The existence of a unique random attractor for 2D SCBF equations perturbed by finite dimensional noise is proved in \cite{KM7} (see Section 5, \cite{KM7}). Upper semicontinuity of global attractors with respect to domain for some physically relevant deterministic models such as Klein-Gordon-Schr\"odinger equation, 2D CBF equations and 2D NSE is established in \cite{LW}, \cite{Mohan2} and \cite{ZD}, respectively. We extend the same idea to prove the upper semicontinuity of random attractors with respect to domain for 2D SCBF equations.} 
	 
{In the literature, the concept of upper semicontinuity of random attractors with respect to domain was introduced in \cite{LL} and the authors proved upper semicontinuity of random attractors with respect to domain for stochastic FitzHugh-Nagumo equations (SFNE). Later, using this concept, upper semicontinuity of random attractors with respect to domain for g-Navier-Stokes equations (g-NSE) was proved in \cite{LL1}. While using the concept introduced in \cite{LL} to the models SFNE and g-NSE, the authors have not taken care of the boundary conditions, while proving the strong convergence of the solutions on the sequence of bounded domains to the solution on the unbounded domain, which is a crucial step in the theory. Due to this reason, it appears to us that the method developed in the works \cite{LL,LL1}, etc have serious flaws, and we have not used the theory  established  in \cite{LL}. But, one can prove the upper semicontinuity of random attractors with respect to domain for SFNE and g-NSE using the same method which we apply for 2D SCBF equation in Section \ref{sec6}.}

	Let us now explain the difficulty for extending the results obtained in this work for the case $r>3$. The following equality 
	\begin{align}\label{AvC}
		&\int_{\mathcal{O}}(-\Delta\v(x))\cdot|\v(x)|^{r-1}\v(x)\d x\nonumber\\&=\int_{\mathcal{O}}|\nabla\v(x)|^2|\v(x)|^{r-1}\d x+4\left[\frac{r-1}{(r+1)^2}\right]\int_{\mathcal{O}}|\nabla|\v(x)|^{\frac{r+1}{2}}|^2\d x
	\end{align} is useful in the whole space or periodic domains. But, we are not able to use the above equality in bounded domains as $\mathcal{P}(|\u|^{r-1}\u)$ need not be zero on the boundary, and $\mathcal{P}$ and $-\Delta$ are not necessarily commuting (for a counterexample, see Example 2.19, \cite{JCR4}). In order to prove the existence of random attractors for  2D SCBF equations in $\V$, we do not  use the above equality for $r\in[1,3]$. Due to this difficulty, it appears to us that the results obtained in the works \cite{HZ,BY} (random exponential attractors and random attractors for the 3D damped NSE) hold true in periodic domains only. 
	\vskip 2mm
	\noindent
	\textbf{Aims of the work:}   The major aims and novelties of this work are:
	\begin{itemize}
		\item [(i)] Existence of a unique random attractor for the system \eqref{SCBF} with $r\in[1,3]$ on unbounded Poincar\'e domains in $\V$.
		\item [(ii)] Existence of a unique invariant measure for the system \eqref{SCBF} in $\V$.
		\item [(iii)]  Upper semicontinuity of the random attractors with respect to domain, that is, when domain changes from bounded to unbounded (Poincar\'e) domain.
	\end{itemize}
	\vskip 2mm
	The rest of the paper is organized as follows. In section \ref{sec2}, we first define the function spaces, which are needed in this work. Then, we define the linear and nonlinear operators and discuss their properties. Abstract formulation of the 2D SCBF equations \eqref{SCBF}, the assumption on the Reproducing Kernel Hilbert space (RKHS) and the solvability results for the 2D SCBF equations are also provided in the same section. In section \ref{sec3}, we define metric dynamical system and random dynamical system for our model. Section \ref{sec4} is devoted for establishing the existence of random attractors for the 2D SCBF equations for $r\in[1,3]$ in $\V$ by proving the asymptotically compactness property of RDS $\varphi$ in $\V$ (Theorem \ref{V_asymptotically}). In section \ref{sec5}, we establish the existence of a unique  invariant measure the 2D SCBF equations \eqref{SCBF} in $\V$ (Theorem \ref{thm6.3}). {In the final section, we first find the random absorbing set (Lemmas \ref{Absorb} and \ref{Absorb1}) and uniform tail estimates (Lemmas \ref{largeradius} and \ref{largeradius1}) for the solution of 2D SCBF equations. Then, we prove the upper semicontinuity of the random attractor, when domain changes from bounded to unbounded (Theorem \ref{Main-T}).}

	\section{Mathematical Formulation}\label{sec2}\setcounter{equation}{0}
	In this section, we provide the necessary function spaces needed to obtain the existence of random attractors for the 2D SCBF equations \eqref{SCBF}. 
	\subsection{Function spaces} Let $\mathcal{O}\subset\R^2$ and satisfies Assumption \ref{assumpO} (may be bounded or unbounded). We define the space $$\mathcal{V}:=\{\u\in\C_0^{\infty}(\mathcal{O},\R^2):\nabla\cdot\u=0\},$$ where $\C_0^{\infty}(\mathcal{O};\R^2)$ denote the space of all infinitely differentiable functions  ($\R^2$-valued) with compact support in $\mathcal{O}$. Let $\H$, $\V$ and $\wi\L^p$ denote the completion of $\mathcal{V}$ in 	$\mathrm{L}^2(\mathcal{O};\R^2):=\L^2(\mathcal{O})$, $\mathrm{H}^1(\mathcal{O};\R^2):=\H^1(\mathcal{O})$ and $\mathrm{L}^p(\mathcal{O};\R^2):=\L^p(\mathcal{O})$ ($2<p<\infty$) norms, respectively. The space $\H$ is endowed with the norm $\|\u\|_{\H}^2:=\int_{\mathcal{O}}|\u(x)|^2\d x,$ the norm on the space $\widetilde{\L}^{p}$ is represented by $\|\u\|_{\wi \L^p}^2:=\int_{\mathcal{O}}|\u(x)|^p\d x,$ for $2<p<\infty$ and the norm on the space $\V$ is given by $\|\u\|^2_{\H^1}=\int_{\mathcal{O}}|\u(x)|^2\d x+\int_{\mathcal{O}}|\nabla\u(x)|^2\d x.$ Using the Poincar\'e inequality \eqref{2.1}, one can easily see that the norm $\|\u\|^2_{\H^1}$ is equivalent to the norm $\|\u\|_{\V}^2:=\int_{\mathcal{O}}|\nabla\u(x)|^2\d x.$ The inner product in the Hilbert space $\H$ is denoted by $( \cdot, \cdot),$ in the Hilbert space $\V$ is denoted by $(\!(\cdot,\cdot)\!)$ and the induced duality, for instance between the spaces $\V$ and $\V'$, and $\widetilde{\L}^p$ and its dual $\widetilde{\L}^{\frac{p}{p-1}}$ is denoted by $\langle\cdot,\cdot\rangle.$ It should be noted that $\V$ is densely and continuously embedded into $\H$ and $\H$ can be identified with its dual $\H'$ and we have the \emph{Gelfand triple}: $\V\subset\H \cong\H'\subset\V'$. In the rest of the paper, we use the notation $\H^2(\mathcal{O}):=\mathrm{H}^2(\mathcal{O};\R^2)$ for the second order Sobolev spaces.
	
	\subsection{Linear operator}\label{proj}
	Let $\mathcal{P}: \L^2(\mathcal{O}) \to\H$ denote the Helmholtz-Hodge orthogonal projection (cf.  \cite{OAL}). Let us define the Stokes operator 
	\begin{equation*}
		\A\u:=-\mathcal{P}\Delta\u,\;\u\in\D(\A).
	\end{equation*}
	The operator $\A$ is a linear continuous operator from $\V$ into $\V'$, satisfying
	\begin{equation*}
		\langle\A\u,\v\rangle=(\!(\u,\v)\!), \ \ \ \u,\v\in\V.
	\end{equation*}
	Since the boundary of $\mathcal{O}$ is uniformly of class $\mathrm{C}^3$, this infer that $\D(\A)=\V\cap\H^2(\mathcal{O})$ and $\|\A\u\|_{\H}$ defines a norm in $\D(\A),$ which is equivalent to the one in $\H^2(\mathcal{O})$ (cf. Lemma 1, \cite{Heywood}). Above argument implies that $\mathcal{P}:\H^2(\mathcal{O})\to\H^2(\mathcal{O})$ is a bounded operator. Note that the operator $\A$ is a non-negative self-adjoint operator in $\H$ and 
	\begin{align}\label{2.7a}
		\langle\A\u,\u\rangle =\|\u\|_{\V}^2,\ \textrm{ for all }\ \u\in\V, \ \text{ so that }\ \|\A\u\|_{\V'}\leq \|\u\|_{\V}.
	\end{align}
	\subsection{Bilinear operator}
	Let us define the \emph{trilinear form} $b(\cdot,\cdot,\cdot):\V\times\V\times\V\to\R$ by $$b(\u,\v,\w)=\int_{\mathcal{O}}(\u(x)\cdot\nabla)\v(x)\cdot\w(x)\d x=\sum_{i,j=1}^2\int_{\mathcal{O}}\u_i(x)\frac{\partial \v_j(x)}{\partial x_i}\w_j(x)\d x.$$ If $\u, \v$ are such that the linear map $b(\u, \v, \cdot) $ is continuous on $\V$, the corresponding element of $\V'$ is denoted by $\B(\u, \v)$. We represent $\B(\u) = \B(\u, \u)=\mathcal{P}[(\u\cdot\nabla)\u]$. Using an integration by parts, one can easily see that 
	\begin{equation}\label{b0}
		\left\{
		\begin{aligned}
			b(\u_1,\u_2,\u_3) &=  -b(\u_1,\u_3,\u_2),\ \text{ for all }\ \u_1,\u_2,\u_3\in \V,\\
			b(\u_1,\u_2,\u_2) &= 0,\ \text{ for all }\ \u_1,\u_2 \in\V.
		\end{aligned}
		\right.\end{equation}
	\begin{remark}
		The following inequalities are used frequently in the sequel.
		\begin{itemize}
			\item [(i)]	Using Ladyzhenskaya's inequality (Lemma 1, Chapter I, \cite{OAL}), we find
			\begin{align}\label{lady}
				\|\u\|_{\L^{4}(\mathcal{O}) } \leq 2^{1/4} \|\u\|^{1/2}_{\L^{2}(\mathcal{O}) } \|\nabla \u\|^{1/2}_{\L^{2}(\mathcal{O}) }, \ \ \ \u\in \H^{1,2}_{0} (\mathcal{O}).
			\end{align}
			\item [(ii)]  The following estimates on the trilinear form $b(\cdot,\cdot,\cdot)$ (see Chapter 2, section 2.3, \cite{Temam1}):
			\begin{itemize}
				\item  [$\bullet$] $\text{for all }\ \u_1\in \V, \u_2\in \D(\A), \u_3\in \H$, we have 
				\begin{align}
					|b(\u_1,\u_2,\u_3)|&\leq C\|\u_1\|^{1/2}_{\H}\|\u_1\|^{1/2}_{\V}\|\u_2\|^{1/2}_{\V}\|\A\u_2\|^{1/2}_{\H}\|\u_3\|_{\H},\label{b2}
				\end{align}
				\item [$\bullet$] $\text{for all }\ \u_1\in \D(\A), \u_2\in \V, \u_3\in \H$, we get 
				\begin{align}
					|b(\u_1,\u_2,\u_3)|&\leq C\|\u_1\|^{1/2}_{\H}\|\A\u_1\|^{1/2}_{\H}\|\u_2\|_{\V}\|\u_3\|_{\H}.\label{b1}
				\end{align}
			\end{itemize}
		\end{itemize}
	\end{remark}
	\begin{remark}\label{R2.2}
		If $\u\in \mathrm{L}^{\infty} (0, T; \V)\cap\mathrm{L}^2(0,T;\D(\A)),$ then $\B(\u)\in \mathrm{L}^2(0, T; \H).$ Indeed, by \eqref{b2}, we have 
		\begin{align}\label{2.8}
			\int_{0}^{T} \|\B(\u(t))\|^2_{\H}\d t \leq C \int_{0}^{T}\|\u(t)\|^{3}_{\V}\|\A\u(t)\|_{\H}\d t \leq C T^{1/2} \|\u\|^{3}_{\mathrm{L}^{\infty}(0,T; \V)} \|\u\|_{\mathrm{L}^{2}(0,T; \D(\A))},
		\end{align}
		which is finite. 
	\end{remark}
	\begin{lemma}[Lemma 5.2, \cite{BL}]\label{convergence_b*}
		Let $\mathcal{O}_1\subset\mathcal{O}$, which is bounded, and $\psi: [0, T]\times \mathcal{O} \to \R^2$ be a $\mathrm{C}^1$-class function such that $\mathrm{supp}\ \psi (t, \cdot) \subset \mathcal{O}_1,$ for $t\in[0, T],$ and 
		$$ \sup_{1\leq i, j\leq 2} \sup_{(t, x)\in [0, T] \times \mathcal{O}_1} |D_i \psi^j (t, x)| = C < \infty.$$ 
		Let $\v_m$ converges to $\v$ in $\mathrm{L}^2(0, T; \V)$ weakly and in $\mathrm{L}^2(0, T; \L^2(\mathcal{O}_1))$ strongly. Then, we have 
		$$\int_{0}^{T} b(\v_m(t), \v_m(t), \psi(t)) \d t \to \int_{0}^{T} b(\v(t), \v(t), \psi(t)) \d t\ \text{ as }\ m\to\infty.$$
	\end{lemma}
	\begin{lemma}\label{Converge_b}
		Let $\{\v_n\}_{n\in \mathbb{N}}$ be a bounded sequence in $\mathrm{L}^{\infty}(0, T; \V),$ and $ \v\in \mathrm{L}^{\infty}(0, T; \V)$ be such that 
		\begin{align}
			\v_n\xrightharpoonup{w}\v &\ \text{ in }\  \mathrm{L}^2(0, T;\D(\A)),\label{210}\\
			\v_n\to \v  &\ \text{ in }\  \mathrm{L}^2(0, T;\V).\label{2.11}
		\end{align}  Then for any $\w \in \mathrm{L}^{\infty}(0,T; \V)\cap\mathrm{L}^2(0,T;\D(\A))$, we have 
		\begin{align}
			\int_{0}^{T}  b\big(\v_n(t), \v_n(t), \A\v_n(t) \big) \d t &\to  \int_{0}^{T}  b\big(\v(t), \v(t), \A\v(t) \big) \d t,\label{C_b1}\\
			\int_{0}^{T}  b\big(\v_n(t), \w(t), \A\v_n(t) \big) \d t &\to  \int_{0}^{T}  b\big(\v(t), \w(t), \A\v(t) \big) \d t,\label{C_b2}\\
			\int_{0}^{T} b\big(\w(t), \v_n(t), \A\v_n(t) \big) \d t &\to  \int_{0}^{T}  b\big(\w(t), \v(t), \A\v(t) \big) \d t,\label{C_b3}
		\end{align}
		as $n\to \infty$, for all $T>0.$
	\end{lemma}
	\begin{proof}
		We shall prove this Lemma in three steps.
		\vskip 0.2cm 
		\noindent
		\textbf{Step I.} \emph{Proof of \eqref{C_b1}:} Let us first consider
		\begin{align}\label{C_b4}
			&\left|  \int_{0}^{T}  b\big(\v_n(t), \v_n(t), \A\v_n(t) \big) \d t - \int_{0}^{T}  b\big(\v(t), \v(t), \A\v(t) \big) \d t\right|\nonumber\\
			&\leq \int_{0}^{T}  \big|b\big(\v_n(t)-\v(t), \v_n(t), \A\v_n(t) \big)\big| \d t+  \int_{0}^{T}  \big|b\big(\v(t), \v_n(t)-\v(t), \A\v_n(t) \big)\big| \d t\nonumber\\&\quad+\int_{0}^{T} \big|   b\big(\v(t), \v(t), \A(\v_n(t)-\v(t)) \big)\big| \d t =: K_1 +K_2+K_3.
		\end{align}
		Using \eqref{b2}, H\"older's and Young's inequalities, we estimate $K_1$ and $K_2$ as 
		\begin{align}
			K_1 &\leq   C \int_{0}^{T} \|\v_n(t)-\v(t)\|^{1/2}_{\H}\|\v_n(t)-\v(t)\|^{1/2}_{\V}\|\v_n(t)\|^{1/2}_{\V} \|\A\v_n(t)\|^{3/2}_{\H} \d t \nonumber\\&\leq C \left[\|\v_n\|^{1/2}_{\mathrm{L}^{\infty}(0, T; \H)}+\|\v\|^{1/2}_{\mathrm{L}^{\infty}(0, T; \H)}\right] \|\v_n\|^{1/2}_{\mathrm{L}^{\infty}(0, T; \V)} \|\v_n-\v\|^{1/2}_{\mathrm{L}^{2}(0, T; \V)}  \|\v_n\|^{3/2}_{\mathrm{L}^{2}(0, T; \D(\A))} \nonumber\\&\to 0 \ \text{ as }\  n\to \infty,\label{C_b5}\\
			K_2 &\leq   C \int_{0}^{T} \|\v(t)\|^{1/2}_{\H}\|\v(t)\|^{1/2}_{\V}\|\v_n(t)-\v(t)\|^{1/2}_{\V} \|\A(\v_n(t)-\v(t))\|^{1/2}_{\H}\|\A\v_n(t)\|_{\H} \d t \nonumber\\&\leq C \|\v\|^{1/2}_{\mathrm{L}^{\infty}(0, T; \H)}\|\v\|^{1/2}_{\mathrm{L}^{\infty}(0, T; \V)}  \|\v_n-\v\|^{1/2}_{\mathrm{L}^{2}(0, T; \V)}  \|\v_n-\v\|^{1/2}_{\mathrm{L}^{2}(0, T; \D(\A))} \|\v_n\|_{\mathrm{L}^{2}(0, T; \D(\A))}\nonumber\\&\leq C \|\v\|_{\mathrm{L}^{\infty}(0, T; \V)}  \|\v_n-\v\|^{1/2}_{\mathrm{L}^{2}(0, T; \V)} \left(\|\v_n\|^{1/2}_{\mathrm{L}^{2}(0, T; \D(\A))}+ \|\v\|^{1/2}_{\mathrm{L}^{2}(0, T; \D(\A))}\right) \|\v_n\|_{\mathrm{L}^{2}(0, T; \D(\A))} \nonumber\\&\to 0 \text{ as } n\to \infty.\label{C_b6}
		\end{align}
		By Remark \ref{R2.2}, we know that $\B(\v)\in \mathrm{L}^2(0,T;\H).$ Hence, using the weak convergence in the space $\mathrm{L}^2(0,T;\D(\A))$ (see \eqref{210}), we get 
		\begin{align}\label{C_b7}
			\lim_{n \to \infty} K_3 \leq\lim_{n \to \infty}  \int_{0}^{T}  \big|\big(\B(\v(t)), \A(\v_n(t)-\v(t)) \big)\big| \d t =0.
		\end{align}
		Combining \eqref{C_b5}-\eqref{C_b7} and using in \eqref{C_b4}, we obtain \eqref{C_b1}.
		\vskip 0.2cm 
		\noindent
		\textbf{Step II.} \emph{Proof of \eqref{C_b2}:}  Next, we consider 
		\begin{align}\label{C_b8}
			&\left|  \int_{0}^{T}  b\big(\v_n(t), \w(t), \A\v_n(t) \big) \d t - \int_{0}^{T}  b\big(\v(t), \w(t), \A\v(t) \big) \d t\right|\nonumber\\
			&	\leq    \int_{0}^{T} \big| b\big(\v_n(t)-\v(t), \w(t), \A\v_n(t) \big)\big| \d t+ \int_{0}^{T} \big|  b\big(\v(t), \w(t), \A(\v_n(t)-\v(t)) \big) \big|\d t \nonumber\\&=: K_4 +K_5.
		\end{align}
		Using \eqref{b2}, H\"older's  and Young's inequalities, we estimate $K_4$ as 
		\begin{align}
			K_4 & \leq C \int_{0}^{T} \|\v_n(t)-\v(t)\|^{1/2}_{\H}\|\v_n(t)-\v(t)\|^{1/2}_{\V}\|\w(t)\|^{1/2}_{\V} \|\A\w(t)\|^{1/2}_{\H}\|\A\v_n(t)\|_{\H} \d t \nonumber\\&\leq C \bigg[\|\v_n\|^{1/2}_{\mathrm{L}^{\infty}(0, T; \H)}+\|\v\|^{1/2}_{\mathrm{L}^{\infty}(0, T; \H)}\bigg] \|\w\|^{1/2}_{\mathrm{L}^{\infty}(0, T; \V)} \|\v_n-\v\|^{1/2}_{\mathrm{L}^{2}(0, T; \V)} \|\w\|^{1/2}_{\mathrm{L}^2(0,T; \D(\A))}\nonumber\\&\quad\quad\times \|\v_n\|_{\mathrm{L}^{2}(0, T; \D(\A))}\nonumber\\& \to 0 \ \text{ as }\  n\to \infty.\label{C_b9}
		\end{align}
		Making use of \eqref{b2}, one can show that $\B(\v,\w) \in \mathrm{L}^2(0, T;\H),$ since 
		\begin{align*}
			\int_{0}^{T} \|\B(\v(t),\w(t))\|^2_{\H}\d t &\leq C\int_{0}^{T} \|\v(t)\|_{\H} \|\v(t)\|_{\V}\|\w(t)\|_{\V}\|\A\w(t)\|_{\H}\d t\nonumber\\ &\leq CT^{1/2}\|\v\|_{\mathrm{L}^{\infty}(0, T; \H)}\|\v\|_{\mathrm{L}^{\infty}(0, T; \V)}\|\w\|_{\mathrm{L}^{\infty}(0, T; \V)}  \|\w\|_{\mathrm{L}^{2}(0, T; \D(\A))}<\infty.
		\end{align*}Hence, using the weak convergence given in \eqref{210}, we get 
		\begin{align}\label{C_b10}
			\lim_{n \to \infty} K_5 \leq\lim_{n \to \infty}  \int_{0}^{T}  \big|\big(\B(\v(t),\w(t)), \A(\v_n(t)-\v(t)) \big)\big| \d t =0.
		\end{align}
		Combining \eqref{C_b9}-\eqref{C_b10} and using it in \eqref{C_b8}, we obtain \eqref{C_b2}.
		\vskip 0.2cm 
		\noindent
		\textbf{Step III.} \emph{Proof of \eqref{C_b3}:}  Finally, we consider 
		\begin{align}\label{C_b11}
			&\left|  \int_{0}^{T}  b\big(\w(t), \v_n(t), \A\v_n(t) \big) \d t - \int_{0}^{T}  b\big(\w(t), \v(t), \A\v(t) \big) \d t\right|\nonumber\\
			&\leq  \int_{0}^{T}  \big| b\big(\w(t), \v_n(t)-\v(t), \A\v_n(t) \big) \big|\d t+  \int_{0}^{T} \big| b\big(\w(t), \v(t), \A(\v_n(t)-\v(t)) \big)\big|  \d t\nonumber\\&=: K_6 +K_7.
		\end{align}
		Once again using \eqref{b2}, H\"older's  and Young's inequalities, we obtain
		\begin{align}
			K_6 &\leq C \int_{0}^{T} \|\w(t)\|^{1/2}_{\H}\|\w(t)\|^{1/2}_{\V}\|\v_n(t)-\v(t)\|^{1/2}_{\V} \|\A(\v_n(t)-\v(t))\|^{1/2}_{\H}\|\A\v_n(t)\|_{\H} \d t \nonumber\\&\leq C \|\w\|^{1/2}_{\mathrm{L}^{\infty}(0, T; \H)}\|\w\|^{1/2}_{\mathrm{L}^{\infty}(0, T; \V)}  \|\v_n-\v\|^{1/2}_{\mathrm{L}^{2}(0, T; \V)}  \|\v_n-\v\|^{1/2}_{\mathrm{L}^{2}(0, T; \D(\A))} \|\v_n\|_{\mathrm{L}^{2}(0, T; \D(\A))}\nonumber\\&\leq  C \|\w\|_{\mathrm{L}^{\infty}(0, T; \V)}  \|\v_n-\v\|^{1/2}_{\mathrm{L}^{2}(0, T; \V)}  \left(\|\v_n\|^{1/2}_{\mathrm{L}^{2}(0, T; \D(\A))}+ \|\v\|^{1/2}_{\mathrm{L}^{2}(0, T; \D(\A))}\right) \|\v_n\|_{\mathrm{L}^{2}(0, T; \D(\A))} \nonumber\\&\to 0 \text{ as } n\to \infty.\label{C_b12}
		\end{align}
		One can prove that  $\B(\w,\v) \in \mathrm{L}^2(0, T;\H),$ by using  \eqref{b2},
		and hence, using the weak convergence in the space $\mathrm{L}^2(0,T;\D(\A))$ (see \eqref{210}), we get 
		\begin{align}\label{C_b13}
			\lim_{n \to \infty} K_7 \leq\lim_{n \to \infty}  \int_{0}^{T}  \big|\big(\B(\w(t),\v(t)), \A(\v_n(t)-\v(t)) \big)\big| \d t =0.
		\end{align}
		Combining \eqref{C_b12}-\eqref{C_b13} and using in \eqref{C_b11}, we finally obtain \eqref{C_b3}.
	\end{proof}
	\subsection{Nonlinear operator}
	Let us now consider the operator $\mathcal{C}(\u):=\mathcal{P}(|\u|^{r-1}\u)$. It is immediate that $\langle\mathcal{C}(\u),\u\rangle =\|\u\|_{\widetilde{\L}^{r+1}}^{r+1}$ and the map $\mathcal{C}(\cdot):\V\cap\widetilde{\L}^{r+1}\to\V'+ \widetilde{\L}^{\frac{r+1}{r}}$. For all $\u\in\V\cap\wi\L^{r+1}$, the map is Gateaux differentiable with Gateaux derivative 
	\begin{align}\label{29}
		\mathcal{C}'(\u)\v&=\left\{\begin{array}{cl}\mathcal{P}(\v),&\text{ for }r=1,\\ \left\{\begin{array}{cc}\mathcal{P}(|\u|^{r-1}\v)+(r-1)\mathcal{P}\left(\frac{\u}{|\u|^{3-r}}(\u\cdot\v)\right),&\text{ if }\u\neq \mathbf{0},\\\mathbf{0},&\text{ if }\u=\mathbf{0},\end{array}\right.&\text{ for } 1<r<3,\\ \mathcal{P}(|\u|^{r-1}\v)+(r-1)\mathcal{P}(\u|\u|^{r-3}(\u\cdot\v)), &\text{ for }r\geq 3,\end{array}\right.
	\end{align}
	for all $\v\in\V\cap\widetilde{\L}^{r+1}$.  For any $r\in [1, \infty)$ and $\u_1, \u_2 \in \V\cap\wi\L^{r+1}$, we have (see subsection 2.4, \cite{MTM})
	\begin{align}\label{MO_c}
		\langle\mathcal{C}(\u_1)-\mathcal{C}(\u_2),\u_1-\u_2\rangle \geq 0.
	\end{align}
	We need the following convergence result in the sequel. 
	\begin{lemma}[\cite{KM}]\label{convergence_c2_1}
		Let $\mathcal{O}_1\subset\mathcal{O}$, which is bounded, and $\psi: [0, T]\times \mathcal{O} \to \R^2$ be a continuous  function such that $\mathrm{supp}\ \psi (t, \cdot) \subset \mathcal{O}_1,$ for $t\in[0, T],$ and 
		$$ \sup_{(t, x)\in [0, T] \times \mathcal{O}_1} |\psi (t, x)| = C < \infty.$$   Assume that $\{\v_m\}_{m\in \mathbb{N}}$ is a bounded sequence in the space $\mathrm{L}^{\infty}(0, T; \H)\cap\mathrm{L}^{r+1} (0, T; \widetilde{\L}^{r+1}), \v \in \mathrm{L}^{\infty}(0, T; \H)\cap\mathrm{L}^{r+1} (0, T; \widetilde{\L}^{r+1}),$ $\v_m$ converges to $\v$ in $\mathrm{L}^2(0, T; \V)$ weakly and $\v_m$ converges to $\v$ in $\mathrm{L}^2(0, T; \L^2(\mathcal{O}_1))$ strongly. Then for any $r\in[1,3]$ with $\w\in\mathrm{L}^{4} (0, T; \widetilde{\L}^{4})\cap \mathrm{L}^2 (0, T; \H)$ and for any $r>3$ with $\w\in\mathrm{L}^{r+1} (0, T; \widetilde{\L}^{r+1})\cap \mathrm{L}^2 (0, T; \H)$, 
		\begin{align}\label{217}
			\int_{0}^{T} \big\langle\mathcal{C}(\v_m(t)+\w(t)) ,\psi(t)  \big\rangle \d t \to \int_{0}^{T} \big\langle\mathcal{C}(\v(t)+\w(t)) ,\psi(t)  \big\rangle \d t\ \text{ as }\ m\to\infty.
		\end{align}
	\end{lemma}
	\begin{lemma}\label{Converge_c}
		Let $\{\v_n\}_{n\in \mathbb{N}}$ be a bounded sequence in $\mathrm{L}^{\infty}(0, T; \V)$ and $ \v\in \mathrm{L}^{\infty}(0, T; \V)$ be such that \eqref{210}-\eqref{2.11} hold true. 
		Then for any $\w \in \mathrm{L}^{\infty}(0,T; \V)\cap\mathrm{L}^2(0,T;\D(\A))$ and for any $r\geq1 $, we have 
		\begin{align}\label{C_c1}
			\int_{0}^{T} \big(\mathcal{C}(\v_n(t) +\w(t)),\A\v_n(t)\big)  \d t \to  \int_{0}^{T}  \big(\mathcal{C}(\v(t) +\w(t)),\A\v(t)\big) \d t,
		\end{align}
		as $n\to \infty$, for all $T>0.$
	\end{lemma}
	\begin{proof}
		Let us consider 
		\begin{align}\label{C_c2}
			&\left|\int_{0}^{T} \big(\mathcal{C}(\v_n(t) +\w(t)),\A\v_n(t)\big)  \d t - \int_{0}^{T}  \big(\mathcal{C}(\v(t) +\w(t)),\A\v(t)\big) \d t\right|\nonumber\\ &\leq  \int_{0}^{T} \big|\big(\mathcal{C}(\v_n(t) +\w(t))-\mathcal{C}(\v(t)+\w(t)),\A\v_n(t)\big)\big|  \d t  \nonumber\\&\quad+\int_{0}^{T}\big| \big(\mathcal{C}(\v(t) +\w(t)),\A(\v_n(t)-\v(t))\big) \big|  \d t=: \widetilde{K}_1  +\widetilde{K}_2.
		\end{align}
		For $r=1$, it is immediate that $\widetilde{K}_1\to 0$ as $n\to \infty.$ For $r>1$, we estimate $\widetilde{K}_1$ using Taylor's formula (Theorem 7.9.1, \cite{PGC}), H\"older's inequality and Sobolev's embedding as
		\begin{align}
			\widetilde{K}_1&\leq \int_{0}^{T} \bigg(\int_{0}^{1} \|\mathcal{C}'(\theta(\v_n(t)+\w(t))+(1-\theta)(\v(t)+\w(t))) (\v_n(t)-\v(t))\|_{\H} \d \theta\bigg)\|\A\v_n(t)\|_{\H} \d t\nonumber\\&\leq r \int_{0}^{T} \bigg[\|\v_n(t)+\w(t)\|^{r-1}_{\wi \L^{2r}} + \|\v(t)+\w(t)\|^{r-1}_{\wi \L^{2r}}\bigg]\|\v_n(t)-\v(t)\|_{\wi \L^{2r}} \|\A\v_n(t)\|_{\H} \d t\nonumber\\&\leq  C \int_{0}^{T} \bigg[\|\v_n(t)+\w(t)\|^{r-1}_{\V} + \|\v(t)+\w(t)\|^{r-1}_{\V}\bigg]\|\v_n(t)-\v(t)\|_{\V} \|\A\v_n(t)\|_{\H} \d t\nonumber\\&\leq C \bigg[\|\v_n +\w\|^{r-1}_{\mathrm{L}^{\infty}(0,T;\V)} + \|\v +\w\|^{r-1}_{\mathrm{L}^{\infty}(0,T;\V)}\bigg] \|\v_n -\v\|_{\mathrm{L}^{2}(0,T;\V)} \|\v_n\|_{\mathrm{L}^{2}(0,T;\D(\A))}\nonumber\\&\to 0 \text{ as } n\to \infty.\label{C_c3}
		\end{align}
		Now, one can show that  $\mathcal{C}(\v+\w)\in \mathrm{L}^{2}(0,T;\H)$, since 
		\begin{align*}
			\int_{0}^{T}\|\mathcal{C}(\v(t)+\w(t))\|^2_{\H} \d t &= \int_{0}^{T} \|\v(t)+\w(t)\|^{2r}_{\wi \L^{2r}}\d t\leq C\int_{0}^{T} \|\v(t)+\w(t)\|^{2r}_{\V}\d t\nonumber\\&  \leq CT \|\v+\w\|^{2r}_{\mathrm{L}^{\infty}(0,T; \V)} <\infty. 
		\end{align*}
		Hence, using the weak convergence \eqref{210} in the space $\mathrm{L}^2(0,T;\D(\A))$, we get 
		\begin{align}\label{C_c4}
			\lim_{n \to \infty} \widetilde{K}_2 \leq\lim_{n \to \infty}  \int_{0}^{T}  \big|\big(\mathcal{C}(\v(t)+\w(t)), \A(\v_n(t)-\v(t)) \big)\big| \d t =0.
		\end{align}
		Combining \eqref{C_c3}-\eqref{C_c4} and using it in \eqref{C_c2}, we finally  deduce  \eqref{C_c1}.
	\end{proof}
	
	\subsection{Stochastic convective Brinkman-Forchheimer equations}
	In this subsection, we provide an abstract formulation of the system \eqref{SCBF}, assumptions on the noise and  discuss  global solvability results. 
	On taking projection $\mathcal{P}$ onto the first equation in \eqref{SCBF}, we obtain 
	\begin{equation}\label{S-CBF}
		\left\{
		\begin{aligned}
			\d\u(t)+\{\mu \A\u(t)+\B(\u(t))+\alpha\u(t)+\beta \mathcal{C}(\u(t))\}\d t&=\f \d t + \d\mathrm{W}(t), \ \ \ t\geq 0, \\ 
			\u(0)&=\x,
		\end{aligned}
		\right.
	\end{equation}
	for $r\in[1,3]$, where $\x\in \V,\ \f\in \H$ and $\text W(t), \ t\in \R,$ is a two-sided cylindrical Wiener process in $\H$ with its Reproducing Kernel Hilbert space (RKHS) $\mathrm{K}$ satisfying the following assumption.
	\begin{assumption}\label{assump}
		$ \mathrm{K} \subset \V \cap \H^{2} (\mathcal{O})$ is a Hilbert space such that for some $\delta\in (0, 1/2),$
		\begin{align}\label{A1}
			\A^{-\delta} : \mathrm{K} \to \V \cap \H^{2} (\mathcal{O}) \   \text{ is }\ \gamma \text{-radonifying.}
		\end{align}
	\end{assumption}
	\begin{remark}
		Since $\D(\A)=\V \cap \H^{2} (\mathcal{O})$, Assumption \ref{assump} can be reformulated in the following way also (see \cite{BL1}). $\mathrm{K}$ is a Hilbert space such that $\mathrm{K}\subset\D(\A)$ and, for some $\delta\in(0,1/2)$, the map \begin{align}\label{34}
			\A^{-\delta-1} : \mathrm{K} \to \H \   \text{ is }\ \gamma \text{-radonifying.}
		\end{align}
		The condition  \eqref{34} also says that the mapping $\A^{-\delta-1} : \mathrm{K} \to \H$  is Hilbert-Schmidt.	If $\mathcal{O}$ is a bounded domain, then $\A^{-s}:\H\to\H$ is Hilbert-Schmidt if and only if $\sum_{j=1}^{\infty} \lambda_j^{-2s}<\infty,$ where  $\A e_j=\lambda_j e_j, j\in \N$ and $\{e_j\}_{j\in\N}$ is an orthogonal basis of $\H$. In bounded domains, it is well known that $\lambda_j\sim j,$ for large $j$ (growth of eigenvalues) and hence $\A^{-s}$ is Hilbert-Schmidt if and only if $s>\frac{1}{2}.$ In other words, with $\mathrm{K}=\D(\A^{s+1}),$ the embedding $\mathrm{K}\hookrightarrow\V\cap\H^2(\mathcal{O})$ is $\gamma$-radonifying if and only if $s>\frac{1}{2}.$ Thus, Assumption \ref{assump} is satisfied for any $\delta>0.$ In fact, the condition \eqref{A1} holds if and only if the operator $\A^{-(s+1+\delta)}:\H\to \V\cap\H^2(\mathcal{O})$ is $\gamma$-radonifying. The requirement of  $\delta<\frac{1}{2}$ in Assumption \ref{assump} is necessary because we need (see subsection \ref{O_up}) the corresponding Ornstein-Uhlenbeck process has to take values in $\V\cap\H^2(\mathcal{O})$.
	\end{remark}

	\section{RDS Generated by 2D SCBF Equations}\label{sec3}\setcounter{equation}{0} In this section, we discuss the random dynamical system generated by the 2D SCBF equations \eqref{S-CBF}. 
	\subsection{Ornstein-Uhlenbeck process}\label{O_up}

	Let us first characterize Ornstein-Uhlenbeck processes under Assumption \ref{assump} (for more details see section 3, \cite{KM1}). Let us define $\mathrm{X} := \V \cap \H^{2} (\mathcal{O})$ and let $\mathrm{E}$ denote the completion of $\A^{-\delta}\mathrm{X}$ with respect to the image norm $\|x\|_{\mathrm{E}}  = \|\A^{-\delta} x\|_{\mathrm{X}} , \ \text{for } \ x\in \mathrm{X}, \text{where } \|\cdot\|_{\mathrm{X}} = \|\cdot\|_{\V} + \|\cdot\|_{\H^{2} }.$ For $\xi \in(0, 1/2),$ we define
	\begin{align*}
		\C^{\xi}_{1/2} (\mathbb{R}, \mathrm{E}) &= \left\{ \omega \in \C(\mathbb{R}, \mathrm{E}) : \omega(0)=0,\  \sup_{t\neq s \in \mathbb{R}} \frac{\|\omega(t) - \omega(s)\|_{\mathrm{E}}}{|t-s|^{\xi}(1+|t|+|s|)^{1/2}} < \infty \right\},\\
		\Omega(\xi, \mathrm{E})&=\text{the closure of } \{ \omega \in \C^\infty_0 (\mathbb{R}, \mathrm{E}) : \omega(0) = 0 \} \ \text{ in } \ \C^{\xi}_{1/2} (\mathbb{R},\mathrm{E}).
	\end{align*}
	The space $\Omega(\xi, \mathrm{E})$ is a separable Banach space. Let us denote $\mathscr{F},$ for the Borel $\sigma$-algebra on $\Omega(\xi, \mathrm{E}).$ For $\xi\in (0, 1/2)$, there exists a Borel probability measure $\mathbb{P}$ on $\Omega(\xi, \mathrm{E})$ (see \cite{Brze}). For $t\in \mathbb{R},$ let $\mathscr{F}_t := \sigma \{ w_s : s \leq t \},$ where $w_t$ is the canonical process defined by the elements of  $\Omega(\xi, \mathrm{E}).$ Then there exists a family $\{\text{W}(t)\}_{t\in \mathbb{R}}$, which is $\mathrm{K}$-cylindrical Wiener process on a filtered probability space $(\Omega(\xi, \mathrm{E}), \mathscr{F}, \{\mathscr{F}_t\}_{t \in \mathbb{R}} , \mathbb{P})$. 
	
	On the space $\Omega(\xi, \mathrm{E}),$ we consider a flow $\theta = (\theta_t)_{t\in \mathbb{R}}$ defined by
	$$ \theta_t \omega(\cdot) = \omega(\cdot + t) - \omega(t), \  \omega\in \Omega(\xi, \mathrm{E}), \ t\in \mathbb{R}.$$ 
	\begin{lemma}[Proposition 6.10, \cite{BL}]\label{SOUP}
		The process $\z_{\eta}(t), \ t\in \mathbb{R},$ is a stationary Ornstein-Uhlenbeck process on $(\Omega(\xi, \mathrm{E}), \mathscr{F}, \mathbb{P})$. It is a solution of the equation 
		\begin{align}\label{OUPe}
			\d\z_{\eta}(t) + (\mu \A + \eta I)\z_{\eta}(t) \d t = \d\mathrm{W}(t), \ \ t\in \mathbb{R},
		\end{align}
		that is, for all $t\in \mathbb{R},$ 
		\begin{align}\label{oup}
			\z_\eta (t) = \int_{-\infty}^{t} e^{-(t-s)(\mu \A + \eta I)} \d\mathrm{W}(s),
		\end{align}
		$\mathbb{P}$-a.s.,	where the integral is the It\^o integral on the M-type 2 Banach space $\mathrm{X}$ in the sense given in \cite{Brze1}. 
		In particular, for some constant  $C$ depending on $\mathrm{X}$, we have 
		\begin{align}\label{E-OUP}
			\mathbb{E}\left[\|\z_{\eta} (t)\|^2_{\mathrm{X}} \right]\leq C \int_{0}^{\infty}  e^{-2\eta s} \|e^{-\mu s \A}\|^2_{\gamma(\mathrm{K},\mathrm{X})} \d s.
		\end{align} 
		Moreover, $\mathbb{E}\left[\|\z_{\eta} (t)\|^2_{\mathrm{X}}\right]$ tends to $0$ as $\eta \to \infty.$
		
	\end{lemma}
	\begin{remark}\label{stationary}
		By Proposition 4.1 \cite{KM}, we deduce the following result for the Ornstein-Uhlenbeck process given in Lemma \ref{SOUP}:
		\begin{align}\label{O-U_conti}
			\z_\eta\in\mathrm{L}^{q} (a, b; \mathrm{X})
		\end{align}
		where $q\in [1, \infty].$
	\end{remark}
	Since by Lemma \ref{SOUP}, the process $\z_{\eta}(t), \ t\in \R $ is an $\mathrm{X}$-valued stationary and ergodic. Hence, by the Strong Law of Large Numbers (see \cite{DZ}),
	\begin{align}\label{SLLN}
		\lim_{t \to \infty} \frac{1}{t} \int_{-t}^{0} \|\z_{\eta}(s)\|^{2}_{\mathrm{X}} \d s = \mathbb{E}\left[ \|\z_{\eta}(0)\|^{2}_{\mathrm{X}}\right], \ \  \mathbb{P}\text{-a.s., on }\ \C^{\xi}_{1/2}(\R, \mathrm{X}).
	\end{align}
	Therefore by Lemma \ref{SOUP}, we find $\eta_0$ such that for all $\eta \geq \eta_0,$ we have 
	\begin{align}\label{Bddns}
		\mathbb{E}\left[\|\z_{\eta} (0)\|^{2}_{\mathrm{X}}\right] \leq \frac{\mu\alpha}{8},
	\end{align}  
	where $\alpha$ is the Darcy constant.
	
	Let us denote 
	$$\Omega_{\eta}(\xi, \mathrm{E})=\{\omega\in \Omega(\xi, \mathrm{E}):\text{the equality \eqref{SLLN} holds true}\},$$ which is of full measure.
	Therefore, we fix $\xi \in (\delta, 1/2)$ and set $$\Omega := \hat{\Omega}(\xi, \mathrm{E}) = \bigcap^{\infty}_{n=0} \Omega_{n}(\xi, \mathrm{E}).$$ Consequently, $\Omega$ also is of full measure.
	\begin{proposition}\label{m-DS}
		The quadruple $(\Omega, \hat{\mathscr{F}}, \hat{\mathbb{P}}, \hat{\theta})$ is a metric DS, where $\hat{\mathscr{F}}$, $\hat{\mathbb{P}}$, $\hat{\theta}$ are respectively the natural restrictions of $\mathscr{F}$, $\mathbb{P}$, and $\theta$ to $\Omega.$ For each $\omega\in \Omega,$ the limit in \eqref{SLLN} exists.
	\end{proposition}
	Now, we provide  the important consequence of the estimates \eqref{SLLN} and \eqref{Bddns}.
	\begin{corollary}\label{Bddns1}
		For each $\omega\in \Omega,$ there exists $t_0=t_0 (\omega) \geq 0 $, such that 
		\begin{align}\label{Bddns2}
			\frac{4}{\mu} \int_{-t}^{0} \|\z_{\eta}(s)\|^{2}_{\V} \d s \leq	\frac{4}{\mu} \int_{-t}^{0} \|\z_{\eta}(s)\|^{2}_{\mathrm{X}} \d s \leq \frac{\alpha t}{2}, \ \  t\geq t_0.
		\end{align}
	\end{corollary}
	\subsection{Random dynamical system}
	Let us remember that Assumption \ref{assump} is satisfied and that $\delta$ has the property stated there. We  take  fixed $\mu,\alpha,\beta > 0$ and some parameter $\eta\geq 0$. We also fix $\xi \in (\delta, 1/2)$. Let us denote $\v^{\eta}(t)=\u(t) - \z_{\eta}(\omega)(t)$. For convenience, we write $\v^{\eta}(t)=\v(t)$ and $\z_{\eta}(\omega)(t)=\z(t)$. Then $\v(t)$ satisfies the following abstract system:
	\begin{equation}\label{cscbf}
		\left\{
		\begin{aligned}
			\frac{\d\v(t)}{\d t} &= -\mu \A\v (t)- \B(\v(t) + \z(t))-\alpha\v(t) - \beta \mathcal{C}(\v (t)+ \z(t)) + (\eta-\alpha) \z(t) + \f, \\
			\v(0)&= \x - \z_{\eta}(0).
		\end{aligned}
		\right.
	\end{equation}
	Since $\z_{\eta}(\omega)$ is $\mathrm{X}$-valued process, $\z_{\eta}(\omega)(0)\in \V$. Note that for $\mathbb{P}$-a.a $\omega\in\Omega$, the system \eqref{cscbf} is a deterministic one. In what follows, the definition of strong solution (in the deterministic sense) for the system \eqref{cscbf} is provided.
	\begin{definition}\label{def_v}
		For $r\in[1,3]$, assume that $\x \in \H$, $\f\in \V'$ and $$\z\in\mathrm{L}^{\infty}_{\emph{loc}} ([0, \infty); \V)\cap \mathrm{L}^2_{\emph{loc}}([0, \infty);\D(\A)).$$ A function $\v(\cdot)$ is called a \emph{weak solution} of the system \eqref{cscbf} on time interval $[0, \infty)$, if $$\v\in  \mathrm{C}([0,\infty); \H) \cap \mathrm{L}^{2}_{\emph{loc}}([0,\infty); \V)\cap\mathrm{L}^{\r+1}_{\emph{loc}}([0,\infty); \widetilde{\L}^{r+1}),$$  $$\frac{\d\v}{\d t}\in\mathrm{L}^{2}_{\emph{loc}}([0,\infty); \V') \text{ for } r\in[1,3)\  \text{ and }\ \frac{\d\v}{\d t}\in\mathrm{L}^{2}_{\emph{loc}}([0,\infty); \V')+\mathrm{L}^{\frac{4}{3}}_{\emph{loc}}([0,\infty); \wi\L^{\frac{4}{3}}) \text{ for } r=3, $$ and it satisfies 
		\begin{itemize}
			\item [(i)] for any $\boldsymbol{\phi}\in \V,$ 
			\begin{align}\label{W-CSCBF}
				\left\langle\frac{\d\v(t)}{\d t}, \boldsymbol{\phi}\right\rangle=  - \left\langle \mu \A\v(t)-\alpha\v(t)+\B(\v(t)+\z(t))+\beta \mathcal{C}(\v(s)+\z(t)) - (\eta-\alpha)\z(t)- \f , \boldsymbol{\phi} \right\rangle,
			\end{align}
			for a.e. $t\in[0,\infty)$.
			\item [(ii)] $\v(t)$ satisfies the initial data:
			$$\v(0)=\x-\z(0).$$
		\end{itemize}
	\end{definition}
	\begin{theorem}\label{solution_v}
		For $r\in[1,3]$, let $\eta\geq0$, $\x \in \H$, $\f\in \V'$, $$\z\in\mathrm{L}^{\infty}_{\emph{loc}} ([0, \infty); \V)\cap \mathrm{L}^2_{\emph{loc}}([0, \infty);\D(\A)).$$ Then there exists a unique solution $\v(\cdot)$ to the system \eqref{cscbf} in the sense of Definition \ref{def_v}. 
		
		In addition, for $\x\in \V$ and $\f\in\H$, there exists a unique strong solution $\v(\cdot)$ to the system \eqref{S-CBF} satisfying the following regularity:
		$$\v\in  \mathrm{C}([0,\infty); \V) \cap \mathrm{L}^{2}_{\emph{loc}}([0,\infty); \D(\A)) \ \text{ and } \ \frac{\d\v}{\d t}\in\mathrm{L}^{2}_{\emph{loc}}([0,\infty); \H).$$ 
	\end{theorem}
	\begin{proof}
		Let us fix $T>0$. In order to complete the proof on interval $[0,\infty)$, it is enough to prove on interval $[0,T]$. 
		\vskip 2mm
		\noindent
		\textbf{Step I.} \emph{Existence of weak solutions.} Since $\V$ is a separable Hilbert space, $\mathcal{V}$ is dense in $\V$ and $\V$ is dense in $\H$, there exists a set $\{w_1,w_2,\cdots,w_n,\cdots\}\subset\mathcal{V}$ which is complete orthonormal basis of $\H$. Let $\H_n$ be the $n$-dimensional subspaces of $\H$ defined by $\H_n=\text{span}\{w_1,w_2,\cdots,w_n\}$ with the norm inherited from $\H$. Let us denote the orthogonal projection from $\H$ onto $\H_n$ by $\P_n$, that is, $\h^n=\P_n\h=\sum_{j=1}^{n}(\h,w_j)w_j$, for $\h\in\H$. Let us consider the following approximate equation for the system \eqref{cscbf} on finite dimensional space $\H_n$:
		\begin{equation}\label{cscbf_n}
			\left\{
			\begin{aligned}
				\frac{\d\v^n}{\d t} &= \P_n\bigg[-\mu \A\v^n -\alpha\v^n- \B(\v^n+\z)- \beta \mathcal{C}(\v^n + \z) + (\eta-\alpha)\z + \f\bigg], \\
				\v^n(0)&= \P_n[\x-\z(0)]:=\v_0^n.
			\end{aligned}
			\right.
		\end{equation}
		We define $\A_n\v^n=\P_n\A\v^n$, $\B_n(\v^n+\z)=\mathrm{P}_n\B(\v^n+\z)$ and $\mathcal{C}_n(\v^n+\z)=\mathrm{P}_n\mathcal{C}(\v^n+\z)$ and consider the following system of ODEs:
		\begin{equation}\label{finite-dimS}
			\left\{
			\begin{aligned}
				\frac{\d\v^n(t)}{\d t}&=-\mu \A_n\v^n(t)-\alpha\v^n(t)-\B_n(\v^n(t)+\z(t))-\beta\mathcal{C}_n(\v^n(t)+\z(t))\\ &\quad+(\eta-\alpha)\z^n(t)+\f^n,\\
				\v^n(0)&=\v_0^n,
			\end{aligned}
			\right.
		\end{equation}
		Since $\B_n(\cdot)$ and $\mathcal{C}_n(\cdot)$ are  locally Lipschitz, the system (\ref{finite-dimS}) has a unique local solution $\v^n\in\mathrm{C}([0,T^*];\H_n)$, for some $0<T^*<T$. The following a priori estimates show that the time $T^*$ can be extended to time $T$. Taking the inner product with $\v^n(\cdot)$ to the first equation of \eqref{cscbf_n}, we obtain
		\begin{align}\label{S1}
			\frac{1}{2}\frac{\d}{\d t}\|\v^n(t)\|^2_{\H}&=-\mu\|\v^n(t)\|^2_{\V}-\alpha\|\v^n(t)\|^2_{\H}-b(\v^n(t)+\z(t),\v^n(t)+\z(t),\v^n(t))\nonumber\\&\quad-\left(\mathcal{C}(\v^n(t)+\z(t)),\v^n(t)\right)+ (\eta-\alpha)(\z(t),\v^n(t))+(\f,\v^n(t))\nonumber\\&=-\mu\|\v^n(t)\|^2_{\V}-\alpha\|\v^n(t)\|^2_{\H}-\beta\|\v^n(t)+\z(t)\|^{r+1}_{\wi\L^{r+1}}+b(\v^n(t),\v^n(t),\z(t))\nonumber\\&\quad+b(\z(t),\v^n(t),\z(t))+\beta\left(\mathcal{C}(\v^n(t)+\z(t)),\z(t)\right)+ ((\eta-\alpha)\z(t),\v^n(t))\nonumber\\&\quad+\langle\f,\v^n(t)\rangle.
		\end{align}
		Next, we estimate each term of RHS of \eqref{S1} as 
		\begin{align}
			\left|b(\v^n,\v^n,\z)\right|&\leq\|\v^n\|_{\wi\L^4}\|\v^n\|_{\V}\|\z\|_{\wi\L^4}\leq C\|\v^n\|^{\frac{1}{2}}_{\H}\|\v^n\|^{\frac{3}{2}}_{\V}\|\z\|_{\wi\L^4}\nonumber\\&\leq \frac{\mu}{6}\|\v^n\|^2_{\V}+C\|\v^n\|^2_{\H}\|\z\|^4_{\wi\L^4},\\
			\left|b(\z,\v^n,\z)\right|&\leq \|\z\|^2_{\wi\L^4}\|\v^n\|_{\V}\leq \frac{\mu}{6}\|\v^n\|^2_{\V}+C\|\z\|^4_{\wi\L^4},\\
			\beta\left|\left(\mathcal{C}(\v^n+\z),\z\right)\right|&\leq\beta \|\v^n+\z\|^r_{\wi\L^{r+1}}\|\z\|_{\wi\L^{r+1}}\leq\frac{\beta}{2}\|\v^n+\z\|^{r+1}_{\wi\L^{r+1}} + C\|\z\|^{r+1}_{\wi\L^{r+1}},\\
			\left|((\eta-\alpha)\z(t),\v^n(t))+\langle\f,\v^n(t)\rangle\right|&\leq \frac{\mu}{6}\|\v^n\|^2_{\V}+\frac{\alpha}{2}\|\v^n\|^2_{\H}+C\|\f\|^2_{\V'}+C\|\z\|^2_{\H}.\label{S2}
		\end{align}
		Combining \eqref{S1}-\eqref{S2}, we deduce that 
		\begin{align}\label{S3}
		&\frac{\d}{\d t}\|\v^n(t)\|^2_{\H}+\mu\|\v^n(t)\|^2_{\H}+\alpha\|\v^n(t)\|^2+\beta\|\v^n(t)+\z(t)\|^{r+1}_{\wi\L^{r+1}}+\beta\|\z(t)\|^{r+1}_{\wi\L^{r+1}}\nonumber\\&\leq C\|\v^n(t)\|^2_{\H}\|\z(t)\|^4_{\wi\L^4}+C\|\z(t)\|^2_{\H}+C\|\z(t)\|^{4}_{\wi\L^{4}} +C\|\z(t)\|^{r+1}_{\wi\L^{r+1}}+C\|\f\|^2_{\V'}
		\end{align}
		An application of Gronwall's inequality  in \eqref{S3} leads to
		\begin{align}\label{S4}
				\|\v^n(t)\|^2_{\H}&\leq \|\v^n(0)\|^2_{\H}e^{\int_{0}^{t}\|\z(\zeta)\|^4_{\wi\L^4}\d\zeta} \nonumber\\&\quad+C\int_{0}^{t}e^{\int_{\xi}^{t}\|\z(\zeta)\|^4_{\wi\L^4}\d\zeta}\big[\|\z(t)\|^2_{\H}+\|\z(t)\|^{4}_{\wi\L^{4}} +\|\z(t)\|^{r+1}_{\wi\L^{r+1}}+\|\f\|^2_{\V'}\big]\d\xi.
		\end{align}
		Furthermore, since $\z\in \mathrm{L}^{\infty}(0,T; \H)\cap\mathrm{L}^{2} (0, T; \D(\A))$, $\|\v^n(0)\|_{\H}\leq\|\v(0)\|_{\H}$ and $\f\in\V'$, we have from \eqref{S4} that 
		\begin{align*}
			\sup_{t\in[0,T]}\|\v^n(t)\|^2_{\H}<\infty,
		\end{align*}
		from which we infer that 
		\begin{align}\label{S5}
			\{\v^n\}_{n\in\N} \text{ is a bounded sequence in }\mathrm{L}^{\infty}(0,T;\H).
		\end{align}
		Now, integrating \eqref{S3} from $0$ to $T$ and using \eqref{S5}, we obtain
		\begin{align}\label{S6}
			\{\v^n\}_{n\in\N} \text{ is a bounded sequence in }\mathrm{L}^{2}(0,T;\V)\cap\mathrm{L}^{r+1}(0,T;\widetilde{\L}^{r+1}).
		\end{align}
\vskip 1mm
\noindent
\textbf{Uniform estimate for $\frac{\d\v^n}{\d t}$:} \emph{When $r\in[1,3)$.} Let us first consider the  case $r\in(1,3)$. For $r=1$, one can do in similar way, even it is much easier. For $\varepsilon\in(0,1)$ and any arbitrary element $\boldsymbol{\psi}\in\mathrm{L}^2(0,T;\V)$, using H\"older's inequality and Sobolev's embeddings, we have from \eqref{cscbf_n}
		\begin{align*}
			&\left|\left\langle\frac{\d\v^n(t)}{\d t},\boldsymbol{\psi}(t)\right\rangle\right|\nonumber\\&\leq \mu\left|(\nabla\v^n(t),\nabla\boldsymbol{\psi}(t))\right|+\alpha\left|(\v^n(t),\boldsymbol{\psi}(t))\right|+\left|b(\v^n(t)+\z(t),\boldsymbol{\psi}(t),\v^n(t)+\z(t))\right|\nonumber\\&\quad+\beta\left|\left\langle\mathcal{C}(\v^n(t)+\z(t)),\boldsymbol{\psi}(t)\right\rangle\right| +(\eta-\alpha)\left|(\z(t),\boldsymbol{\psi}(t))\right|+\left|(\f,\boldsymbol{\psi}(t))\right|\nonumber\\&\leq \mu\|\v^n(t)\|_{\V}\|\boldsymbol{\psi}(t)\|_{\V}+\alpha\|\v^n(t)\|_{\H}\|\boldsymbol{\psi}(t)\|_{\H}+\|\v^n(t)+\z(t)\|^2_{\wi\L^4}\|\boldsymbol{\psi}(t)\|_{\V}+\|\z(t)\|_{\H}\|\boldsymbol{\psi}(t)\|_{\H}\nonumber\\&\quad+\|\f\|_{\V'}\|\boldsymbol{\psi}(t)\|_{\V}+\beta\|\v^n(t)+\z(t)\|^r_{\wi\L^{r(1+\varepsilon)}}\|\boldsymbol{\psi}(t)\|_{\wi\L^{\frac{1+\varepsilon}{\varepsilon}}}\nonumber\\&\leq C\big[\|\v^n(t)\|_{\V}+\|\v^n(t)\|_{\H}+\|\v^n(t)+\z(t)\|_{\H}\|\v^n(t)+\z(t)\|_{\V}+\|\z(t)\|_{\H}+\|\f\|_{\V'}\nonumber\\&\quad+\|\v^n(t)+\z(t)\|^r_{\wi\L^{r(1+\varepsilon)}}\big]\|\boldsymbol{\psi}(t)\|_{\V},
		\end{align*}
		which implies that
		\begin{align*}
			&\left\|\frac{\d\v^n(t)}{\d t}\right\|_{\V'}\\&\leq C\big[\|\v^n(t)\|_{\V}+\|\v^n(t)+\z(t)\|_{\H}\|\v^n(t)+\z(t)\|_{\V}+\|\z(t)\|_{\H}+\|\f\|_{\V'}+\|\v^n(t)+\z(t)\|^r_{\wi\L^{r(1+\varepsilon)}}\big].
		\end{align*}
	Now since $r<3$, we can choose $\varepsilon$ small enough such that $r<1+\frac{2}{1+\varepsilon}$, so that 
	$$q:=\frac{2}{2-(r-1)(1+\varepsilon)}\in(1,\infty).$$
	Let $q':=\frac{q}{q-1}=\frac{2}{(r-1)(1+\varepsilon)}$ and $\lambda\in(0,1)$. Applying interpolation inequality and choosing $\lambda:=\frac{1}{r}$, we obtain
	\begin{align*}
		\|\v^n(t)+\z(t)\|^{r}_{\wi\L^{r(1+\varepsilon)}}&\leq\|\v^n(t)+\z(t)\|^{\lambda r}_{\wi\L^{\lambda qr(1+\varepsilon)}}\|\v^n(t)+\z(t)\|^{(1-\lambda)r}_{\wi\L^{(1-\lambda)q'r(1+\varepsilon)}}\\&=\|\v^n(t)+\z(t)\|_{\wi\L^{q(1+\varepsilon)}}\|\v^n(t)+\z(t)\|^{r-1}_{\H}\\&\leq C\|\v^n(t)+\z(t)\|_{\V}\|\v^n(t)+\z(t)\|^{r-1}_{\H}.
	\end{align*}
	Therefore, we have 
		\begin{align*}
			\int_{0}^{T}\left\|\frac{\d\v^n(t)}{\d t}\right\|^{2}_{\V'}\d t&\leq C\int_{0}^{T}\big[\|\v^n(t)\|^{2}_{\V}+\|\v^n(t)+\z(t)\|^2_{\H}\|\v^n(t)+\z(t)\|^2_{\V}+\|\z(t)\|^2_{\H}+\|\f\|^2_{\V'}\nonumber\\&\qquad\qquad+\|\v^n(t)+\z(t)\|^2_{\V}\|\v^n(t)+\z(t)\|^{2(r-1)}_{\H}\big]\d t,
		\end{align*}
		which implies that $\frac{\d \v^n}{\d t}\in \mathrm{L}^{2}(0,T;\V')$.
		\vskip 2mm
		\noindent
	 \emph{When $r=3$.} For any arbitrary element $\boldsymbol{\psi}\in\mathrm{L}^2(0,T;\V)\cap\mathrm{L}^{4}(0,T;\wi\L^{4})$, using H\"older's inequality and Sobolev embeddings, we have from \eqref{cscbf_n}
		\begin{align*}
			&\int_{0}^{T}\left|\left\langle\frac{\d\v^n(t)}{\d t},\boldsymbol{\psi}(t)\right\rangle\right|\d t\nonumber\\&\leq \mu\int_{0}^{T}\left|(\nabla\v^n(t),\nabla\boldsymbol{\psi}(t))\right|\d t+\int_{0}^{T}\left|b(\v^n(t)+\z(t),\boldsymbol{\psi}(t),\v^n(t)+\z(t))\right|\d t\nonumber\\&\quad+\alpha\int_{0}^{T}\left|(\v^n(t),\boldsymbol{\psi}(t))\right|\d t+\beta\int_{0}^{T}\left|\left\langle\mathcal{C}(\v^n(t)+\z(t)),\boldsymbol{\psi}(t)\right\rangle\right|\d t +(\eta-\alpha)\int_{0}^{T}\left|(\z(t),\boldsymbol{\psi}(t))\right|\d t\nonumber\\&\quad+\int_{0}^{T}\left|(\f,\boldsymbol{\psi}(t))\right|\d t\nonumber\\&\leq C\bigg(\int_{0}^{T}\bigg[\|\v^n(t)\|^2_{\V}+\|\v^n(t)+\z(t)\|^4_{\wi\L^4}+\|\z(t)\|^2_{\H}+\|\f\|^2_{\V'}\bigg]\d t\bigg)^{\frac{1}{2}}\bigg(\int_{0}^{T}\|\boldsymbol{\psi}(t)\|^2_{\V}\d t\bigg)^{\frac{1}{2}}\nonumber\\&\qquad+\beta\bigg(\int_{0}^{T}\|\v^n(t)+\z(t)\|^4_{\wi\L^4}\d t\bigg)^{\frac{3}{4}}\bigg(\int_{0}^{T}\|\boldsymbol{\psi}(t)\|^4_{\wi\L^{4}}\d t\bigg)^{\frac{1}{4}},
		\end{align*}
		which gives that $\frac{\d \v^n}{\d t}\in \mathrm{L}^{2}(0,T;\V')+\mathrm{L}^{\frac{4}{3}}(0,T;\wi\L^{\frac{4}{3}})$.
	
		 Using \eqref{S5}, \eqref{S6} and the \emph{Banach-Alaoglu theorem}, we infer that there exists an element $\v\in\mathrm{L}^{\infty}(0,T;\H)\cap\mathrm{L}^{2}(0,T;\V)\cap\mathrm{L}^{r+1}(0,T;\widetilde{\L}^{r+1})$ and $\frac{\d \v}{\d t}\in \mathrm{L}^{2}(0,T;\V')$ (for $r\in[1,3)$) and $\frac{\d \v}{\d t}\in\mathrm{L}^{2}(0,T;\V')+\mathrm{L}^{\frac{4}{3}}(0,T;\wi\L^{\frac{4}{3}})$ (for $r=3$) such that
		\begin{align}
			\v^n\xrightharpoonup{w^*}&\ \v\text{ in }	\mathrm{L}^{\infty}(0,T;\H),\label{S7}\\
			\v^n\xrightharpoonup{w}&\ \v\text{ in } \mathrm{L}^{2}(0,T;\V)\cap\mathrm{L}^{r+1}(0,T;\widetilde{\L}^{r+1}),\label{S8}
		\end{align}
		\begin{equation}\label{S8d}
		\left\{
		\begin{aligned}
			\frac{\d \v^n}{\d t}\xrightharpoonup{w}&\frac{\d \v}{\d t} \text{ in }\mathrm{L}^{2}(0,T;\V'), &\text{ for } r\in[1,3), \\
	\frac{\d \v^n}{\d t}\xrightharpoonup{w}&\frac{\d \v}{\d t} \text{ in } \mathrm{L}^{2}(0,T;\V')+\mathrm{L}^{\frac{4}{3}}(0,T;\wi\L^{\frac{4}{3}}) ,&\text{ for } r=3,
		\end{aligned}
		\right.
	\end{equation}
		along a subsequence. The fact that $\v^n\in\mathrm{L}^{2}(0,T;\V)$ and $\frac{\d \v^n}{\d t}\in \mathrm{L}^{\frac{r+1}{r}}(0,T;\V')$ imply $\v^n\in\mathrm{L}^{2}(0,T;\H^1_0(\mathcal{O}_R))$ and $\frac{\d \v^n}{\d t}\in \mathrm{L}^{\frac{r+1}{r}}(0,T;\H^{-1}(\mathcal{O}_R))$, where $\mathcal{O}_R=\mathcal{O}\cap\{x\in\R^2:|x|< R\}.$ Since, $\v^n\in\mathrm{L}^{2}(0,T;\H^1_0(\mathcal{O}_R))$, $\frac{\d \v^n}{\d t}\in \mathrm{L}^{\frac{r+1}{r}}(0,T;\H^{-1}(\mathcal{O}_R))$, the embedding $\H^1_0(\mathcal{O}_R)\subset\L^2(\mathcal{O}_R)\subset\H^{-1}(\mathcal{O}_R)$ is continuous  and the embedding $\H^1_0(\mathcal{O}_R)\subset\L^2(\mathcal{O}_R)$ is compact, then  the \emph{Aubin-Lions compactness lemma} implies that 
		\begin{align}\label{S9}
			\v^n\to\v \ \text{ strongly in } \ \mathrm{L}^2(0,T;\L^2(\mathcal{O}_R)).
		\end{align}
		Next, we prove that $\v$ is a solution to the system \eqref{cscbf}. Let $\psi:[0,T]\to\R$ be a continuously differentiable function. Also, let $\boldsymbol{\phi}\in\H_m$ for some $m\in\N$. Then from \eqref{finite-dimS}, we have 
		\begin{align}\label{S10}
			&\int_{0}^{T}\left(\frac{\d\v^n(t)}{\d t},\psi(t)\boldsymbol{\phi}\right)\d t\nonumber\\&=-\mu\int_{0}^{T} (\A_n\v^n(t),\psi(t)\boldsymbol{\phi})\d t-\alpha\int_{0}^{T}(\v^n(t),\psi(t)\boldsymbol{\phi})\d t-\int_{0}^{T}(\B_n(\v^n(t)+\z(t)),\psi(t)\boldsymbol{\phi})\d t\nonumber\\&\quad-\beta\int_{0}^{T}(\mathcal{C}_n(\v^n(t)+\z(t)),\psi(t)\boldsymbol{\phi})\d t +(\eta-\alpha)\int_{0}^{T}(\z^n(t),\psi(t)\boldsymbol{\phi})\d t+\int_{0}^{T}(\f^n,\psi(t)\boldsymbol{\phi})\d t,
		\end{align}
		where we have used an integration by parts. Our next goal is to pass limit in \eqref{S10} as $n\to \infty$. Due to the choice of $\boldsymbol{\phi}\in\H_m$ for some $m\in\N$, we can say that there exists $R\in\N$ such that $\text{supp}\ \boldsymbol{\phi}\subset\mathcal{O}_R.$ Since $\psi(\cdot)\boldsymbol{\phi}\in\mathrm{L}^{2}(0,T;\V)\cap\mathrm{L}^{r+1}(0,T;\widetilde{\L}^{r+1})$, in view of \eqref{S8d}, we obtain 
		\begin{align}\label{S11}
			\int_{0}^{T}\left(\frac{\d\v^n(t)}{\d t},\psi(t)\boldsymbol{\phi}\right)\d t-\int_{0}^{T}\left\langle\frac{\d\v(t)}{\d t},\psi(t)\boldsymbol{\phi}\right\rangle\d t=\int_{0}^{T}\left\langle\frac{\d\v^n(t)}{\d t}-\frac{\d\v}{\d t},\psi(t)\boldsymbol{\phi}\right\rangle\d t\to 0,
		\end{align}
		as $n\to \infty$. Since $\psi(\cdot)\boldsymbol{\phi}\in\mathrm{L}^2(0,T;\L^2(\mathcal{O}_R))$, we obtain
		\begin{align}\label{S12}
			\left|\int_{0}^{T}(\v^n(t),\psi(t)\boldsymbol{\phi})\d t-\int_{0}^{T}(\v(t),\psi(t)\boldsymbol{\phi})\d t\right|&\leq\|\v^n-\v\|_{\mathrm{L}^2(0,T;\L^2(\mathcal{O}_R))}\|\psi(\cdot)\boldsymbol{\phi}\|_{\mathrm{L}^2(0,T;\L^2(\mathcal{O}_R))}\nonumber\\&\to0 \text{ as } n\to\infty,
		\end{align}
		where we have used the strong convergence obtained in \eqref{S9}. Let us choose $n\geq m$ so that $\H_m\subset\H_n$ and $\P_n\boldsymbol{\phi}=\boldsymbol{\phi}$. Since $\psi(\cdot)\boldsymbol{\phi}\in\mathrm{L}^2(0,T;\V)$, it is immediate that
		\begin{align}\label{S13}
			\int_{0}^{T} (\A_n\v^n(t),\psi(t)\boldsymbol{\phi})\d t-	\int_{0}^{T} (\!(\v(t),\psi(t)\boldsymbol{\phi})\!)\d t=\int_{0}^{T} (\!(\v^n(t)-\v(t),\psi(t)\boldsymbol{\phi})\!)\d t\to0 \text{ as } n\to\infty,
		\end{align}
		where we have used the weak convergence given in \eqref{S8}. To prove the convergence of third term of the right hand side of \eqref{S10}, we consider
		\begin{align}\label{S14}
			&\left|\int_{0}^{T}(\B_n(\v^n(t)+\z(t)),\psi(t)\boldsymbol{\phi})\d t-\int_{0}^{T}(\B(\v(t)+\z(t)),\psi(t)\boldsymbol{\phi})\d t\right|\nonumber\\&
			\leq \underbrace{\left|\int_{0}^{T}b(\v^n(t),\v^n(t),\psi(t)\boldsymbol{\phi})\d t-\int_{0}^{T}b(\v(t),\v(t),\psi(t)\boldsymbol{\phi})\d t\right|}_{:=B_1(n)}\nonumber\\&\qquad+\left|\int_{0}^{T}b(\z(t),\v^n(t)-\v(t),\psi(t)\boldsymbol{\phi})\d t+\int_{0}^{T}b(\v^n(t)-v(t),\z(t),\psi(t)\boldsymbol{\phi})\d t\right|\nonumber\\&
			\leq B_1(n)+2\int_{0}^{T}\|\z(t)\|_{\wi\L^4}\|\v^n(t)-\v(t)\|_{\L^4(\mathcal{O}_R)}\|\psi(t)\boldsymbol{\phi}\|_{\wi\L^2}\d t\nonumber\\&\leq
			B_1(n)+C\int_{0}^{T}\|\z(t)\|_{\wi\L^4}\|\v^n(t)-\v(t)\|^{1/2}_{\L^2(\mathcal{O}_R)}\|\v^n(t)-\v(t)\|^{1/2}_{\V}\d t\nonumber\\&
			\leq B_1(n)+CT^{1/4} \|\z\|_{\mathrm{L}^4(0,T;\widetilde{\L}^4)}\|\v^n-\v\|^{1/2}_{\mathrm{L}^2(0,T;\L^2(\mathcal{O}_R))}\left[\|\v^n\|^{1/2}_{\mathrm{L}^2(0,T;\V)}+\|\v\|^{1/2}_{\mathrm{L}^2(0,T;\V)}\right] \nonumber\\& \to0 \text{ as } n\to \infty,
		\end{align}
		where we have used the convergence from Lemma \ref{convergence_b*} and \eqref{S9}. From Lemma \ref{convergence_c2_1}, we get 
		\begin{align}\label{S15}
			\int_{0}^{T}(\mathcal{C}_n(\v^n(t)+\z(t)),\psi(t)\boldsymbol{\phi})\d t&=\int_{0}^{T}\left\langle\mathcal{C}(\v^n(t)+\z(t)),\psi(t)\boldsymbol{\phi}\right\rangle\d t\nonumber\\&\to \int_{0}^{T}\left\langle\mathcal{C}(\v(t)+\z(t)),\psi(t)\boldsymbol{\phi}\right\rangle\d t \text{ as } n\to\infty.
		\end{align}
		Furthermore, it is immediate that
		\begin{align}\label{S16}
			\int_{0}^{T}((\eta-\alpha)\z^n(t)+\f^n,\psi(t)\boldsymbol{\phi})\d t=(\eta-\alpha)\int_{0}^{T}(\z^n(t),\psi(t)\boldsymbol{\phi})\d t+\int_{0}^{T}\langle\f,\psi(t)\boldsymbol{\phi}\rangle\d t.
		\end{align}
		Hence, on passing limit to \eqref{S10} as $n\to\infty$ with the help of \eqref{S11}-\eqref{S16}, we obtain
		\begin{align}\label{S17}
			&\int_{0}^{T}\left\langle\frac{\d\v(t)}{\d t},\psi(t)\boldsymbol{\phi}\right\rangle\d t\nonumber\\&=-\mu\int_{0}^{T} \left\langle\A\v(t),\psi(t)\boldsymbol{\phi}\right\rangle\d t-\alpha\int_{0}^{T}(\v(t),\psi(t)\boldsymbol{\phi})\d t-\int_{0}^{T}\left\langle\B(\v(t)+\z(t)),\psi(t)\boldsymbol{\phi}\right\rangle\d t\nonumber\\&\quad-\beta\int_{0}^{T}\left\langle\mathcal{C}(\v(t)+\z(t)),\psi(t)\boldsymbol{\phi}\right\rangle\d t +(\eta-\alpha)\int_{0}^{T}(\z(t),\psi(t)\boldsymbol{\phi})\d t+\int_{0}^{T}\left\langle\f,\psi(t)\boldsymbol{\phi}\right\rangle\d t.
		\end{align}
		Since \eqref{S17} holds for any $\boldsymbol{\phi}\in\cup_{m=1}^{\infty}\H_m$ and $\cup_{m=1}^{\infty}\H_m$ is dense in $\V$, we have that \eqref{S17} holds true for all $\boldsymbol{\phi}\in\V$ and $\psi\in\mathrm{C}^1([0,T])$. Hence $\v(\cdot)$ solves \eqref{W-CSCBF} and satisfies first equation of \eqref{cscbf}. 
		\vskip 2mm
		\noindent
	\textbf{Step II.}	\emph{Energy equality:} \emph{When $r\in[1,3)$.} Since $\v\in\mathrm{L}^{2}(0,T;\V)$ and $\frac{\d\v}{\d t}\in \mathrm{L}^{2}(0,T;\V')$, we infer from Theorem 3, page 303, \cite{LCE} that $\v\in\C([0,T];\H)$, and hence the energy equality is satisfied.
	\vskip 2mm
	\noindent
	 \emph{When $r=3$.} Note that the embedding of $\H\subset\V'$ is continuous and $\v\in \mathrm{L}^{\infty}(0,T;\H)$ implies $\v\in \mathrm{L}^{\infty}(0,T;\V')$. Thus, we get $\v,\frac{\d\v}{\d t}\in \mathrm{L}^{\frac{r+1}{r}}(0,T;\V')$ and then invoking Theorem 2, section 5.9.2 \cite{LCE}, it is immediate that $\v\in\C([0,T];\V')$. Since $\H$ is reflexive, using Proposition 1.7.1 \cite{PCAM}, we obtain $\v\in\C_w([0,T];\H)$ and the map $t\mapsto\|\v(t)\|_{\H}$ is bounded, where $\C_w([0,T];\H)$ denotes the space of functions $\v:[0,T]\to\H$ which are weakly continuous. Now, we  show that $\v$ satisfies energy equality and hence $\v\in\C([0, T];\H)$. First define $\mathcal{V}_{T}:=\{\boldsymbol{\phi}:\boldsymbol{\phi}\in\mathrm{C}^{\infty}_{0}(\mathcal{O}\times[0,T))\}$. Observe that, for each $\boldsymbol{\phi}\in\mathcal{V}_{T}$, $\boldsymbol{\phi}(\cdot,T)=0$ and $\mathcal{V}_{T}$ is dense in $\mathrm{L}^{p}(0,T;\mathbb{H}^1(\mathcal{O}))$ (for case $p=2$ see Lemma 2.6, \cite{GPG}). For $\v\in\mathrm{L}^p(0,T;\mathrm{X})$, $1\leq p<\infty$ and $T>h>0$, the mollifier $\v_h$ (in the sense of Friederichs) of $\v$ is defined by 
		\begin{align*}
			\v_h(t)=\int_{0}^{T}j_h(t-\zeta)\v(\zeta)\d\zeta,
		\end{align*}
		where $j_h(\cdot)$ is an infinite times differentiable function having support in $(-h,h)$, which is even and positive, such that $\int_{-\infty}^{+\infty}j_h(\zeta)\d\zeta=1$. In view of Lemma 2.5, \cite{GPG}, we have that for $\v\in\mathrm{L}^p(0,T;\mathrm{X})$ with $1\leq p<\infty$, $\v_h\in\mathrm{C}^k([0,T];\mathrm{X})$ for all $k\geq0$ and
		\begin{align}\label{335}
			\lim_{h\to0}\|\v_h-\v\|_{\mathrm{L}^p(0,T;\mathrm{X})}=0.
		\end{align}
		Furthermore, if $\{\v_m\}_{m\in\N}\subset\mathrm{L}^p(0,T;\mathrm{X})$ converges to $\v$ in the norm of $\mathrm{L}^p(0,T;\mathrm{X})$, then
		\begin{align}\label{336}
			\lim_{m\to\infty}\|(\v_m)_h-\v_h\|_{\mathrm{L}^p(0,T;\mathrm{X})}=0.
		\end{align}
		We write the weak solution of \eqref{cscbf} as
		\begin{align}\label{337}
			\int_{0}^{t}\biggl\{\left\langle\frac{\d \v}{\d t}+\mu\A\v+\B(\v+\z)+\beta\mathcal{C}(\v+\z)-\f,\boldsymbol{\phi}\right\rangle+\left(\alpha\v+(\alpha-\eta)\z,\boldsymbol{\phi}\right)\biggr\}\d\zeta=0,
		\end{align}
		for all $t<T$ and $\boldsymbol{\phi}\in\mathcal{V}_{T}$. Let $\{\v_m\}_{m\in\N}\subset\mathcal{V}_T$ be a sequence converging to $\v\in\mathrm{L}^{2}(0,T;\V)$, that is, $\|\v_m-\v\|_{\mathrm{L}^{2}(0,T;\V)}\to0$ as $m\to\infty$. If $\v\in\mathrm{L}^{\infty}(0,T;\H)$, then we also have the convergence $\|\v_m-\v\|_{\mathrm{L}^{4}(0,T;\wi\L^{4})}\to0$ as $m\to\infty$ (by \eqref{lady}). Choosing $\boldsymbol{\phi}=(\v_m)_h=:\v_{m,h}$ in \eqref{337}, where $(\cdot)_h$ is the mollification operator discussed above, for $0\leq t<T$, we get
		\begin{align}\label{338}
			\int_{0}^{t}\biggl\{\left\langle\frac{\d \v}{\d t},\v_{m,h}\right\rangle+\mu\left(\nabla\v,\nabla\v_{m,h}\right)&+\left\langle\B(\v+\z)+\beta\mathcal{C}(\v+\z)-\f,\v_{m,h}\right\rangle\nonumber\\&\quad+\left(\alpha\v+(\alpha-\eta)\z,\v_{m,h}\right)\biggr\}\d\zeta=0.
		\end{align}
		Using \eqref{336}, we obtain 
		\begin{align*}
			\left|\int_{0}^{t}\left\langle\frac{\d \v}{\d t},\v_{m,h}-\v_{h}\right\rangle\d\zeta\right|\leq\left\|\frac{\d \v}{\d t}\right\|_{\mathrm{L}^{2}(0,T;\V')+\mathrm{L}^{\frac{4}{3}}(0,T;\wi\L^{\frac{4}{3}})}\|\v_{m,h}-\v_h\|_{\mathrm{L}^2(0,T;\V)\cap\mathrm{L}^{4}(0,T;\wi\L^{4})}\to0,
		\end{align*}
		and
		\begin{align*}
			\left|\int_{0}^{t}\left(\nabla \v,\nabla\v_{m,h}-\nabla\v_{h}\right)\d\zeta\right|\leq\| \v\|_{\mathrm{L}^{2}(0,T;\V)}\|\v_{m,h}-\v_h\|_{\mathrm{L}^{2}(0,T;\V)}\to0,
		\end{align*}
		as $m\to\infty$. Since $\v\in\mathrm{L}^4(0,T;\wi\L^4)$, using \eqref{336}, we have 
		\begin{align*}	\left|\int_{0}^{t}\left\langle\B(\v+\z),\v_{m,h}-\v_{h}\right\rangle\d\zeta\right|\leq\| \v+\z\|^2_{\mathrm{L}^{4}(0,T;\widetilde{\L}^4)}\|\v_{m,h}-\v_h\|_{\mathrm{L}^{2}(0,T;\V)}\to0,
		\end{align*}
and
		\begin{align*}	\left|\int_{0}^{t}\left\langle\mathcal{C}(\v+\z),\v_{m,h}-\v_{h}\right\rangle\d\zeta\right|\leq\| \v+\z\|^{3}_{\mathrm{L}^{4}(0,T;\widetilde{\L}^{4})}\|\v_{m,h}-\v_h\|_{\mathrm{L}^{4}(0,T;\wi\L^{4})}\to0,
		\end{align*}
		as $m\to\infty$. Similarly, using \eqref{336}, we get
		\begin{align*}
			\int_{0}^{t}\biggl\{-\left\langle\f,\v_{m,h}\right\rangle+\left(\alpha\v+(\alpha-\eta)\z,\v_{m,h}\right)\biggr\}\d\zeta\to		\int_{0}^{t}\biggl\{-\left\langle\f,\v_{h}\right\rangle+\left(\alpha\v+(\alpha-\eta)\z,\v_{h}\right)\biggr\}\d\zeta,
		\end{align*}
		as $m\to\infty$. Hence, passing limit $m\to\infty$ in \eqref{338}, we obtain
		\begin{align}\label{339}
			\int_{0}^{t}\biggl\{\left\langle\frac{\d \v}{\d t},\v_{h}\right\rangle+\mu\left(\nabla\v,\nabla\v_{h}\right)&+\left\langle\B(\v+\z)+\beta\mathcal{C}(\v+\z)-\f,\v_{h}\right\rangle\nonumber\\&+\left(\alpha\v+(\alpha-\eta)\z,\v_{h}\right)\biggr\}\d\zeta=0.
		\end{align}
		Using \eqref{335} and similar arguments as above, we obtain the following convergence
		\begin{align}\label{340}
			&\lim_{h\to0}\int_{0}^{t}\biggl\{\mu\left(\nabla\v,\nabla\v_{h}\right)+\left\langle\B(\v+\z)+\beta\mathcal{C}(\v+\z)-\f,\v_{h}\right\rangle+\left(\alpha\v+(\alpha-\eta)\z,\v_{h}\right)\biggr\}\d\zeta\nonumber\\&=\int_{0}^{t}\biggl\{\mu\left(\nabla\v,\nabla\v\right)+\left\langle\B(\v+\z)+\beta\mathcal{C}(\v+\z)-\f,\v\right\rangle+\left(\alpha\v+(\alpha-\eta)\z,\v\right)\biggr\}\d\zeta.
		\end{align}
		Using integration by parts, we get
		\begin{align}\label{341}
			\int_{0}^{t}\left\langle\frac{\d \v}{\d t},\v_{h}\right\rangle\d\zeta&=-\int_{0}^{t}\left\langle\v,\frac{\d \v_h}{\d t}\right\rangle\d\zeta+\left(\v(0),\v_h(0)\right)-\left(\v(t),\v_h(t)\right)\nonumber\\&=-\int_{0}^{t}\int_{0}^{\zeta}\frac{\d j_h(\zeta-s)}{\d t}\left(\v(\zeta),\v(s)\right)\d s\d\zeta+\left(\v(0),\v_h(0)\right)-\left(\v(t),\v_h(t)\right)\nonumber\\&=\left(\v(0),\v_h(0)\right)-\left(\v(t),\v_h(t)\right)\nonumber\\&\to\left(\v(0),\v(0)\right)-\left(\v(t),\v(t)\right),
		\end{align}
		as $h\to0$, where we have used the property of mollifiers and the fact that the kernel $j_{h}(s)$ in the definition of mollifier is even in $(-h,h)$. From \eqref{339}-\eqref{341}, we infer that $\v(\cdot)$ satisfies the energy equality: 
		\begin{align}\label{521}
			&\|\v(t)\|_{\H}^2+2\mu\int_0^t\|\v(\zeta)\|_{\V}^2\d\zeta+2\alpha\int_0^t\|\v(\zeta)\|_{\H}^2\d\zeta\nonumber\\&= \|\boldsymbol{x}-\z(0)\|_{\H}^2-2\int_0^t\langle\B(\v(\zeta)+\z(\zeta)),\v(\zeta)\rangle\d \zeta -2\beta\int_0^t\langle\mathcal{C}(\v(\zeta)+\z(\zeta)),\v(\zeta)\rangle\d s\nonumber\\&\quad+2\int_0^t\langle\f,\v(\zeta)\rangle\d \zeta+2(\eta-\alpha)\int_0^t(\z(\zeta),\v(\zeta))\d \zeta,
		\end{align}
		for all $t\in[0,T]$. Recalling that every weak solution of \eqref{cscbf} is $\H$ weakly continuous in time, all weak solutions satisfy the energy equality \eqref{521} and so, all weak solutions of \eqref{cscbf} belong to $\C([0, T];\H)$ (see \cite{GPG,HR} also).
		\vskip 2mm
		\noindent
	\textbf{Step III.}	\emph{Uniqueness :} Define $\mathfrak{F}=\v_1-\v_2$, where $\v_1$ and $\v_2$ are two weak solutions of the system \eqref{cscbf}. Then $\mathfrak{F}\in\mathrm{C}(0,T;\H)\cap\mathrm{L}^{2}(0,T;\V)\cap\mathrm{L}^{r+1}(0,T;\widetilde{\L}^{r+1})$ and satisfies
		\begin{equation}\label{Uni}
			\left\{
			\begin{aligned}
				\frac{\d\mathfrak{F}(t)}{\d t} &= -\mu \A\mathfrak{F} (t)-\alpha\mathfrak{F}(t)- \B(\v_1(t)+\z(t))+\B(\v_2(t)+\z(t))- \beta \mathcal{C}(\v_1 (t)+ \z(t)) \\&\quad+ \beta \mathcal{C}(\v_2 (t)+ \z(t)), \\
				\mathfrak{F}(0)&= \textbf{0},
			\end{aligned}
			\right.
		\end{equation}
	in the weak sense.	From the above equation, using energy equality, we obtain
		\begin{align}\label{U1}
			&	\frac{1}{2}\frac{\d}{\d t}\|\mathfrak{F}(t)\|^2_{\H}+\mu\|\mathfrak{F}(t)\|^2_{\V}+\alpha\|\mathfrak{F}(t)\|^2_{\H}+\beta\left\langle \mathcal{C}(\v_1 (t)+ \z(t)) - \mathcal{C}(\v_2 (t)+ \z(t)),\v_1(t)-\v_2(t)\right\rangle\nonumber\\&=-b(\v_1(t)+\z(t),\v_1(t)+\z(t),\mathfrak{F}(t)) +b(\v_2(t)+\z(t),\v_2(t)+\z(t),\mathfrak{F}(t))\nonumber\\&=b(\mathfrak{F}(t),\mathfrak{F}(t),\v_2(t)) +b(\mathfrak{F}(t),\mathfrak{F}(t),\z(t))\leq C\|\mathfrak{F}(t)\|^{1/2}_{\H}\|\mathfrak{F}(t)\|^{3/2}_{\V}\left[\|\v_2(t)\|_{\widetilde{\L}^4}+\|\z(t)\|_{\wi\L^4}\right]\nonumber\\&\leq \frac{\mu}{2}\|\mathfrak{F}(t)\|^{2}_{\V}+C\|\mathfrak{F}(t)\|^{2}_{\H}\left[\|\v_2(t)\|^4_{\widetilde{\L}^4}+\|\z(t)\|^4_{\wi\L^4}\right].
		\end{align}
		Due to \eqref{MO_c}, fourth term on the left hand side  of \eqref{U1} is positive. Therefore,
		\begin{align}\label{U2}
			&\frac{\d}{\d t}\|\mathfrak{F}(t)\|^2_{\H}\leq C\|\mathfrak{F}(t)\|^{2}_{\H}\left[\|\v_2(t)\|^4_{\widetilde{\L}^4}+\|\z(t)\|^4_{\wi\L^4}\right].
		\end{align}
		Applying Gronwall's inequality and using the fact that $\v_2,\z\in\mathrm{L}^{4}(0,T;\widetilde{\L}^{4})$ and $\mathfrak{F}(0)=\textbf{0}$, we obtain that $\v_1(t)=\v_2(t)$, for all $t\in[0,T]$, which completes the proof of uniqueness.
		\vskip 2mm
		\noindent
		\textbf{Step IV.}	\emph{Strong solution:} Let $r\in[1,3],$ $\x\in \V$, $\f\in\H$ and $\z\in\mathrm{L}^{\infty} (0, T; \V)\cap \mathrm{L}^2(0, T;\D(\A))$. Taking the inner product with $\A\v^n(\cdot)$ to the first equation of \eqref{cscbf_n}, we obtain
		\begin{align}\label{S18}
			&\frac{1}{2}\frac{\d}{\d t}\|\v^n(t)\|^2_{\V} +\mu\|\A\v^n(t)\|^2_{\H}+\alpha\|\v^n(t)\|^2_{\V}\nonumber\\&=   -b(\v^n(t)+\z^n(t),\v^n(t)+\z^n(t),\A\v^n(t))-\beta(\mathcal{C}\big(\v^n(t)+\z(t)\big),\A\v^n(t))   \nonumber\\&\quad+(\eta-\alpha)\big(\z(t), \A\v^n(t)\big)+\big(\f,\A\v^n(t)\big),\nonumber\\&=-b(\v^n(t),\v^n(t),\A\v^n(t))-b(\v^n(t),\z(t),\A\v^n(t))-b(\z(t),\v^n(t),\A\v^n(t))\nonumber\\&\quad-b(\z(t),\z(t),\A\v^n(t))-\beta(\mathcal{C}\big(\v^n(t)+\z(t)\big),\A\v^n(t))  +\big((\eta-\alpha)\z(t)+\f, \A\v^n(t)\big).
		\end{align}
		Using  \eqref{b1}, H\"older's  and Young's inequalities, we estimate
		\begin{align*}
			&\big|b(\v^n,\v^n,\A\v^n)+b(\v^n,\z,\A\v^n)+b(\z,\v^n,\A\v^n)+b(\z,\z,\A\v^n)\big|\nonumber\\&\leq  \frac{\mu}{2}\|\A\v^n\|^2_{\H} + C\|\v^n\|_{\H}^2 \|\v^n\|^4_{\V}+C\|\v^n\|_{\H}^2 \|\z\|^4_{\V}+C\|\z\|_{\H}\|\A\z\|_{\H} \|\v^n\|^2_{\V}+C\|\A\z\|_{\H} \|\z\|^3_{\V},
		\end{align*}
	and 
	\begin{align*}
			&|(\eta-\alpha)\big(\z,\A\v^n\big)+\big(\f, \A\v^n\big)|\nonumber\\&\leq \left|\eta-\alpha\right| \|\z\|_{\H}\|\A\v^n\|_{\H}	+\|\f\|_{\H}\|\A\v^n\|_{\H}\leq\frac{\mu}{2}\|\A\v^n\|^2_{\H} + C\|\z\|^2_{\V}+C\|\f\|^2_{\H}.
		\end{align*}
		Next, using  H\"older's inequality, Young's inequality and Sobolev's embedding, we find
		\begin{align}
				\big|(\mathcal{C}\big(\v^n+\z\big),\A\v^n)\big|&\leq \|\v^n+\z\|^{r}_{\wi \L^{2r}}\|\A\v^n\|_{\H}\leq \frac{\mu}{14}\|\A\v^n\|^2_{\H} + C\|\v^n+\z\|^{2r}_{\wi \L^{2r}}\nonumber\\&\leq \frac{\mu}{14}\|\A\v^n\|^2_{\H} + C\left[\|\z\|^{2r}_{\wi \L^{2r}}+ \|\v^n\|^{2r}_{\wi \L^{2r}}\right]\nonumber\\&\leq \frac{\mu}{14}\|\A\v^n\|^2_{\H} + C\left[\|\z\|^{2r}_{\V}+ \|\v^n\|^{2r}_{\wi \L^{2r}}\right].\label{S20}
		\end{align}
		We estimate the final term from \eqref{S20} using Gagliardo-Nirenberg's inequality (Theorem 1, \cite{Nirenberg}) as
		\begin{align}\label{S21}
			\|\v^n\|^{2r}_{\wi \L^{2r}}\leq\begin{cases}
				C\|\v\|^{2}_{\H}\|\v\|^{2r-2}_{\V} \leq C\|\v^n\|^{2}_{\H}+C\|\v^n\|^2_{\H}\|\v^n\|^4_{\V}, \ &\text{ for } r\in[1,3),\\
				C\|\v^n\|^{2}_{\H}\|\v^n\|^{4}_{\V}, \ \ \ \ &\text{ for } r=3.
			\end{cases}
		\end{align}
		We infer from the inequalities \eqref{S18}-\eqref{S21} that 
		\begin{align}\label{S22}
			&\frac{\d}{\d t}\|\v^n(t)\|^2_{\V} +\|\A\v^n(t)\|^2_{\H}+\|\v^n(t)\|^2_{\V}\leq\begin{cases}
				C [Q_1(t)\|\v^n(t)\|^2_{\V} + \widetilde{Q}_1(t)],  \ \text{ for } r\in[1,3),\\
				C[Q_2(t)\|\v^n(t)\|^2_{\V} +\widetilde{Q}_2(t)],  \ \text{ for } r=3,\\
			\end{cases} 
		\end{align}
		where
		\begin{align*}
			Q_1(t)&=\|\v^n(t)\|_{\H}^2\|\v^n(t)\|^2_{\V} + \|\z(t)\|_{\H}\|\A\z(t)\|_{\H},\ \\	Q_2(t)&=	Q_1(t)+\|\v^n(t)\|^{2}_{\H}\|\v^n(t)\|^2_{\V},\\
			\widetilde{Q}_1(t)&=  \|\v^n(t)\|_{\H}^2\|\z(t)\|^4_{\V}+ \|\A\z(t)\|_{\H}\|\z(t)\|^3_{\V} + \|\z(t)\|^{2r}_{\V} +\|\z(t)\|^2_{\V}+\|\f\|^2_{\H}+\|\v^n(t)\|^{2}_{\H},\\
			\widetilde{Q}_2(t)&=  \|\v^n(t)\|_{\H}^2\|\z(t)\|^4_{\V}+ \|\A\z(t)\|_{\H}\|\z(t)\|^3_{\V} + \|\z(t)\|^{6}_{\V} +\|\z(t)\|^2_{\V}+\|\f\|^2_{\H}.
		\end{align*}
		Note that  $\f\in\H$, $\z\in\mathrm{L}^{\infty} (0, T; \V)\cap \mathrm{L}^2(0, T;\D(\A))$ and $\v^n\in\mathrm{L}^{\infty} (0, T; \H)\cap \mathrm{L}^2(0, T;\V)$ imply that 
		\begin{align}\label{S23}
			\int_{0}^{T}Q_1(\zeta)\d\zeta<\infty,  \ \ \int_{0}^{T}\widetilde{Q}_1(\zeta)\d\zeta<\infty,  \ \ \int_{0}^{T}Q_2(\zeta)\d\zeta<\infty   \ \text{ and }\  \int_{0}^{T}\widetilde{Q}_2(\zeta)\d\zeta<\infty. 
		\end{align}
		Hence, an application of Gronwall's inequality gives that $\v^n\in\mathrm{L}^{\infty} (0, T; \V)\cap \mathrm{L}^2(0, T;\D(\A))$.
		
		For any arbitrary element $\boldsymbol{\psi}\in\H$, using H\"older's inequality and Sobolev embeddings, we have from \eqref{cscbf}
		\begin{align*}
			&\left|\left(\frac{\d\v^n(t)}{\d t},\boldsymbol{\psi}\right)\right|\nonumber\\&\leq \mu\left|(\A\v^n(t),\boldsymbol{\psi})\right|+\alpha\left|(\v^n(t),\boldsymbol{\psi})\right|+\left|b(\v^n(t)+\z(t),\v^n(t)+\z(t),\boldsymbol{\psi})\right|\nonumber\\&\quad+\beta\left|\left(\mathcal{C}(\v^n(t)+\z(t)),\boldsymbol{\psi}\right)\right| +(\eta-\alpha)\left|(\z(t),\boldsymbol{\psi})\right|+\left|(\f,\boldsymbol{\psi})\right|\nonumber\\&\leq C\left[\|\A\v^n(t)\|_{\H}+\|\v^n(t)\|_{\H}+\|\v^n(t)+\z(t)\|^2_{\wi\L^4}+\|\v^n(t)+\z(t)\|^r_{\wi\L^{2r}}+\|\z(t)\|_{\H}+\|\f\|_{\H}\right]\|\boldsymbol{\psi}\|_{\H},
		\end{align*}
		which gives that
		\begin{align*}
			&\int_{0}^{T}\left\|\frac{\d\v^n(t)}{\d t}\right\|^{2}_{\H}\d t\nonumber\\&\leq C\int_{0}^{T}\left[\|\A\v^n(t)\|^{2}_{\H}+\|\v^n(t)\|^2_{\H}+\|\v^n(t)+\z(t)\|^{4}_{\V}+\|\v^n(t)+\z(t)\|^{2r}_{\V}+\|\z(t)\|^{2}_{\H}+\|\f\|^{2}_{\H}\right]\d t.
		\end{align*}
		It implies that $\frac{\d \v^n}{\d t}\in \mathrm{L}^{2}(0,T;\H)$. From the fact that $ \v^n\in \mathrm{L}^{\infty}(0,T;\V)\cap \mathrm{L}^{2}(0,T;\D(\A))$, $\frac{\d \v^n}{\d t}\in \mathrm{L}^{2}(0,T;\H)$, and an application of the \emph{Banach Alaoglu theorem} yield the existence of  $\widehat{\v}\in\mathrm{L}^{\infty}(0,T;\V)\cap\mathrm{L}^{2}(0,T;\D(\A))$ and $\frac{\d \widehat{\v}}{\d t}\in \mathrm{L}^{2}(0,T;\H)$ such that
		\begin{align*}
			\v^n\xrightharpoonup{w^*}&\ \widehat{\v}\text{ in }	\mathrm{L}^{\infty}(0,T;\V),\\
			\v^n\xrightharpoonup{w}&\ \widehat{\v}\text{ in } \mathrm{L}^{2}(0,T;\D(\A)),\\
			\frac{\d \v^n}{\d t}\xrightharpoonup{w}&\frac{\d \widehat{\v}}{\d t} \text{ in }\mathrm{L}^{2}(0,T;\H).
		\end{align*}
		In view of the uniqueness of weak limit, we have $\widehat{\v}=\v$. Moreover, $ \v\in \mathrm{L}^{2}(0,T;\D(\A))$ and $\frac{\d \v}{\d t}\in \mathrm{L}^{2}(0,T;\H)$ implies that $\v\in  \mathrm{C}([0,T]; \V)$, which completes the proof.
	\end{proof}

	\begin{lemma}[\cite{KM}]\label{RDS_Conti1}
		Assume that, for $r\in[1,3]$ and for some $T >0$ fixed, $\x^n \to \x$ in $\H$, 
		\begin{align*}
			\z_n \to \z\ \text{ in }\ \mathrm{L}^{\infty} (0, T; \V)\cap \mathrm{L}^2(0, T;\D(\A)),\ \ \f_n \to \f \ \text{ in }\  \V'.
		\end{align*}
		Let us denote by $\v(t, \z)\x$, the solution of the problem \eqref{cscbf} and by $\v(t, \z_n)\x^n$, the solution of the problem \eqref{cscbf} with $\z, \f, \x$ being replaced by $\z_n, \f_n, \x^n$, respectively. Then $$\v(\cdot, \z_n)\x^n \to \v(\cdot, \z)\x \ \text{ in } \mathrm{L}^2 (0, T;\V).$$
		In particular, $\v(T, \z_n)\x^n \to \v(T, \z)\x$ in $\H$.
	\end{lemma}	
	\begin{proof}
		See the proof of Theorem 4.8, \cite{KM}.
	\end{proof}
	\begin{lemma}[\cite{KM1}]\label{RDS_Conti}
		Assume that, for $r\in[1,3]$ and for some $T >0$ fixed, $\x^n \to \x$ in $\V$, 
		\begin{align*}
			\z_n \to \z\ \text{ in }\ \mathrm{L}^{\infty} (0, T; \V)\cap \mathrm{L}^2(0, T;\D(\A)),\ \ \f_n \to \f \ \text{ in }\ \mathrm{L}^2 (0, T; \H).
		\end{align*}
		Let us denote by $\v(t, \z)\x$, the solution of the problem \eqref{cscbf} and by $\v(t, \z_n)\x^n$, the solution of the problem \eqref{cscbf} with $\z, \f, \x$ being replaced by $\z_n, \f_n, \x^n$, respectively. Then $$\v(\cdot, \z_n)\x^n \to \v(\cdot, \z)\x \ \text{ in } \ \mathrm{C}([0,T];\V)\cap\mathrm{L}^2 (0, T;\D(\A)).$$
		In particular, $\v(T, \z_n)\x^n \to \v(T, \z)\x$ in $\V$.
	\end{lemma}	
	\begin{proof}
		See the proof of Theorem 3.8, \cite{KM1}.
	\end{proof}
	\begin{definition}
		We define a map $\varphi^{\eta} : \mathbb{R}^+ \times \Omega \times \V \to \V$ by
		\begin{align}
			(t, \omega, \x) \mapsto \v^{\eta}(t)  + \z_{\eta}(\omega)(t) \in \V,
		\end{align}
		where $\v^{\eta}(t) = \v(t, \z_{\eta}(\omega)(t))(\x - \z_{\eta}(\omega)(0))$ is a solution to the problem \eqref{cscbf} with the initial condition $\x - \z_{\eta}(\omega)(0).$
	\end{definition}
	\begin{lemma}[\cite{KM}.]\label{alpha_ind}
		If $\eta_1, \eta_2 \geq 0$, then $\varphi^{\eta_1} = \varphi^{\eta_2}.$ 
	\end{lemma}
	Using Lemma \ref{alpha_ind}, we denote $\varphi^{\eta}$  by $\varphi$.
	\begin{lemma}[\cite{KM1}]
		$(\varphi, \theta)$ is a random dynamical system.
	\end{lemma}
	Let us now define, for $\x \in \V,\ \omega \in \Omega,$ and $t\geq s,$
	\begin{align}\label{combine_sol}
		\u(t, s;\omega, \x) := \varphi(t-s; \theta_s \omega)\x = \v\big(t, s; \omega, \x - \z(s)\big) + \z(t),
	\end{align}
	then for each $s\in \mathbb{R}$ and each $\x \in \V,$ the process $\u(t), \ t\geq s,$ is a solution to the problem \eqref{S-CBF}.

	\section{Random attractors for 2D SCBF equations in $\V$}\label{sec4}\setcounter{equation}{0}
	The existence of random attractors in $\V$ for the 2D SCBF equations \eqref{S-CBF} is established in this section. We consider the RDS $\varphi$ over the metric DS $(\Omega, \hat{\mathscr{F}}, \hat{\mathbb{P}}, \hat{\theta})$.
	\begin{lemma}\label{weak_topo1}
		Let	$\v(t, \v_0)$ be the unique solution to the initial value problem \eqref{cscbf} with initial condition  $\v_0 \in \H $, and with a deterministic function $\z \in \mathrm{L}^{\infty}_{loc}(\R^+; \V)\cap \mathrm{L}^2_{loc} (\R^+; \D(\A))$. For $r\in[1,3]$ and for $T>0,$ if $\boldsymbol{y}_n\to \boldsymbol{y}$ in $\V$ weakly, then 
		\begin{align}
			\v(t, \boldsymbol{y}_n) \xrightharpoonup{w} \v(t, \boldsymbol{y})\  \text{ in } \ &\V,\label{Weak1}\\
			\v(\cdot, \boldsymbol{y}_n) \xrightharpoonup{w}  \v(\cdot, \boldsymbol{y})\  \text{ in } \ &\mathrm{L}^2 (0, T; \D(\A) ).\label{Weak2}
		\end{align}
	\end{lemma}
	\begin{proof}
		Assume that $\{\boldsymbol{y}_n\}_{n\in \N}$ is an $\V$-valued sequence such that $\boldsymbol{y}_n$ converges to $\boldsymbol{y}\in \V$ weakly. Let $\v_n(t)= \v(t, \boldsymbol{y}_n)$ and $\v(t)=\v(t, \boldsymbol{y})$. Let us first prove \eqref{Weak2}. Since  $\{\boldsymbol{y}_n\}_n$ is a bounded sequence in $\V$, we infer that 
		\begin{align}\label{Bounded}
			\text{ the sequence } \ \{\v_n\}_{n\in \N} \ \text{ is bounded in }\  \mathrm{L}^{\infty}(0, T; \V)\cap\mathrm{L}^{2}(0, T; \D(\A)).
		\end{align}
		Hence, by the Banach-Alaoglu theorem, there exists a subsequence $\{\v_{n'}\}_{n'\in \N}$ of $\{\v_n\}_{n\in \N}$ and  $\widetilde{\v}\in \mathrm{L}^{\infty}(0, T; \H)\cap\mathrm{L}^{2}(0, T; \V)$, such that, as $n' \to \infty,$
		\begin{align}
			\v_{n'} \xrightharpoonup{w^*} \widetilde{\v}  &\ \text{ in } \ \mathrm{L}^{\infty}(0, T; \V),\label{lim1}\\ \v_{n'} \xrightharpoonup{w} \widetilde{\v} &\ \text{ in }\ \mathrm{L}^{2}(0, T; \D(\A)),\label{lim2}
		\end{align}
		as $n' \to \infty$. For $\boldsymbol{\psi}\in\mathrm{L}^{2}(0,T;\H)$, we consider  
		\begin{align*}
			&	\int_0^T\bigg(\frac{\d\v_{n}(t)}{\d t},\boldsymbol{\psi}(t)\bigg)\d t\\&\leq \mu\int_0^T\|\A\v_n(t)\|_{\H}\|\boldsymbol{\psi}(t)\|_{\H}\d t+\int_0^T\|\v_n (t)+ \z(t)\|^2_{\wi \L^4}\|\boldsymbol{\psi}(t)\|_{\H}\d t + \alpha\int_{0}^{T}\|\v_n(t)\|_{\H}\|\boldsymbol{\psi}(t)\|_{\H}\d t\nonumber\\&+\beta\int_0^T\|\v_n (t)+ \z(t)\|^r_{\wi \L^{2r}}\|\boldsymbol{\psi}(t)\|_{\H}\d t+\left|\eta-\alpha\right|\int_0^T\|\z(t)\|_{\H}\|\boldsymbol{\psi}(t)\|_{\H}\d t+\int_0^T\|\f\|_{\H}\|\boldsymbol{\psi}(t)\|_{\H}\d t\nonumber\\&\leq C\bigg\{\left(\int_0^T\|\A\v_n(t)\|_{\H}^2\d t\right)^{1/2}+\left(\int_0^T\|\v_n (t)+ \z(t)\|_{\V}^4\d t\right)^{1/2}+\left(\int_0^T\|\z(t)\|_{\V}^2\d t\right)^{1/2}\nonumber\\&+\left(\int_0^T\|\v_n(t)\|_{\V}^2\d t\right)^{1/2}+T^{\frac{1}{2}}\|\f\|_{\H}+\left(\int_0^T\|\v_n(t)+\z(t)\|_{\V}^{2r}\d t\right)^{\frac{1}{2}}\bigg\}\left(\int_0^T\|\boldsymbol{\psi}(t)\|_{\H}^{2}\d t\right)^{\frac{1}{2}},
		\end{align*}
		and thus, we obtain $\big\|\frac{\d\v_{n}}{\d t}\big\|_{\mathrm{L}^{2}(0, T; \H)}\leq C,$ for some $C>0$ independent of $n$. Hence by the Cauchy-Schwartz inequality, for all $0\leq t \leq t+a \leq T$, $\boldsymbol{\phi}\in\D(\A)$ and $n\in \N,$ we have 
		\begin{align}\label{V'}
			|(\!(\v_n(t+a)-\v_n(t), \boldsymbol{\phi})\!)|\leq \int_{t}^{t+a}\bigg|\bigg(\frac{\d\v_{n}(s)}{\d t} , \A\boldsymbol{\phi} \bigg)\bigg|\d s\leq C \|\boldsymbol{\phi}\|_{\D(\A)} \sqrt{a}.
		\end{align}
		For $\boldsymbol{\phi} =\v_n(t+a)-\v_n(t)$ and for a.e. $t\in (0, T)$,  from \eqref{V'} we get
		$$\|\v_n(t+a)-\v_n(t)\|^2_{\V}\leq C\sqrt{a}\|\v_n(t+a)-\v_n(t)\|_{\D(\A)}.$$
		Integrating from $0$ to $T-a$ and, using the Cauchy-Schwarz inequality and \eqref{Bounded}, we further deduce that
		\begin{align}
			\int_{0}^{T-a}\|\v_n(t+a)-\v_n(t)\|^2_{\V}\d t&\leq C\sqrt{a}\int_{0}^{T-a}\|\v_n(t+a)-\v_n(t)\|_{\D(\A)}\d t\nonumber\\&\leq C\sqrt{a}(T-a)^{1/2}\left(\int_{0}^{T-a}\|\v_n(t+a)-\v_n(t)\|^2_{\D(\A)}\d t\right)^{1/2}\nonumber\\&\leq C\sqrt{a}.
		\end{align}Furthermore, we have 
		\begin{align}
			\lim_{a\to 0}\sup_{n} \int_{0}^{T-a}\|\v_n(t+a)-\v_n(t)\|^2_{\H^1(\mathcal{O}_{R})}\d t=0,
		\end{align}
		for all $R>0$, where $\mathcal{O}_{R}= \mathcal{O}\cap\{x\in \R^2:|x|\leq R\}.$ 
		Due to compact embedding $\H^2(\mathcal{O}_{R})\hookrightarrow\H^1_0(\mathcal{O}_{R})$, we have (see Theorem  13.3, \cite{Temam1}) $$\v_{n'} \to \widetilde{\v}\  \text{ strongly in } \ \mathrm{L}^2(0,T; \H_0^1(\mathcal{O}_{R})).$$ 
		The convergences given above can be used to pass limit in the equation \eqref{cscbf} for $\v_{n}$ and deduce that $\widetilde{\v}$ is a solution of \eqref{cscbf} with initial data $\y$ at time $0$. Since the system \eqref{cscbf} has a unique solution, we infer that $\widetilde{\v}=\v.$ By using a standard contradiction argument, we conclude that the whole sequence $$\v_n\to \v\ \text{ in }\  \mathrm{L}^{2}(0, T; \D(\A))\text{ weakly},$$ and hence the convergence given in \eqref{Weak2} is established.
		
		Let us now prove the convergence \eqref{Weak1}. Take any function $\boldsymbol{\phi} \in \mathcal{V}.$ Then, by \eqref{lim1}, for a.e. $t\in[0, T],$ we have  $$(\!(\v_n(t), \boldsymbol{\phi})\!) \to (\!(\v(t), \boldsymbol{\phi})\!).$$ Since $\{\v_n\}_{n\in\N}$ is a bounded sequence in $\mathrm{L}^{\infty}(0,T; \V), \{(\!(\v_n(\cdot), \boldsymbol{\phi})\!)\}_{n\in \N}$ is uniformly bounded on $[0, T].$ Also, the estimate \eqref{V'} shows that the sequence $\{(\!(\v_n(\cdot), \boldsymbol{\phi})\!)\}_{n\in \N}$ is uniformly equicontinuous on $[0, T].$ Hence, there exists a subsequence $\{(\!(\v_{n'}(\cdot), \boldsymbol{\phi})\!)\}_{n'\in \N}$ of $\{(\!(\v_n(\cdot), \boldsymbol{\phi})\!)\}_{n\in \N}$ (by the Arzela-Ascoli theorem), such that $$(\!(\v_{n'}(\cdot), \boldsymbol{\phi})\!)\to (\!(\v(\cdot), \boldsymbol{\phi})\!) \text{ uniformly on } [0, T].$$ Again, using the standard contradiction argument, we assert that 
		$$
		(\!(\v_{n}(\cdot), \boldsymbol{\phi})\!) \to (\!(\v(\cdot), \boldsymbol{\phi})\!) \text{ uniformly on  } [0, T].
		$$
		Using the density of $\mathcal{V}$ in $\V$ and $\sup\limits_{t  \in [0, T]} \|\v_n(t)\|_{\V}< \infty$,  for any $\boldsymbol{\phi}\in \V,$
		\begin{align*}
			(\!(\v_{n}(\cdot), \boldsymbol{\phi})\!) \to (\!(\v(\cdot), \boldsymbol{\phi})\!) \text{ uniformly on  } [0, T],
		\end{align*}
		which completes the proof.
	\end{proof}
	
	\begin{lemma}[Lemma 4.1 \cite{KM1}]\label{Bddns4}
		For each $\omega\in \Omega,$ 
		\begin{align*}
			\limsup_{t\to - \infty} \|\z(\omega)(t)\|^2_{\H}\  e^{\alpha t +\frac{4}{ \mu} \int_{t}^{0}\|\z(\zeta)\|^{2}_{\V}\d\zeta} = 0.
		\end{align*}
	\end{lemma}
	\begin{lemma}[Lemma 4.2 \cite{KM1}]\label{Bddns5}
		For each $\omega\in \Omega,$
		\begin{align*}
			\int_{- \infty}^{0} \bigg\{ 1 + \|\z(t)\|^{2}_{\V} + \|\z(s)\|^{r+1}_{\V} + \|\z(t)\|^4_{\V}  \bigg\}e^{\alpha t +\frac{4}{\mu} \int_{t}^{0}\|\z(\zeta)\|^{2}_{\V}\d\zeta} \d t < \infty.
		\end{align*}
	\end{lemma}
	\begin{definition}\label{RA2}
		A function $\kappa: \Omega\to (0, \infty)$ belongs to class $\mathfrak{K}$ if and only if 
		\begin{align}
			\limsup_{t\to \infty} [\kappa(\theta_{-t}\omega)]^2 e^{-\alpha t +\frac{4}{\mu} \int_{-t}^{0}\|\z(\omega)(s)\|^{2}_{\V}\d s} = 0,
		\end{align}
		where $\alpha$ is Darcy's constant.
		
		We denote by $\hat{\mathfrak{DK}},$ the class of all closed and bounded random sets $\D$ on $\V$ such that the radius function $\Omega\ni \omega \mapsto \kappa(\D(\omega)):= \sup\{\|x\|_{\V}:x\in \D(\omega)\}$ belongs to the class $\mathfrak{K}.$
	\end{definition}
	By Corollary \ref{Bddns1}, we infer that the constant functions belong to $\mathfrak{K}$. 
	The class $\mathfrak{K}$ is closed with respect to sum, multiplication by a constant and if $\kappa \in \mathfrak{K}, 0\leq \bar{\kappa} \leq \kappa,$ then $\bar{\kappa}\in \mathfrak{K}.$
	\begin{lemma}[Proposition 4.4 \cite{KM1}]\label{radius}
		Define functions $\kappa_{i}:\Omega\to (0, \infty), \ i= 1, 2, 3, 4, 5, 6,$ by the following formulae, for $\omega\in\Omega,$
		\begin{align*}
			[\kappa_1(\omega)]^2 &:= \|\z(\omega)(0)\|_{\V},\ \ 
			[\kappa_2(\omega)]^2 := \sup_{s\leq 0} \|\z(\omega)(s)\|^2_{\H}\  e^{\alpha s +\frac{4}{ \mu} \int_{s}^{0}\|\z(\omega)(\zeta)\|^{2}_{\V}\d\zeta}, \\
			[\kappa_3(\omega)]^2 &:= \int_{- \infty}^{0} \|\z(\omega)(t)\|^{2}_{\V}\ e^{\alpha t +\frac{4}{ \mu} \int_{t}^{0}\|\z(\omega)(\zeta)\|^{2}_{\V}\d\zeta} \d t, \\
			[\kappa_4(\omega)]^2 &:= \int_{- \infty}^{0} \|\z(\omega)(t)\|^{r+1}_{\V}\ e^{\alpha t +\frac{4}{ \mu} \int_{t}^{0}\|\z(\omega)(\zeta)\|^{2}_{\V}\d\zeta} \d t,\\
			[\kappa_5(\omega)]^2 &:= \int_{- \infty}^{0} \|\z(\omega)(t)\|^4_{\V}\ e^{\alpha t +\frac{4}{ \mu} \int_{t}^{0}\|\z(\omega)(\zeta)\|^{2}_{\V}\d\zeta} \d t,\ \ 
			[\kappa_6(\omega)]^2 := \int_{- \infty}^{0} e^{\alpha t +\frac{4}{ \mu} \int_{t}^{0}\|\z(\omega)(\zeta)\|^{2}_{\V}\d\zeta} \d t.
		\end{align*}
		Then all these functions belongs to class $\mathfrak{K}.$
	\end{lemma}
	\begin{theorem}\label{V_ab}
		Assume that, for $r\in[1,3],$ $\x\in\V,\ \f\in\H$ and Assumption \ref{assump} holds. Then there exists a closed and bounded $\hat{\mathfrak{DK}}$-random absorbing sets in $\V$ corresponding to the RDS $\varphi.$
	\end{theorem}
	\begin{proof}
		Let ${\mathrm{D}}(\omega)\in \hat{\mathfrak{DK}}$. Let $\kappa_{\mathrm{D}}(\omega)$ be the radius of ${\mathrm{D}}(\omega)$, that is, $\kappa_{\mathrm{D}}(\omega):= \sup\{\|x\|_{\V} : x \in {\mathrm{D}}(\omega)\},\ \omega\in \Omega.$	Let $\omega\in \Omega$ be fixed. For given $s\leq 0$ and $\x\in \V$, let $\v(\cdot)$ be the unique solution of \eqref{cscbf} on time interval $[s, \infty)$ with the initial condition $\v(s)= \x-\z(s).$
		Multiplying the first equation of \eqref{cscbf} by $\v(\cdot)$ and integrating the resulting equation over $\mathcal{O}$, we obtain	
		\begin{align}\label{H_ab1}
			\frac{1}{2}\frac{\d}{\d t}\|\v(t)\|^2_{\H} = & -\mu\|\v(t)\|^2_{\V}-\alpha\|\v(t)\|^2_{\H} -b(\v(t)+\z(t),\v(t)+\z(t),\v(t))\nonumber\\&-\beta(\mathcal{C}\big(\v(t)+\z(t)\big),\v(t)) +(\eta-\alpha)\big(\z(t), \v(t)\big) +\big(\f,\v(t)\big)\nonumber\\
			=&-\mu\|\v(t)\|^2_{\V} -\alpha\|\v(t)\|^2_{\H}-\beta\|\v(t)+\z(t)\|^{r+1}_{\widetilde{\L}^{r+1}}- b(\v(t),\z(t),\v(t))\nonumber\\&-b(\z(t),\z(t),\v(t))+\beta(\mathcal{C}\big(\v(t)+\z(t)\big),\z(t)) +(\eta-\alpha)\big(\z(t), \v(t)\big) \nonumber\\&+\big(\f,\v(t)\big) .
		\end{align}
		Using  H\"older's inequality, Young's inequality, Sobolev's embedding and \eqref{lady}, we have
		\begin{align*}
			| b(\v,\z, \v)|&\leq \|\v\|^2_{\widetilde{\L}^{4}}\|\z\|_{\V}\leq\sqrt{2}\|\v\|_{\H}\|\v\|_{\V}\|\z\|_{\V}\leq \frac{\mu}{4} \|\v\|^2_{\V}+ \frac{2}{\mu}\|\z\|^2_{\V}\|\v\|^2_{\H} ,\\
			|b(\z,\z, \v)|&\leq \|\z\|^2_{\widetilde{\L}^{4}}\|\v\|_{\V}\leq\sqrt{2}\|\z\|_{\H}\|\z\|_{\V}\|\v\|_{\V}\leq \frac{\mu}{4} \|\v\|^2_{\V}+ \frac{2}{\mu}\|\z\|^2_{\H}\|\z\|^2_{\V},\\
			|\beta\big\langle\mathcal{C}(\v+\z),\z\big\rangle|& \leq \beta \|\v+\z\|^{r}_{\widetilde{\L}^{r+1}} \|\z\|_{\widetilde{\L}^{r+1}} \leq \frac{\beta}{2} \|\v+\z\|^{r+1}_{\widetilde{\L}^{r+1}} + \frac{\beta(2r)^r}{(r+1)^{r+1}}\|\z\|^{r+1}_{\widetilde{\L}^{r+1}}\\
			& \leq \frac{\beta}{2} \|\v+\z\|^{r+1}_{\widetilde{\L}^{r+1}} + C\|\z\|^{r+1}_{\V},\\
			|(\eta-\alpha)\big( \z, \v\big)+\big(\f, \v\big) |  &\leq \big[\left|\eta-\alpha\right|\|\z\|_{\H} +\|\f\|_{\H}\big] \|\v\|_{\H}\nonumber\\&\leq \frac{\alpha}{2} \|\v\|^2_{\H} + \frac{(\eta-\alpha)^2}{\alpha} \|\z\|^2_{\H}+\frac{1}{\alpha} \|\f\|^2_{\H}.
		\end{align*}
		Thus from \eqref{H_ab1}, we deduce that
		\begin{align}\label{H_ab2}
			&\frac{\d}{\d t} \|\v(t)\|^2_{\H}  + \alpha \|\v(t)\|^2_{\H}+\mu\|\v(t)\|^2_{\V}+\beta\|\v(t)+\z(t)\|^{r+1}_{\wi\L^{r+1}} \nonumber\\& \leq  \frac{4}{\mu} \|\v(t)\|^2_{\H}\ \|\z(t)\|^{2}_{\V} +C\|\z(t)\|^4_{\V} + C\|\z(t)\|^{r+1}_{\V}+ C \|\z(t)\|^2_{\V} + C \|\f\|^2_{\H}.
		\end{align}
		We infer from  the classical Gronwall inequality that 
		\begin{align}\label{H_ab3}
			\|\v(0)\|^2_{\H} &\leq \|\v(s)\|^2_{\H} e^{\alpha s + \frac{4}{\mu}\int_{s}^{0}\|\z(\zeta)\|^2_{\V}\d\zeta} +C\int_{s}^{0}e^{\alpha t  + \frac{4}{\mu}\int_{t}^{0}\|\z(\zeta)\|^2_{\V}\d\zeta}\biggl\{\|\z(t)\|^4_{\V}  + \|\z(t)\|^{r+1}_{\V}+   \|\z(t)\|^2_{\V} \nonumber\\& \qquad +  \|\f\|^2_{\H}\biggr\}\d t\nonumber\\
			&\leq2 \|\x\|^2_{\H} e^{\alpha s + \frac{4}{\mu}\int_{s}^{0}\|\z(\zeta)\|^2_{\V}\d\zeta} +2\|\z(s)\|^2_{\H} e^{\alpha s + \frac{4}{\mu}\int_{s}^{0}\|\z(\zeta)\|^2_{\V}\d\zeta} +C\int_{s}^{0}e^{\alpha t  + \frac{4}{\mu}\int_{t}^{0}\|\z(\zeta)\|^2_{\V}\d\zeta}\nonumber \\& \quad\times\bigg\{\|\z(t)\|^4_{\V}  + \|\z(t)\|^{r+1}_{\V}+   \|\z(t)\|^2_{\V}  +  \|\f\|^2_{\H}\biggr\}\d t.
		\end{align}
		For $\omega\in \Omega,$	let us set 
		\begin{align}
			[\kappa_{11}(\omega)]^2  &=2 +  2\sup_{s\leq 0}\bigg\{ \|\z(s)\|^2_{\H}\  e^{\alpha s +\frac{4}{ \mu} \int_{s}^{0}\|\z(\zeta)\|^{2}_{\V}\d\zeta}\bigg\} + C\int_{- \infty}^{0} e^{\alpha t +\frac{4}{ \mu} \int_{t}^{0}\|\z(\zeta)\|^{2}_{\V}\d\zeta}\bigg\{\|\z(t)\|^4_{\V}  \nonumber\\& \quad + \|\z(t)\|^{r+1}_{\V}+ \|\z(t)\|^2_{\V}  +  \|\f\|^2_{\H}\bigg\} \d t,\\	
			\kappa_{12}(\omega)&=  \|\z(\omega)(0)\|_{\V}.
		\end{align}
		By Lemmas \ref{Bddns5} and \ref{radius}, it is immediate that both $\kappa_{11}$ and $\kappa_{12}$ belong to the class $\mathfrak{K}$.
		Let $\omega\in\Omega$ be fixed. Since $\kappa_{\mathrm{D}}(\omega)\in \mathfrak{K}$, there exists $t_{\mathrm{D}}(\omega)\geq 0$ such that 
		\begin{align*}
			[\kappa_{\mathrm{D}}(\theta_{-t}\omega)]^2 e^{-\alpha t +\frac{4}{ \mu} \int_{-t}^{0}\|\z(\omega)(s)\|^{2}_{\V}\d s} \leq 1 \  \text{ for }\  t\geq t_{\mathrm{D}}(\omega). 
		\end{align*}
		Thus, if $\x\in {\mathrm{D}}(\theta_{-t}\omega)$ and $s\leq- t_{\mathrm{D}}(\omega),$ then, by \eqref{H_ab3}, we get
		\begin{align}\label{ab_H}
			\|\v(0,s;\omega, \x-\z(s))\|_{\H}\leq \kappa_{11}(\omega).	
		\end{align}
		Using a similar argument, we observe that for all $t\in[-1,0]$ and for any $\omega\in \Omega$, there exists a $\kappa_{13}(\omega)\geq0$ such that 
		\begin{align}\label{H_ab7}
			\|\v(t,s;\omega, \x-\z(\omega)(s))\|_{\H} \leq \kappa_{13}(\omega),
		\end{align}
		for all $s\leq -(t_{\mathrm{D}}(\omega)+1).$
		Furthermore, integrating \eqref{H_ab2} over $(-1,0)$, we find that for any $\omega\in \Omega$, there exists a $\kappa_{14}(\omega)\geq 0$ such that
		\begin{align}\label{H_ab4}
			\int_{-1}^{0} \|\v(t)\|^2_{\V} \d t+ \int_{-1}^{0} \|\v(t)\|^{r+1}_{\wi\L^{r+1}} \d t \leq \kappa_{14}(\omega),
		\end{align}
		for all $s\leq -(t_{\mathrm{D}}(\omega)+1).$ Taking the inner product with $\A\v(\cdot)$ to the first equation of \eqref{cscbf}, similar calculations as we have performed for \eqref{S22}, provide us
		\begin{align}\label{V_ab8}
			&\frac{\d}{\d t}\|\v(t)\|^2_{\V} +\|\A\v(t)\|^2_{\H}\leq\begin{cases}
				C [Q_1(t)\|\v(t)\|^2_{\V} + \widetilde{Q}_1(t)],  \ \text{ for } r\in[1,3),\\
				C[Q_2(t)\|\v(t)\|^2_{\V} +\widetilde{Q}_2(t)],  \ \text{ for } r=3,\\
			\end{cases} 
		\end{align}
		where
		\begin{align*}
			Q_1(t)&=\|\v(t)\|_{\H}^2\|\v(t)\|^2_{\V} + \|\z(t)\|_{\H}\|\A\z(t)\|_{\H},\ \\
			Q_2(t)&=	Q_1(t)+\|\v(t)\|^{2}_{\H}\|\v(t)\|^2_{\V},\\
			\widetilde{Q}_1(t)&=  \|\v(t)\|_{\H}^2\|\z(t)\|^4_{\V}+ \|\A\z(t)\|_{\H}\|\z(t)\|^3_{\V} + \|\z(t)\|^{2r}_{\V} +\|\z(t)\|^2_{\V}+\|\f\|^2_{\H}+\|\v\|^{2}_{\H},\\
			\widetilde{Q}_2(t)&=  \|\v(t)\|_{\H}^2\|\z(t)\|^4_{\V}+ \|\A\z(t)\|_{\H}\|\z(t)\|^3_{\V} + \|\z(t)\|^{6}_{\V} +\|\z(t)\|^2_{\V}+\|\f\|^2_{\H}.
		\end{align*}
		Applying Gronwall's inequality on the time interval $[-1,0]$ in \eqref{V_ab8}, we get for $i\in\{1,2\}$
		\begin{align}\label{V_ab11}
			\|\v(0)\|^2_{\V}\leq \|\v(t)\|^2_{\V} e^{\int_{t}^{0}Q_i(s)\d s}+\int_{t}^{0}\widetilde{Q}_i(s)e^{\int_{t}^{0}Q_i(\tau)\d \tau}\d s, \quad t\in[-1,0].
		\end{align}
		Integrating \eqref{V_ab11} from $-1$ to $0$, we obtain for $i\in\{1,2\}$
		\begin{align}\label{V_ab12}
			\|\v(0)\|^2_{\V}&\leq \int_{-1}^{0}\|\v(t)\|^2_{\V} e^{\int_{t}^{0}Q_i(s)\d s}\d t+\int_{-1}^{0}\int_{t}^{0}\widetilde{Q}_i(s)e^{\int_{t}^{0}Q_i(\tau)\d \tau}\d s\d t\nonumber\\
			&\leq e^{\int_{-1}^{0}Q_i(s)\d s}\int_{-1}^{0}\|\v(t)\|^2_{\V} \d t +e^{\int_{-1}^{0}Q_i(s)\d s}\int_{-1}^{0}\widetilde{Q}_i(t)\d t.
		\end{align}
		In view of \eqref{O-U_conti}, \eqref{H_ab7} and \eqref{H_ab4}, we infer from \eqref{V_ab12} that for any $r\in[1,3]$ and for any $\omega\in\Omega,$ there exists $\kappa_{15}(\omega)\geq0$ such that 
		\begin{align}\label{ab_V}
			\|\v(0,\omega; s, \x-\z(s))\|_{\V}\leq \kappa_{15}(\omega),
		\end{align}
		for any $s\leq-(t_{\mathrm{D}}(\omega)+1)$. 	Moreover, we have 
		\begin{align}\label{V_ab13}
			\|\u(0,s;\omega, \x)\|_{\V} \leq \|\v(0,s;\omega, \x-\z(s))\|_{\V} + \|\z(\omega)(0)\|_{\V}\leq \kappa_{15}(\omega)+\kappa_{12}(\omega):=\kappa_{16}(\omega).
		\end{align}
		This implies that the ball $B_{\V}(0, \kappa_{16}(\omega))\in \hat{\mathfrak{DK}}$, that is, the ball in $\V$ centered at the origin and of radius $\kappa_{16}$, absorbs ${\mathrm{D}}(\omega)$.
	\end{proof}	
	\begin{theorem}\label{V_asymptotically}
		Suppose that $r\in[1,3]$, the domain $\mathcal{O}\subset\R^2$ satisfies Assumption \ref{assumpO},  Assumption \ref{assump} is satisfied and $\f\in \H$, then the RDS $\varphi$ generated by \eqref{S-CBF} on $\V$ is $\V$-asymptotically compact.  
	\end{theorem}
	\begin{proof}
		Assume that $\mathrm{D}(\omega) \in \hat{\mathfrak{DK}}$ and $\textbf{B}(\omega)\in \hat{\mathfrak{DK}}$ is such that $\textbf{B}(\omega)$ absorbs $\mathrm{D}(\omega)$. Let us fix $\omega\in \Omega$ and take a sequence of positive numbers $\{t_n\}^{\infty}_{n=1}$ such that $t_1\leq t_2 \leq t_3 \leq \cdots$ and $t_n \to \infty$. We take a $\V$-valued sequence $\{\boldsymbol{x}_n\}^{\infty}_{n=1}$ such that $\boldsymbol{x}_n \in \mathrm{D}(\theta_{-t}\omega),$ for all $n\in \mathbb{N}.$
		\vskip 0.2 cm 
		\noindent 
		\textbf{Step I.} \textit{Reduction.} Since $\textbf{B}(\omega)$ absorbs $\mathrm{D}(\omega)$, we obtain  \begin{align}\label{619}\varphi(t_n, \theta_{-t_n}\omega, \mathrm{D}(\theta_{-t_n}\omega))\subset \textbf{B}(\omega),\end{align} 
		for sufficient large $n\in \mathbb{N}.$ Since $\textbf{B}(\omega) \subset \V$, is a bounded set, which implies that $\textbf{B}(\omega)$ is weakly pre-compact in $\V$, without loss of generality, we may assume that \eqref{619} holds for all $n\in \mathbb{N}$ and, for some $\y_0\in\V,$ 
		\begin{align}\label{weak_lim1}
			\varphi(t_n, \theta_{-t_n}\omega, \boldsymbol{x}_n)\to \y_0 \ \text{ in } \ \V\ \text{ weakly}.
		\end{align}
		Since $\z(0)\in \V,$ in order to complete the proof, we only need to show that for some subsequence $\{n'\}\subset \mathbb{N}$ (for more details, see Theorem 5.9 of \cite{KM})
		\begin{align}\label{weak_lim4}
			\|\y_0-\z(0)\|_{\V} \geq \limsup_{n'\to \infty} \|\varphi(t_{n'}, \theta_{-t_{n'}}\omega, \boldsymbol{x}_{n'})-\z(0)\|_{\V},
		\end{align}
		as the norm is weakly lower semicontinuous. 
		\vskip 0.2 cm
		\noindent
		\textbf{Step II.} \textit{Construction of a negative trajectory.}  That is, the construction of a sequence $\{\y_n\}^0_{n=-\infty}$ such that $\y_n\in \boldsymbol{B}(\theta_n\omega), n\in \mathbb{Z}^{-},$ and 
		$$\y_k = \varphi(k-n, \theta_n\omega, \y_n), \ \ \ n<k\leq 0.$$
		Since $\boldsymbol{B}(\omega)$ is bounded and closed in $\V$, it is also weakly compact in $\V$. Therefore, analogous to the proof of Proposition 5.9 in \cite{KM}, for each $k = 1, 2, \cdots,$ we can construct a subsequence $\{n^{(k)}\}\subset \{n^{(k-1)}\}$ and $\y_{-k}\in \textbf{B}(\theta_{-k}\omega),$ such that $\varphi(1, \theta_{-k}\omega, \y_{-k})= \y_{-k+1}$ and 
		\begin{align}\label{weak_lim6}
			\varphi(-k + t_{n^{(k)}}, \theta_{-t_{n^{(k)}}}\omega, \boldsymbol{x}_{n^{(k)}})\xrightharpoonup{w} \y_{-k} \  \text{ in } \ \V  \  \text{ as }\   n^{(k)} \to \infty.
		\end{align}
		Now, using the cocycle property of RDS $\varphi$, with $t=k,\ s=-k +t_{n^{(k)}}$ and $\omega$ being replaced by $\theta_{-t_{n^{(k)}}}\omega,$ we get 
		\begin{align}\label{weak_lim7}
			\varphi(t_{n^{(k)}}, \theta_{-t_{n^{(k)}}}\omega) = \varphi(k, \theta_{-k}\omega)\varphi(t_{n^{(k)}}-k, \theta_{-t_{n^{(k)}}}\omega), \ \ \ k\in \mathbb{N}.
		\end{align}
		Hence, using the convergence \eqref{Weak1} of Lemma \ref{weak_topo1} with \eqref{weak_lim6}, we obtain 
		\begin{align}\label{weak_lim8}
			\y_{-j} = & \ \textrm{w}_{\V}\text{-} \lim_{n^{(k)} \to \infty} \varphi (-j+t_{n^{(k)}}, \theta_{-t_{n^{(k)}}}\omega, \boldsymbol{x}_{n^{(k)}})\nonumber\\
			=&\ \textrm{w}_{\V}\text{-} \lim_{n^{(k)} \to \infty}\varphi\big(-j+k, \theta_{-k}\omega, \varphi (t_{n^{(k)}}-k, \theta_{-t_{n^{(k)}}}\omega, \boldsymbol{x}_{n^{(k)}})\big)\nonumber\\
			=&\  \varphi\bigg(-j+k, \theta_{-k}\omega,\big(\textrm{w}_{\V}\text{-} \lim_{n^{(k)} \to \infty}\varphi (t_{n^{(k)}}-k, \theta_{-t_{n^{(k)}}}\omega, \boldsymbol{x}_{n^{(k)}})\big)\bigg)\nonumber\\
			= & \ \varphi (-j+k, \theta_{-k}\omega, \y_{-k}),
		\end{align}
		where $\textrm{w}_{\V}\text{-}\lim$ represents the limit in the weak topology on $\V$ given in \eqref{Weak1} of Lemma \ref{weak_topo1}.
		The same proof provides a more general property: $$\varphi(j, \theta_{-k}\omega, \y_{-k})= \y_{-k+j}, \ \text{ if }\ 0\leq j\leq k.$$	
		More precisely, in \eqref{weak_lim8}, $\y_{-j} = \u(-j, -k; \omega,\y_{-k}),$ where $\u$ is defined by \eqref{combine_sol}.
		
		\vskip 0.2 cm
		\noindent
		\textbf{Step III.} \textit{Proof of \eqref{weak_lim4}.}  From \eqref{S18}, we obtain 
		\begin{align*}
			\frac{\d}{\d t}\|\v(t)\|^2_{\V} +\alpha\|\v(t)\|_{\V}^2 &= 2\bigg\{\big((\eta-\alpha) \z(t)+\f, \A\v(t)\big)-\big(\B(\v(t)+\z(t)), \A\v(t)\big) \\&\quad\quad\quad- \beta\big(\mathcal{C}(\v(t)+\z(t)), \A\v(t)\big)- \mu \|\A\v(t)\|_{\H}^2-\frac{\alpha}{2}\|\v(t)\|_{\V}^2 \bigg\}.
		\end{align*}
		Then, using the variation of constant formula, we have for $t_0\leq \tau\leq t$,
		\begin{equation}\label{Energy_esti2}
			\begin{aligned}
				&\|\v(t)\|^2_{\V} = \|\v(\tau)\|_{\V}^2 e^{-\alpha(t-\tau)} + 2\int_{\tau}^{t}e^{-\alpha(t-s)}\bigg\{((\eta-\alpha) \z(s)+\f, \A\v(s))\\&\quad-(\B(\v(s)+\z(s)), \A\v(s)) -\beta(\mathcal{C}(\v(s)+\z(s)), \A\v(s))- \mu \|\A\v(s)\|_{\H}^2-\frac{\alpha}{2}\|\v(s)\|_{\V}^2 \bigg\}\d s.
			\end{aligned}
		\end{equation}
		From now onward (until explicitly stated), we fix $k\in\mathbb{N},$ and consider the problem \eqref{S-CBF} on the interval $[-k, 0].$ From \eqref{combine_sol} and \eqref{weak_lim7}, with $t=0$ and $s=-k,$ we have
		\begin{align}\label{weak_lim9}
			\|\varphi(t_{n^{(k)}}, \theta_{-t_{n^{(k)}}}\omega, \boldsymbol{x}_{n^{(k)}}) - \z(0)\|^2_{\V} = & \|\varphi\big(k, \theta_{-k}\omega,\varphi(t_{n^{(k)}}-k, \theta_{-t_{n^{(k)}}}\omega, \boldsymbol{x}_{n^{(k)}})\big) - \z(0)\|^2_{\V}\nonumber\\
			=&\|\v\big(0, -k; \omega, \varphi(t_{n^{(k)}}-k, \theta_{-t_{n^{(k)}}}\omega, \boldsymbol{x}_{n^{(k)}}) - \z(-k)\big)\|^2_{\V}.
		\end{align}
		Let $\v$ be the solution to \eqref{cscbf} on $[-k, \infty)$ with $\z= \z_{\eta}(\cdot, \omega)$ and the initial condition at time $-k:$ $$\v(-k) = \varphi(t_{n^{(k)}}-k, \theta_{-t_{n^{(k)}}}\omega, \boldsymbol{x}_{n^{(k)}}) - \z(-k).$$ Also, we can write $$\v(s) = \v\big(s, -k; \omega, \varphi(t_{n^{(k)}}-k, \theta_{-t_{n^{(k)}}}\omega, \boldsymbol{x}_{n^{(k)}}) - \z(-k)\big), \  s\geq -k.$$
		Using \eqref{Energy_esti2} with $t=0$ and $\tau = -k$, we obtain 
		\begin{align}\label{Energy_esti6}
			&\|\varphi(t_{n^{(k)}}, \theta_{-t_{n^{(k)}}}\omega, \boldsymbol{x}_{n^{(k)}}) - \z(0)\|^2_{\V}\nonumber\\ 
			&= e^{- \alpha k} \|\varphi(t_{n^{(k)}}-k, \theta_{-t_{n^{(k)}}}\omega, \boldsymbol{x}_{n^{(k)}}) - \z(-k)\|^2_{\V}  + 2 \int_{-k}^{0} e^{\alpha s}\bigg\{\big((\eta-\alpha) \z(s)+\f, \A\v(s)\big)\nonumber\\&\quad-\big(\B(\v(s)+\z(s)), \A\v(s)\big) - \beta\big(\mathcal{C}(\v(s)+\z(s)), \A\v(s)\big)- \mu \|\A\v(s)\|_{\H}^2-\frac{\alpha}{2}\|\v(s)\|_{\V}^2 \bigg\} \d s\nonumber\\ 
			&= e^{- \alpha k} \|\varphi(t_{n^{(k)}}-k, \theta_{-t_{n^{(k)}}}\omega, \boldsymbol{x}_{n^{(k)}}) - \z(-k)\|^2_{\V}  + 2 \int_{-k}^{0} e^{\alpha s}\bigg\{\big((\eta-\alpha) \z(s)+\f, \A\v(s)\big)\nonumber\\&\quad-b\big(\v(s), \v(s), \A\v(s)\big)-b\big(\v(s), \z(s), \A\v(s)\big)-b\big(\z(s), \v(s), \A\v(s)\big) \nonumber\\&\quad-b\big(\z(s), \z(s), \A\v(s)\big)- \beta\big(\mathcal{C}(\v(s)+\z(s)), \A\v(s)\big)- \mu \|\A\v(s)\|_{\H}^2-\frac{\alpha}{2}\|\v(s)\|_{\V}^2 \bigg\} \d s.
		\end{align}	
		Analogous to the methods used in the proof of Theorem 5.9, \cite{KM}, using \eqref{V_ab8}, we can deduce that there exists a non negative function $h\in \mathrm{L}^1 (-\infty, 0)$ such that
		\begin{align}\label{positive_function}
			\limsup_{n^{(k)} \to \infty}\ e^{- \alpha k} \|\varphi(t_{n^{(k)}}-k, \theta_{-t_{n^{(k)}}}\omega, \boldsymbol{x}_{n^{(k)}}) - \z(-k)\|^2_{\V}\leq \int_{-\infty}^{-k} h(s)\ \d s, \ k\in \N.
		\end{align}
	Let us denote 
		\begin{align*}
			\v^{n^{(k)}}(s) &= \v\big(s, -k; \omega, \varphi(t_{n^{(k)}}-k, \theta_{-t_{n^{(k)}}}\omega)\boldsymbol{x}_{n^{(k)}} - \z(-k)\big), \ s\in (-k, 0),\\
			\v_k(s) &= \v\big(s, -k; \omega, \y_{-k} - \z(-k)\big), \ s\in (-k, 0).
		\end{align*}
		By Lemma \ref{weak_topo1} and the convergence \eqref{weak_lim6}, we conclude that 
		\begin{align}\label{weak_lim10}
			\v^{n^{(k)}}(\cdot) \text{ converges to } \v_k(\cdot) \ \text{ in }  \ \mathrm{L}^2(-k,0;\D(\A)) \text{ weakly}.
		\end{align}
		Note that $\f\in \H$ implies  $s\mapsto e^{\alpha s } \f \in \mathrm{L}^2(-k,0;\H).$ Thus, using \eqref{weak_lim10}, we obtain 
		\begin{align}\label{weak_lim11}
			\lim_{n^{(k)} \to \infty} \int_{-k}^{0} e^{\alpha s} \big( \f, \A\v^{n^{(k)}}(s)\big) \ \d s = \int_{-k}^{0} e^{\alpha s} \big( \f,\A \v_k(s)\big) \ \d s.
		\end{align}
		Now, $\z\in \mathrm{L}^{\infty}(0,T;\V)$ (see \eqref{O-U_conti}) implies that $s\mapsto  e^{\alpha s } \z(s) \in \mathrm{L}^2(-k,0;\H).$ Therefore, using \eqref{weak_lim10}, we obtain
		\begin{align}\label{weak_lim12}
			\lim_{n^{(k)} \to \infty} \int_{-k}^{0} (\eta-\alpha) e^{\alpha s} \big( \z(s), \A\v^{n^{(k)}}(s)\big)  \d s = \int_{-k}^{0} (\eta-\alpha) e^{\alpha s} \big( \z(s), \A\v_k(s)\big) \d s.
		\end{align}
		Using  $\z\in \mathrm{L}^{\infty}(0,T;\V)\cap\mathrm{L}^2(0,T;\D(\A))$ (see \eqref{O-U_conti}), and \eqref{2.8}, we obtain  that $s\mapsto e^{\alpha s } \B(\z(s)) \in \mathrm{L}^2(-k,0;\H).$ Hence, using \eqref{weak_lim10}, we obtain
		\begin{align}\label{weak_lim13}
			\lim_{n^{(k)} \to \infty}\int_{-k}^{0} e^{\alpha s} b\big(\z(s), \z(s), \A\v^{n^{(k)}}(s) \big) \d s = \int_{-k}^{0} e^{\alpha s} b\big(\z(s), \z(s),\A \v_k(s) \big)\d s.
		\end{align}
		Since we have the convergence \eqref{weak_lim6}, by the Banach-Alaoglu theorem, we can find  a subsequence of $\{\varphi(-k + t_{n^{(k)}}, \theta_{-t_{n^{(k)}}}\omega, \boldsymbol{x}_{n^{(k)}})\}$ (denoted as the same) such that
		\begin{align}\label{weak_lim17}
			\varphi(-k + t_{n^{(k)}}, \theta_{-t_{n^{(k)}}}\omega, \boldsymbol{x}_{n^{(k)}})\ \text{ converges to }\ \y_{-k} \  \text{ in }\ \H\ \text{ strongly as }  n^{(k)} \to \infty.
		\end{align}
		By Theorem \ref{RDS_Conti1}, we have
		\begin{align}\label{weak_lim18}
			\v^{n^{(k)}}(\cdot) \text{ converges to } \v_k(\cdot) \ \text{ in }\  \mathrm{L}^2(-k,0;\V) \text{ strongly}.
		\end{align}
		We have the convergences \eqref{weak_lim10} and \eqref{weak_lim18} with $\z\in \mathrm{L}^{\infty}(0,T;\V)\cap\mathrm{L}^2(0,T;\D(\A))$ (see \eqref{O-U_conti}). Therefore, by Theorems \ref{Converge_b} and \ref{Converge_c}, we conclude that
		\begin{align}
			\lim_{n^{(k)} \to \infty}\int_{-k}^{0} e^{\alpha s} b\big(\v^{n^{(k)}}(s), \v^{n^{(k)}}(s), \A\v^{n^{(k)}}(s) \big) \d s &= \int_{-k}^{0} e^{\alpha s} b\big(\v_k(s), \v_k(s), \A\v_k(s) \big) \d s,\label{weak_lim19}\\
			\lim_{n^{(k)} \to \infty}\int_{-k}^{0} e^{\alpha s} b\big(\v^{n^{(k)}}(s), \z(s), \A\v^{n^{(k)}}(s) \big) \d s &= \int_{-k}^{0} e^{\alpha s} b\big(\v_k(s), \z(s), \A\v_k(s) \big) \d s,\\
			\lim_{n^{(k)} \to \infty}\int_{-k}^{0} e^{\alpha s} b\big(\z(s), \v^{n^{(k)}}(s), \A\v^{n^{(k)}}(s) \big) \d s& = \int_{-k}^{0} e^{\alpha s} b\big(\z(s), \v_k(s), \A\v_k(s) \big) \d s,\\
			\lim_{n^{(k)} \to \infty}\int_{-k}^{0} e^{\alpha s}\big(\mathcal{C}(\v^{n^{(k)}}(s) +\z(s)),\A\v^{n^{(k)}}(s)\big) \d s &= \int_{-k}^{0} e^{\alpha s} \big(\mathcal{C}(\v_k(s) +\z(s)),\A\v_k(s)\big)\d s.\label{weak_lim20}
		\end{align}
		From \eqref{weak_lim10}, we get
		\begin{align*}
			\int_{-k}^{0} e^{\alpha s} \left[\mu\|\A\v_k(s)\|^2_{\H}+\frac{\alpha}{2}\|\v_k(s)\|^2_{\V}\right]\d s \leq \liminf_{n^{(k)}\to \infty} \int_{-k}^{0} e^{\alpha s}\left[ \mu\|\A\v^{n^{(k)}}(s)\|^2_{\H}+\frac{\alpha}{2}\|\v^{n^{(k)}}(s)\|^2_{\V}\right] \d s.
		\end{align*}
		We can also write the above inequality as 
		\begin{align}\label{weak_lim21}
			&\limsup_{n^{(k)}\to \infty} \bigg\{- \int_{-k}^{0} e^{\alpha s}\left[ \mu\|\A\v^{n^{(k)}}(s)\|^2_{\H}+\frac{\alpha}{2}\|\v^{n^{(k)}}(s)\|^2_{\V}\right] \d s\bigg\}  \nonumber\\&\leq - \int_{-k}^{0} e^{\alpha s} \left[\mu\|\A\v_k(s)\|^2_{\H}+\frac{\alpha}{2}\|\v_k(s)\|^2_{\V}\right] \d s .
		\end{align}
		From \eqref{Energy_esti6}, \eqref{positive_function}, \eqref{weak_lim11}-\eqref{weak_lim13}, \eqref{weak_lim19}-\eqref{weak_lim20}, and inequality \eqref{weak_lim21}, we conclude that
		\begin{align}\label{Energy_esti8}
			&	\limsup_{n^{(k)}\to \infty} \|\varphi(t_{n^{(k)}}, \theta_{-t_{n^{(k)}}}\omega, \boldsymbol{x}_{n^{(k)}}) - \z(0)\|^2_{\V}\nonumber\\ 
			&\leq \int_{- \infty}^{-k} h(s)\d s  + 2 \int_{-k}^{0} e^{\alpha s}\bigg\{\big((\eta-\alpha) \z(s)+\f, \A\v_k(s)\big)-b\big(\v_k(s), \v(s), \A\v_k(s)\big)\nonumber\\&\quad-b\big(\v_k(s), \z(s), \A\v_k(s)\big)-b\big(\z(s), \v_k(s), \A\v_k(s)\big) -b\big(\z(s), \z(s), \A\v_k(s)\big)\nonumber\\&\quad-\beta \big(\mathcal{C}(\v_k(s)+\z(s)), \A\v_k(s)\big)-\mu\|\A\v_k(s)\|^2_{\H}-\frac{\alpha}{2}\|\v_k(s)\|^2_{\V} \bigg\} \d s.
		\end{align} 
		Using \eqref{weak_lim8} and \eqref{Energy_esti2}, we obtain 
		\begin{align}\label{Energy_esti9}
			&	\|\y_0-\z(0)\|^2_{\V}\nonumber \\&=  \|\varphi(k, \theta_{-k}\omega,\y_{-k}) - \z(0)\|^2_{\V}= \|\v\big(0, -k; \omega, \y_{-k} - \z(-k)\big)\|^2_{\V}\nonumber\\
			&	= \|\y_{-k}-\z(-k)\|^2_{\V}\ e^{-\alpha k} + 2 \int_{-k}^{0} e^{\alpha s}\bigg\{\big((\eta-\alpha) \z(s)+\f, \A\v_k(s)\big)-b\big(\v_k(s), \v(s), \A\v_k(s)\big)\nonumber\\&\quad-b\big(\v_k(s), \z(s), \A\v_k(s)\big)-b\big(\z(s), \v_k(s), \A\v_k(s)\big) -b\big(\z(s), \z(s), \A\v_k(s)\big)\nonumber\\&\quad- \beta\big(\mathcal{C}(\v_k(s)+\z(s)), \A\v_k(s)\big)-\mu\|\A\v_k(s)\|^2_{\H}-\frac{\alpha}{2}\|\v_k(s)\|^2_{\V} \bigg\} \d s.
		\end{align}
		After combining \eqref{Energy_esti8} with \eqref{Energy_esti9}, we find 
		\begin{align}\label{Energy_esti10}
			&	\limsup_{n^{(k)}\to \infty} \|\varphi(t_{n^{(k)}}, \theta_{-t_{n^{(k)}}}\omega, \boldsymbol{x}_{n^{(k)}}) - \z(0)\|^2_{\V}\nonumber\\ &\leq \int_{- \infty}^{-k} h(s)\ \d s + \|\y_0-\z(0)\|^2_{\V} - \|\y_{-k}-\z(-k)\|^2_{\V}\ e^{-\alpha k}\nonumber\\
			&\leq \int_{- \infty}^{-k} h(s)\ \d s + \|\y_0-\z(0)\|^2_{\V}.
		\end{align}
		If we define the diagonal process $\{m_j\}^{\infty}_{j=1}$ by $m_j=j^{(j)}, j\in \mathbb{N},$ the sequence $\{m_j\}^{\infty}_{j=k}$ is a subsequence of the sequence $(n^{(k)})$ and hence by \eqref{Energy_esti10},
		\begin{align}\label{weak_lim22}
			\limsup_{j} \|\varphi(t_{m_j}, \theta_{-t_{m_j}}\omega, \boldsymbol{x}_{m_j}) - \z(0)\|^2_{\V}\leq \int_{-\infty}^{-k} h(s)\ \d s + \|\y_0 -\z(0)\|^2_{\V}.
		\end{align}
		Taking the limit $k\to \infty$ in \eqref{weak_lim22}, we arrive at
		\begin{align*}
			\limsup_{j} \|\varphi(t_{m_j}, \theta_{-t_{m_j}}\omega, \boldsymbol{x}_{m_j}) - \z(0)\|^2_{\V}\leq \|\y_0 -\z(0)\|^2_{\V},
		\end{align*}
		which proves the claim \eqref{weak_lim4} and hence the proof of Theorem \ref{V_asymptotically} is completed.
	\end{proof}
	
	From Theorems \ref{V_ab}, \ref{V_asymptotically} and Theorem 2.8, \cite{BCLLLR}, we immediately conclude the main result of this section.

	\begin{theorem}\label{Main_theorem_1}
		Suppose that, for $r\in[1,3],$ the domain $\mathcal{O}\subset\R^2$ satisfies Assumption \ref{assumpO} and Assumption \ref{assump} is satisfied. Consider the metric dynamical system, $\Im = (\Omega, \hat{\mathcal{F}}, \hat{\mathbb{P}}, \hat{\theta})$ from Proposition \ref{m-DS}, and the RDS $\varphi$ on $\V$ over $\Im$ generated by the 2D SCBF equations \eqref{S-CBF} subjected to additive noise satisfying Assumption \ref{assump}. Then, the family $\mathcal{G}$ of sets defined by $\mathcal{G}(\omega)= \Omega_{\mathrm{\bf B}}(\omega),$ for all $\omega\in \Omega$, is the minimal $\hat{\mathcal{F}}^{u}$-measurable $\hat{\mathfrak{DK}}$-random attractor for $\varphi$, where $\mathrm{\bf B}$ is $\hat{\mathfrak{DK}}$-absorbing set.
	\end{theorem}	

	\section{Invariant Measures}\label{sec5}\setcounter{equation}{0}
	This section is devoted to show the existence of invariant measures for our model in $\V$. The existence of  invariant measures for  2D SCBF equations defined on bounded domains in $\H$ is established in \cite{KM1}. It is established in \cite{CF} that the existence of compact invariant random set is a sufficient condition for the existence of invariant measures, that is, if a random dynamical system $\varphi$ has compact invariant random set, then there exist invariant measures for $\varphi$ (Corollary 4.4, \cite{CF}). Since, the random attractor itself is a compact invariant random set, the existence of invariant measures for the 2D SCBF equations \eqref{S-CBF} is a direct consequence of Corollary 4.4, \cite{CF} and Theorem \ref{Main_theorem_1}.
	
	\subsection{Existence of invariant measures}
	Let us define the transition operator $\{\mathrm{P}_t\}_{t\geq 0}$ by \begin{align}\label{71}\mathrm{P}_t f(\x)=\int_{\Omega}f(\varphi(\omega,t,\x))\d\mathbb{P}(\omega)=\E\left[f(\varphi(t,\x))\right],\end{align}  for all $f\in\mathcal{B}_b(\V)$, where $\mathcal{B}_b(\V)$ is the space of all bounded and Borel measurable functions on $\V$ and $\varphi$ is the random dynamical system corresponding to the 2D SCBF equations \eqref{S-CBF}, which is defined by \eqref{combine_sol}. Since $\varphi$ is continuous (Theorem \ref{RDS_Conti}), Proposition 3.8, \cite{BL} provides the following result: 
	\begin{lemma}
		The family $\{\mathrm{P}_t\}_{t\geq 0}$ is Feller, that is, $\mathrm{P}_tf\in\C_{b}(\V)$ if $f\in\C_b(\V)$, where $\C_b(\V)$ is the space of all bounded and continuous functions on $\V$. Furthermore, for any $f\in\C_b(\V)$, $\mathrm{P}_tf(\x)\to f(\x)$ as $t\downarrow 0$. 
	\end{lemma}
	Analogously as in the proof of Theorem 5.6, \cite{CF}, one can prove that $\varphi$ is a Markov random dynamical system, that is, $\mathrm{P}_{t_1+t_2}=\mathrm{P}_{t_1}\mathrm{P}_{t_2}$, for all $t_1,t_2\geq 0$. 
	Since, we know by Corollary 4.4, \cite{CF} that if a Markov RDS on a Polish space has an invariant compact random set, then there exists a Feller invariant probability measure $\nu$ for $\varphi$. 
	\begin{definition}
		A Borel probability measure $\nu$ on $\V$  is called an \emph{invariant measure}	for a Markov semigroup $\{\mathrm{P}_t\}_{t\geq 0}$ of Feller operators on $\C_b(\V)$ if and only if $$\mathrm{P}_{t}^*\nu=\nu, \ t\geq 0,$$ where $(\mathrm{P}_{t}^*\nu)(\Gamma)=\int_{\V}\mathrm{P}_{t}(\y,\Gamma)\nu(\d\y),$ for $\Gamma\in\mathcal{B}(\V)$ and  $\mathrm{P}_t(\y,\cdot)$ is the transition probability, $\mathrm{P}_{t}(\y,\Gamma)=\mathrm{P}_{t}(\chi_{\Gamma})(\y),\ \y\in\V$.
	\end{definition}

	By the definition of random attractors, it is clear  that there exists an invariant compact random set in $\V$. A Feller invariant probability measure for a Markov RDS $\varphi$ on $\V$ is, by definition, an invariant probability measure for the semigroup $\{\mathrm{P}_t\}_{t\geq 0}$ defined by \eqref{71}. Hence, we have the following result on the existence of invariant measures for the 2D SCBF equations \eqref{S-CBF} defined on Poincar\'e domains in $\V$.
	\begin{theorem}\label{thm6.3}
		There exists an invariant measure for the 2D SCBF equations \eqref{S-CBF} in $\V$.
	\end{theorem}
	\subsection{Uniqueness of invariant measures}
	In this work, $\W(t)$ is a Wiener process with RKHS $\mathrm{K}$ satisfying Assumption \ref{assump}. In particular, $\mathrm{K} \subset\H$ and the natural embedding  $i : \mathrm{K}\hookrightarrow \H$ is a Hilbert-Schmidt operator. For a fixed orthonormal basis $\{e_k\}_{k\in\N}$ of $\mathrm{K}$ and a sequence $\{\beta_k\}_{k\in\N}$ of independent Brownian motions defined on some filtered probability space $(\Omega, \mathscr{F}, (\mathscr{F}_t)_{t\in \R}, \mathbb{P})$ such that $\W(t)$ can be written in the following form
	\begin{align*}
		\W(t)=\sum_{k=1}^{\infty}\beta_k(t) e_k,  \ \ \ t\geq0.
	\end{align*}
	Moreover, there exists a covariance operator $\J \in \mathcal{L}(\H)$ associated with $\W(t)$ defined by 
	\begin{align*}
		\left\langle \J h_1,h_2\right\rangle=\mathbb{E}\left[\left\langle h_1,\W(1)\right\rangle_{\H}\left\langle \W(1),h_2\right\rangle_{\H}\right], \ \ \ h_1,h_2\in \H. 
	\end{align*}
	It is well known from \cite{DZ1} that $\J$ is a non-negative self-adjoint and trace class operator in $\H$. Furthermore, $\J = ii^*$ and $K = R(\J^{1/2} ),$ where $R(\J^{1/2} )$ is the range of the operator $\J^{1/2}$ (see \cite{BN}). Note that 
	\begin{align*}
		\sum_{k=1}^{\infty}\|ie_k\|^2_{\H}= \text{Tr}\left[\J\right]<\infty.
	\end{align*}
	\subsubsection{Exponential estimates}
	Here, we obtain exponential estimates which is used to get the uniqueness of invariant measures.
	\begin{theorem}\label{UIM1}
		Let $\u_1(\cdot)$ and $\u_2(\cdot)$ be two solutions of the system \eqref{S-CBF} with the initial data $\u_1^0,\u_2^0\in\H$, respectively. Then, we have
		\begin{align}\label{62}
			\mathbb{E}\left[\|\u_1(t)-\u_2(t)\|^2_{\H}\right] \leq\|\u_1^0-\u_2^0\|^2_{\H}\ \emph{\text{exp}}\left\{\frac{2}{\mu^2}\|\u_1^0\|^2_{\H}+\left(\frac{2}{\mu^2\alpha}\|\f\|^2_{\H}+\frac{2}{\mu^2} \emph{\text{Tr}}\left[\J\right]-\mu\lambda_1-2\alpha\right)t\right\},
		\end{align}
		provided $\frac{2}{\mu^2}\leq\frac{\alpha}{2\|i^*\|^2_{\mathcal{L}(\H,\mathrm{K})}}$.
	\end{theorem}
	\begin{proof}
		Let $\mathfrak{X}(\cdot)=\u_1(\cdot)-\u_2(\cdot)$, then $\mathfrak{X}(\cdot)$ satisfies  the following equality:
		\begin{align*}
			\|\mathfrak{X}(t)\|^2_{\H}&=\|\mathfrak{X}(0)\|^2_{\H}-2\mu\int_{0}^{t}\|\mathfrak{X}(\zeta)\|^2_{\V}\d\zeta-2\alpha\int_{0}^{t}\|\mathfrak{X}(\zeta)\|^2_{\H}\d\zeta\nonumber\\&\quad-2\int_{0}^{t}\left\langle\B(\u_1(\zeta))-\B(\u_2(\zeta)),\mathfrak{X}(\zeta)\right\rangle\d\zeta-2\beta\int_{0}^{t}\left\langle\mathcal{C}(\u_1(\zeta))-\mathcal{C}(\u_2(\zeta)),\mathfrak{X}(\zeta)\right\rangle\d\zeta\nonumber\\&\leq\|\mathfrak{X}(0)\|^2_{\H}-2\mu\int_{0}^{t}\|\mathfrak{X}(\zeta)\|^2_{\V}\d\zeta-2\alpha\int_{0}^{t}\|\mathfrak{X}(\zeta)\|^2_{\H}\d\zeta+2\int_{0}^{t} b(\mathfrak{X}(\zeta),\u_1(\zeta),\mathfrak{X}(\zeta))\d\zeta\nonumber\\&\leq\|\mathfrak{X}(0)\|^2_{\H}-\mu\int_{0}^{t}\|\mathfrak{X}(\zeta)\|^2_{\V}\d\zeta-2\alpha\int_{0}^{t}\|\mathfrak{X}(\zeta)\|^2_{\H}\d\zeta+\frac{2}{\mu}\int_{0}^{t} \|\u_1(\zeta)\|^2_{\V}\|\mathfrak{X}(\zeta)\|^2_{\H}\d\zeta\nonumber\\&\leq\|\mathfrak{X}(0)\|^2_{\H}-\int_{0}^{t}\left[(\mu\lambda_1+2\alpha)-\frac{2}{\mu}\|\u_1(\zeta)\|^2_{\V}\right]\|\mathfrak{X}(\zeta)\|^2_{\H}\d\zeta,
		\end{align*}
	for a.e. $t\in[0,T]$,	where we have used \eqref{2.1}, \eqref{b0}-\eqref{lady}, \eqref{MO_c}, H\"older's and Young's inequalities. The above estimate implies (using Gronwall's inequality) that
		\begin{align}\label{63}
			\|\mathfrak{X}(t)\|^2_{\H}\leq\|\mathfrak{X}(0)\|^2_{\H}\ \text{exp}\left(-(\mu\lambda_1+2\alpha)t+\frac{2}{\mu}\int_{0}^{t}\|\u_1(\zeta)\|^2_{\V}\d\zeta\right).
		\end{align}
	
		Let $\mathcal{Z}(t)=\|\u_1(t)\|^2_{\H}+\mu\int_{0}^{t}\|\u_1(\zeta)\|^2_{\V}\d\zeta$, then
		\begin{align*}
			\d\mathcal{Z}&=-\mu\|\u_1(t)\|^2_{\V}\d t -2\alpha\|\u_1(t)\|^2_{\H}\d t -2\beta\|\u_1(t)\|^{r+1}_{\wi\L^{r+1}}\d t+2\left(\u_1(t),\f\right)\d t \nonumber\\&\quad+2 \left(\u_1(t),\d\W(t)\right) + \text{Tr}\left[\J\right]\d t. 
		\end{align*}
		Applying It\^o's formula to $\mathfrak{Z}=\text{exp}(\varepsilon\mathcal{Z})$, we have
		\begin{align*}
			\d\mathfrak{Z}&=\varepsilon\mathcal{Z}\bigg(-\mu\|\u_1(t)\|^2_{\V}\d t -2\alpha\|\u_1(t)\|^2_{\H}\d t -2\beta\|\u_1(t)\|^{r+1}_{\wi\L^{r+1}}\d t+2\left(\u_1(t),\f\right)\d t \nonumber\\&\qquad\qquad +2 \left(\u_1(t),\d\W(t)\right)+ \text{Tr}\left[\J\right]\d t+2\varepsilon\|i^*\u_1(t)\|^2_{\H}\d t\bigg)\nonumber\\&\leq\varepsilon\mathcal{Z}\bigg(-\mu\|\u_1(t)\|^2_{\V}\d t -\left[\alpha-2\varepsilon\|i^*\|^2_{\mathcal{L}(\H,\mathrm{K})}\right]\|\u_1(t)\|^2_{\H}\d t -2\beta\|\u_1(t)\|^{r+1}_{\wi\L^{r+1}}\d t \nonumber\\&\qquad\qquad+\frac{1}{\alpha}\|\f\|^2_{\H} \d t+2 \left(\u_1(t),\d\W(t)\right) + \text{Tr}\left[\J\right]\d t\bigg).
		\end{align*}
		Choose $\varepsilon>0$ such that $2\varepsilon\|i^*\|^2_{\mathcal{L}(\H,\mathrm{K})}\leq\alpha,$ then we get
		\begin{align*}
			\mathfrak{Z}(t)\leq \text{exp}\left(\varepsilon\|\u_1(0)\|^2_{\H}\right) +2\varepsilon\int_{0}^{t}\mathfrak{Z}(\zeta)\left(\u_1(\zeta),\d\W(\zeta)\right) + \varepsilon\left(\frac{1}{\alpha}\|\f\|^2_{\H}+\text{Tr}\left[\J\right]\right)\int_{0}^{t}\mathfrak{Z}(\zeta)\d\zeta.
		\end{align*}
		Taking the expectation and using Gronwall's inequality, we obtain
		\begin{align}\label{64}
			\mathbb{E}\bigg[\text{exp}\bigg(\varepsilon\bigg(\|\u_1(t)\|^2_{\H}+\mu\int_{0}^{t}\|\u_1(\zeta)\|^2_{\V}\d\zeta\bigg)\bigg)\bigg]\leq\text{exp}\left(\varepsilon\left(\|\u_1(0)\|^2_{\H}+\frac{t}{\alpha}\|\f\|^2_{\H}+t\ \text{Tr}\left[\J\right]\right)\right).
		\end{align}
		Therefore, if $\frac{2}{\mu^2}\leq\frac{\alpha}{2\|i^*\|^2_{\mathcal{L}(\H,\mathrm{K})}},$ we obtain \eqref{62} from \eqref{63}-\eqref{64}.
	\end{proof}
	\begin{theorem}\label{UIM2}
		Let the condition given in Theorem \ref{UIM1} be satisfied and $\u_1^0\in\H$ be given. Then, for the condition $$\frac{2}{\mu^2\alpha}\|\f\|^2_{\H}+\frac{2}{\mu^2} \emph{\text{Tr}}\left[\J\right]\leq\mu\lambda_1+2\alpha,$$ there is a unique invariant measure to system \eqref{S-CBF}. Moreover, the invariant measure is ergodic and strongly mixing.
	\end{theorem}
	\begin{proof}
		See the proof of Theorem 5.5, \cite{MTM}.
	\end{proof}
\begin{remark}
	Theorem \ref{UIM2}  establishes the unique invariant measure in $\H$. As $\V\subset\H$ is a closed subspace, the uniqueness of invariant measure in $\V$ also follows from Theorem \ref{UIM2}. 	The existence and uniqueness of invariant measure in $\H$ for 2D SCBF equations (with $r\in[1,3]$)   perturbed by a white noise (non degenerate) is obtained in \cite{AKMTM}. The  uniqueness of invariant measure  in $\H$ for 2D as well as 3D SCBF equations driven by additive as well as multiplicative degenerate noise via the asymptotic coupling method is established in \cite{MTM3}.
\end{remark}
	\begin{remark}
		Note that, in the proof of Theorem \ref{UIM1}, even though we are using \eqref{2.1}, one can prove  without using \eqref{2.1} as $\alpha>0$.
	\end{remark}
	
	\begin{remark}
		For 2D SCBF equations \eqref{SCBF}, the results of sections \ref{sec4}-\ref{sec5} can be proved in general unbounded domains also. In that case, one has to take the norm defined on $\V$ space as $\|\u\|^2_{\V} := \|\u\|^2_{\H} + \|\nabla\u\|^2_{\H}$. Since the Stokes operator $\A$ is not invertible in general unbounded domains, one has to make a change in Assumption \ref{assump} also. Instead of \eqref{A1}, we need to assume the following: the map
		\begin{align}\label{A2}
			(\I+\A)^{-\delta} : \mathrm{K} \to \V \cap \H^{2} (\mathcal{O}) \   \text{ is }\ \gamma \text{-radonifying,}
		\end{align}
		for some $\delta\in(0,1/2)$. Under the above change in Assumption \ref{assump} (which help us to prove Lemma \ref{SOUP} in general unbounded domains) and with some minor changes in calculations, all the results of this work hold true in general unbounded domains also. 
	\end{remark}

		\section{Approximation of Attractors from Bounded to Unbounded Domain}\label{sec6}\setcounter{equation}{0}
	In this section, we prove the upper semicontinuity of the random attractor, when the domain changes from bounded to unbounded. This upper semicontinuity of the random attractor is entirely different from the upper semicontinuity results established in the work \cite{KM1}. For deterministic NSE and CBF equations, upper semicontinuity of global attractors with respect to domain is obtained in \cite{ZD} and \cite{Mohan2}, respectively. In order to prove the results of this section, we are extending the theory used for the deterministic model in \cite{Mohan2,ZD} to the stochastic CBF model. 
	
	Let $\mathcal{O}=:\mathcal{O}_{\infty}\subset\R^2$ be an unbounded Poincar\'e domain (for example, one can take $\mathcal{O}=\R\times(-L,L)$). Also let $$\mathcal{O}_m=\{x\in\mathcal{O}:|x|\leq m\}, \ m\in\bar{\N}:= \N\cup\{\infty\}.$$ It is clear that $\mathcal{O}_m\subset\mathcal{O}_{m+1}\cdots\subset\mathcal{O}$, for $m\in \N$. Due to some technical difficulties (related to the RKHS of Wiener process), we are not able to prove the upper semicontinuity of random attractors with respect to domain for the system \eqref{SCBF}. But, we are able to prove if noise in \eqref{SCBF} replaced by a finite dimensional noise (see \eqref{SCBF_in} below). Consider the following 2D SCBF equations perturbed by additive noise on $\mathcal{O}$:
		\begin{equation}\label{SCBF_in}
			\left\{
			\begin{aligned}
				\d\u(t)&=[\mu \Delta\u(t)-(\u(t)\cdot\nabla)\u(t)-\alpha\u(t)-\beta|\u(t)|^{r-1}\u(t)\\ &\qquad-\nabla p(t)+\boldsymbol{f}]\d t +\boldsymbol{\mathrm{g}} \d\mathcal{W}(t),  \text{ in } \ \mathcal{O}\times(0,\infty), \\ \nabla\cdot\u&=0, \ \text{ in } \ \mathcal{O}\times(0,\infty), \\
				\u&=\boldsymbol{0},\  \text{ on } \ \partial\mathcal{O}\times(0,\infty), \\
				\u(0)&=\x, \ \text{ in } \ \mathcal{O},
			\end{aligned}
			\right.
		\end{equation} 	
	with $r\geq1$ and $\mathcal{W}=\mathcal{W}(t,\omega)$ is an one-dimensional two-sided Wiener process defined on a probability space $(\tilde{\Omega},\tilde{\mathscr{F}},\tilde{\mathbb{P}})$. Here  $\tilde{\Omega}$ is given  by
	\begin{align*}
		\tilde{\Omega}=\{\omega\in\mathrm{C}(\R;\R):\omega(0)=0\}, 
	\end{align*}
	$\tilde{\mathscr{F}}$ is the Borel sigma-algebra induced by the compact-open topology of $\tilde{\Omega}$, and $\tilde{\mathbb{P}}$ is the two-sided Gaussian measure on $(\tilde{\Omega},\tilde{\mathscr{F}})$.	Consider the following 2D SCBF equations perturbed by additive noise on $\mathcal{O}_m$ ($m\in\bar{\N}=:\N\cup\{\infty\}$):
	\begin{equation}\label{SCBF_m}
		\left\{
		\begin{aligned}
			\d\u_m(t)&=[\mu \Delta\u_m(t)-(\u_m(t)\cdot\nabla)\u_m(t)-\alpha\u_m(t)-\beta|\u_m(t)|^{r-1}\u_m(t)\\ &\qquad-\nabla p_m(t)+\boldsymbol{f}_m]\d t +\boldsymbol{\mathrm{g}}_m \d\mathcal{W}(t),  \text{ in } \ \mathcal{O}_m\times(0,\infty), \\ \nabla\cdot\u_m&=0, \ \text{ in } \ \mathcal{O}_m\times(0,\infty), \\
			\u_m&=\boldsymbol{0},\  \text{ on } \ \partial\mathcal{O}_m\times(0,\infty), \\
			\u_m(0)&=\x_m, \ \text{ in } \ \mathcal{O}_m,
		\end{aligned}
		\right.
	\end{equation}
	 where
	 	\begin{align*}
	 	\boldsymbol{f}_m(x) &= \begin{cases}
	 		\boldsymbol{f}(x),\ &x\in\mathcal{O}_m,\\
	 		\boldsymbol{0}, \ &x\in\mathcal{O} \backslash\mathcal{O}_m,
	 	\end{cases}\ \ \
 		\boldsymbol{\mathrm{g}}_m(x) = \begin{cases}
 			\boldsymbol{\mathrm{g}}(x),\ &x\in\mathcal{O}_m,\\
 			\boldsymbol{0}, \ &x\in\mathcal{O} \backslash\mathcal{O}_m,
 		\end{cases}\\
 		\x_m(x) &= \begin{cases}
 		\x(x),\ &x\in\mathcal{O}_m,\\
 		\boldsymbol{0}, \ &x\in\mathcal{O} \backslash\mathcal{O}_m.
 	\end{cases}
	 \end{align*}
 Consider, for some $\ell>0$ 
 \begin{align*}
 	\y(\theta_{t}\omega) =  \int_{-\infty}^{t} e^{-\ell(t-s)}\d \mathcal{W}(s), \ \ \omega\in \tilde{\Omega},
 \end{align*} which is the stationary solution of the one dimensional Ornstein-Uhlenbeck equation
 \begin{align*}
 	\d\y(\theta_t\omega) + \ell\y(\theta_t\omega)\d t =\d\mathcal{W}(t).
 \end{align*}
 Let us recall from \cite{FAN} that there exists a $\theta$-invariant subset of $\tilde{\Omega}$ (will be denoted by $\tilde{\Omega}$ itself) of full measure such that $\y(\theta_t\omega)$ is continuous in $t$ for every $\omega\in \tilde{\Omega},$ and
 \begin{align}
 	\lim_{t\to \pm \infty} \frac{|\y(\theta_t\omega)|}{|t|}=0   \text{\ \ and  \ \ }
 	\lim_{t\to \pm \infty} \frac{1}{t} \int_{0}^{t} \y(\theta_{s}\omega)\d s =0.\label{Y2}
 \end{align}
Moreover,
\begin{align}\label{Y3}
	\lim_{t\to \infty} e^{-\delta t}|\y(\theta_{-t}\omega)| &=0, \ \text{ for all } \ \delta>0.
\end{align}

	Assume that $D=\{D(\omega):\omega\in\tilde{\Omega}\}$ is a family of non-empty subsets of $E$ satisfying, for every $c>0$ and $\omega\in\tilde{\Omega}$, 
\begin{align}\label{D_1}
	\lim_{t\to\infty}e^{-ct}\|D(\theta_{-t}\omega)\|_{E}=0,
\end{align}
where $\|D\|_{E}=\sup\limits_{\x\in D}\|\x\|_{E}.$ Let $\mathfrak{D}$ be the collection of all tempered families of bounded non-empty subsets of $E$, that is,
\begin{align}\label{D_11}
	\mathfrak{D}=\big\{D=\{D(\omega):\omega\in\tilde{\Omega}\}:D \text{ satisfying } \eqref{D_1}\big\}.
\end{align} 

 Define
 \begin{align}\label{T_add}
 	\v(t,\omega)=\u(t,\omega)-\textbf{g}(x)\y(\theta_{t}\omega) \ \text{ and }\ \v_m(t,\omega)=\u_m(t,\omega)-\textbf{g}_m(x)\y(\theta_{t}\omega).
 \end{align}
 Then, from \eqref{SCBF_in} and \eqref{SCBF_m}, we obtain pathwise deterministic systems (on $\mathcal{O}$)
 	\begin{equation}\label{C_SCBF_in}
 	\left\{
 	\begin{aligned}
 		\frac{\d\v}{\d t}&=\mu \Delta\v-\big((\v+\mathrm{\textbf{g}}\y)\cdot\nabla\big)(\v+\mathrm{g}\y)-\alpha\v-\beta|\v+\mathrm{\textbf{g}}\y|^{r-1}(\v+\mathrm{\textbf{g}}\y)\\ &\qquad-\nabla p+\boldsymbol{f} +(\ell-\alpha)\boldsymbol{\mathrm{g}}\y +\mu\y\Delta\mathrm{\textbf{g}},  \text{ in } \ \mathcal{O}\times(0,\infty), \\ \nabla\cdot\v&=0, \qquad\qquad\qquad\qquad\qquad\qquad\qquad\quad \ \text{ in } \ \mathcal{O}\times(0,\infty), \\
 		\v&=\boldsymbol{0},\ \qquad\qquad\qquad\qquad\qquad\qquad\qquad\quad \text{ on } \ \partial\mathcal{O}\times(0,\infty), \\
 		\v(0)&=\x-\mathrm{\textbf{g}}\y(\omega),\qquad\qquad\qquad\qquad\qquad\quad \ \text{ in } \ \mathcal{O},
 	\end{aligned}
 	\right.
 \end{equation} 	
and (on $\mathcal{O}_m$, $m\in\bar{\N}$)
	\begin{equation}\label{C_SCBF_m}
	\left\{
	\begin{aligned}
		\frac{\d\v_m}{\d t}&=\mu \Delta\v_m-\big((\v_m+\mathrm{\textbf{g}}_m\y)\cdot\nabla\big)(\v_m+\mathrm{g}_m\y)-\alpha\v_m-\beta|\v_m+\mathrm{\textbf{g}}_m\y|^{r-1}(\v_m+\mathrm{\textbf{g}}_m\y)\\ &\qquad-\nabla p_m+\boldsymbol{f}_m +(\ell-\alpha)\boldsymbol{\mathrm{g}}_m\y +\mu\y\Delta\mathrm{\textbf{g}}_m,  \text{ in } \ \mathcal{O}_m\times(0,\infty), \\ \nabla\cdot\v_m&=0, \qquad\qquad\qquad\qquad\qquad\qquad\qquad\qquad\qquad \ \text{ in } \ \mathcal{O}_m\times(0,\infty), \\
		\v_m&=\boldsymbol{0},\ \qquad\qquad\qquad\qquad\qquad\qquad\qquad\qquad\qquad \text{ on } \ \partial\mathcal{O}_m\times(0,\infty), \\
		\v_m(0)&=\x_m-\mathrm{\textbf{g}}_m\y(\omega),\qquad\qquad\qquad\qquad\qquad\qquad\quad \ \text{ in } \ \mathcal{O}_m,
	\end{aligned}
	\right.
\end{equation} 	
respectively. Let us define the space (for $m\in\bar{\N}$) $$\mathcal{V}_m:=\{\u_m\in\C_0^{\infty}(\mathcal{O}_m,\R^2):\nabla\cdot\u_m=0\},$$ where $\C_0^{\infty}(\mathcal{O}_m;\R^2)$ denote the space of all infinitely differentiable functions  ($\R^2$-valued) with compact support in $\mathcal{O}_m$. Let $\H(\mathcal{O}_m)$, $\V(\mathcal{O}_m)$ and $\wi\L^p(\mathcal{O}_m)$, for $p\in(2,\infty)$ denote the completion of $\mathcal{V}_m$ in 	$\mathrm{L}^2(\mathcal{O}_m;\R^2)$, $\mathrm{H}_0^1(\mathcal{O}_m;\R^2)$ and $\mathrm{L}^p(\mathcal{O}_m;\R^2)$ norms respectively. The spaces  $\H(\mathcal{O}_m)$, $\V(\mathcal{O}_m)$ and $\wi\L^p(\mathcal{O}_m)$  are endowed with the norms $\|\u\|_{\H(\mathcal{O}_m)}^2:=\int_{\mathcal{O}_m}|\u(x)|^2\d x$, $\|\u\|_{\V(\mathcal{O}_m)}^2:=\int_{\mathcal{O}_m}|\nabla\u(x)|^2\d x$ and  $\|\u\|_{\wi \L^p(\mathcal{O}_m)}^2:=\int_{\mathcal{O}_m}|\u(x)|^p\d x,$ respectively. The induced duality between the spaces $\V(\mathcal{O}_m)$ and $\V'(\mathcal{O}_m)$, and $\widetilde{\L}^p(\mathcal{O}_m)$ and its dual $\widetilde{\L}^{\frac{p}{p-1}}(\mathcal{O}_m)$ is denoted by $\langle\cdot,\cdot\rangle.$
	
	 For all $t\geq0,$ and for every $\v(0)\in\H(\mathcal{O}_m)$ and $\omega\in\Omega$, \eqref{C_SCBF_m} has a unique solution in $\mathrm{C}\big([0,T];\H(\mathcal{O}_m)\big)\cap\mathrm{L}^2\big(0,T;\V(\mathcal{O}_m)\big)\cap\mathrm{L}^{r+1}\big(0,T;\widetilde{\L}^{r+1}(\mathcal{O}_m)\big)$, for all $m\in\bar{\N}$. Define a cocycle $\Phi_m:\R^+\times\Omega\times\H(\mathcal{O}_m)\to\H(\mathcal{O}_m)$ for the system \eqref{SCBF_m} such that for given $\omega\in\Omega$, $\u(s)\in\H$ and for all $t\geq s$,
	\begin{align}\label{Phi_m}
		 \u_m(t,s;\omega,\u_m(s))=\Phi_m(t-s;\theta_{s}\omega)\u_m(s)=\v_m(t,s;\omega,\v_m(s))+\textbf{g}_m\y(\theta_{t}\omega).
	\end{align}
	For the existence of a unique random attractor of the system \eqref{SCBF_m} in $\H(\mathcal{O}_m)$, we refer the readers to \cite{KM7}. See section 5 in \cite{KM7} for unbounded domains, and for bounded domains see section 3 in \cite{KM7} and the work \cite{KM1}, where authors have used compact Sobolev embedding to prove their results in bounded domains. In this work, we are proving upper semicontinuity of random attractor for 2D SCBF equations when $r>1$ (for $r=1$, see Remark \ref{r=1} below).
	
	\begin{definition}[Expansion and restriction of a function]
	For a function $\u:\mathcal{O}_m\to \R^2,$ its null-expansion $\widetilde{\u}:\mathcal{O}\to \R^2$ is defined by 
	\begin{align}
		\widetilde{\u}(x) = \begin{cases}
			\u(x),\ &x\in\mathcal{O}_m,\\
			\boldsymbol{0}, \ &x\in\mathcal{O} \backslash\mathcal{O}_m.
		\end{cases}
	\end{align}
	For a function $\v:\mathcal{O}\to \R^2,$ the restriction $\v|_{\mathcal{O}_m}:\mathcal{O}_{m}\to \R^2$ is given by
	$$\v|_{\mathcal{O}_m}(x)=\v(x), \quad x\in \mathcal{O}_m.$$
\end{definition}
	On taking the orthogonal projections $\mathcal{P}$ and $\mathcal{P}_{\mathcal{O}_m}: \L^2(\mathcal{O}_m) \to\H(\mathcal{O}_m)$ (see subsection \ref{proj}) onto the equations \eqref{C_SCBF_in} and \eqref{C_SCBF_m}, respectively,  we obtain following systems on $\mathcal{O}$ and $\mathcal{O}_m$ ($m\in\N$):
	\begin{equation}\label{cscbf_in}
		\left\{
		\begin{aligned}
			\frac{\d\v}{\d t} &= -\mu \A\v (t)- \B(\v+\textbf{g}\y)-\alpha\v - \beta \mathcal{C}(\v +\textbf{g}\y) \\&\quad+ \boldsymbol{f} + (\ell-\alpha)\textbf{g}\y +\mu\y\Delta\textbf{g}, \\
			\v(0)&= \x - \textbf{g}\y(\omega),
		\end{aligned}
		\right.
	\end{equation}	
and
\begin{equation}\label{cscbf_m}
	\left\{
	\begin{aligned}
		\frac{\d\v_m}{\d t} &= -\mu \A_m\v_m (t)- \B_m(\v_m+\textbf{g}_m\y)-\alpha\v_m - \beta \mathcal{C}_m(\v_m +\textbf{g}_m\y) \\&\quad+ \boldsymbol{f}_m + (\ell-\alpha)\textbf{g}_m\y +\mu\y\Delta\textbf{g}_m, \\
		\v_m(0)&= \x_{m} - \textbf{g}_m\y(\omega),
	\end{aligned}
	\right.
\end{equation}	
respectively, where $\A_m, \B_m$ and $\mathcal{C}_m$ are defined on $\mathcal{O}_m$ in a similar way as of $\A, \B$ and $\mathcal{C}$ defined on $\mathcal{O}$ in section \ref{sec2}.

Let $\mathcal{A}_m(\omega)$ represent the random attractor for the 2D SCBF equations \eqref{SCBF_m} in $\H(\mathcal{O}_m)$ for $m\in\bar{\N}$. In this section, we prove the upper semicontinuity from $\mathcal{A}_m(\omega)$ to $\mathcal{A}_{\infty}(\omega)$ as $m\to \infty$, that is, as the bounded domain $\mathcal{O}_m$ is expanded to the unbounded domain $\mathcal{O}.$ The main difficulty is the proper definition of the Hausdorff semidistance between $\mathcal{A}_m(\omega)$ ($m\in\N$) and $\mathcal{A}_{\infty}(\omega)$, since they lie in different spaces $\H(\mathcal{O}_m)$ ($m\in\N$) and $\H(\mathcal{O})$. In order to overcome the above difficulty, we consider the null-expansion $\widetilde{\u}_m$ of solution $\u_m$ defined by 
	\begin{align}
		\widetilde{\u}_m = \begin{cases}
			\u_m,\ &x\in\mathcal{O}_m,\\
			\boldsymbol{0}, \ &x\in\mathcal{O} \backslash\mathcal{O}_m.
		\end{cases}
	\end{align}
	It can be easily seen that $\widetilde{\u}_m\in \H(\mathcal{O})$ if $\u_m\in \H(\mathcal{O}_m)$. Notice that $\widetilde{\u}_m$ ($m\in\N$) is  not equal to $\u_{\infty}.$ Hence, the null-expansion of random attractor $\mathcal{A}_m(\omega)$ is denoted by $\widetilde{\mathcal{A}}_m(\omega)$ and is define as follows: 
	$$\widetilde{\mathcal{A}}_m(\omega)=\{\u\in\H(\mathcal{O}) : \text{there exists }\  \v\in\mathcal{A}_m(\omega), \text{ such that }\   \u=\widetilde{\v}\}, \text{ for all } \omega\in \Omega.$$ It follows that  all $\widetilde{\mathcal{A}}_m(\omega)$ lies in the same space $\H(\mathcal{O}),$ and thus, the Hausdorff semidistance can be understood in the following sense:
	\begin{align}
		d_m(\omega):=\text{dist}_{\H(\mathcal{O})}\big(\widetilde{\mathcal{A}}_m(\omega), \mathcal{A}_{\infty}(\omega)\big)= \sup_{\u\in\widetilde{\mathcal{A}}_m(\omega)}\inf_{\v\in\mathcal{A}_{\infty}(\omega)}\|\u-\v\|_{\H(\mathcal{O})}.
	\end{align}
	Hence, our aim in this section  is to prove $d_m(\omega)\to 0$ as $m\to \infty.$ In the sequel, we  use the tilde symbol $\ \widetilde{\cdot}\ $ to denote the null-expansion of a function, a set, an operator, etc.

	\subsection{Random Dynamical systems and their expansions for SCBF equations}
	We will assume $m\in\N$ for this section unless stated. Let $\Phi_{m}(t, \omega):\H(\mathcal{O}_m)\to \H(\mathcal{O}_m)$ be the RDS generated by 2D SCBF equations \eqref{SCBF_m} for domain $\mathcal{O}_{m}$ and a natural expansion of $\Phi_{m}$ is given by $\widetilde{\Phi}_{m}:\H(\mathcal{O})\to \H(\mathcal{O})$,
	\begin{align}
		(\widetilde{\Phi}_{m}\u)(x) = \begin{cases}
			(\Phi_{m}(\u|_{\mathcal{O}_m}))(x),\ &\text{ for all }x\in\mathcal{O}_m,\\
			0, \ &\text{ for all }x\in\mathcal{O} \backslash\mathcal{O}_m.
		\end{cases}
	\end{align}
The following Lemma is useful to obtain the absorbing set.
	\begin{lemma}\label{Absorb}
	For $r>1$, let $\f\in\H(\mathcal{O})$ and $\textbf{g}\in\D(\A)$. Then for all $\omega\in\tilde{\Omega}$ and for every $D=\{D(\omega):\omega\in\tilde{\Omega}\}\in\mathfrak{D}$, there exists $\mathcal{T}_{\infty}=\mathcal{T}_{\infty}(\omega,D)>0$ such that for all $t\geq\mathcal{T}_{\infty}$, the solution $\v(\cdot)$ of \eqref{cscbf_in} satisfies
	\begin{align}\label{H0^*}
	&		\|\v(0,-t;\omega, \v_{0})\|^2_{\H(\mathcal{O})} +\mu\int_{-t}^{0} e^{\alpha \zeta} \|\v(\zeta,-t;\omega,\v_{0})\|^2_{\V(\mathcal{O})}\ \d \zeta \nonumber\\&+\beta\int_{-t}^{0} e^{\alpha\zeta} \|\v(\zeta,-t;\omega,\v_{0})+\y(\theta_{\zeta}\omega)\textbf{g}\|^{r+1}_{\wi\L^{r+1}(\mathcal{O})} \d \zeta \nonumber\\&  \leq M \int_{-\infty}^{0} e^{\alpha\zeta} \bigg[\|\f\|_{\H(\mathcal{O})}^2+\left(\|\textbf{g}\|^2_{\H(\mathcal{O})}+\|\textbf{g}\|^2_{\D(\A)}\right)|\y(\theta_{\zeta}\omega)|^{2}+\left|\y(\theta_{\zeta}\omega)\right|^{\frac{2(r+1)}{r-1}}\|\textbf{g}\|^{\frac{2(r+1)}{r-1}}_{\wi\L^{\frac{2(r+1)}{r-1}}(\mathcal{O})}\nonumber\\&\qquad\qquad\qquad+  |\y(\theta_{\zeta}\omega)|^{r+1}\|\textbf{g}\|^{r+1}_{\wi\L^{r+1}(\mathcal{O})}\bigg]\d \zeta=:\mathscr{L}(\omega),
\end{align}
and
\begin{align}\label{H0}
	\sup_{s\in[-t,0]}\|\v(s,-t;\omega, \v_{0})\|^2_{\H(\mathcal{O})}\leq\mathscr{L}(\omega).
\end{align}
where $\v_{0}\in D(\theta_{-t}\omega)$ and $M>0$ is a constant independent of $\omega$ and $D$.
\end{lemma}

\begin{proof}
	Taking the inner product with $\v(\cdot)$ to the first equation in \eqref{cscbf_in}, we have
	\begin{align}\label{H1}
		\frac{1}{2}\frac{\d}{\d t} \|\v(t)\|^2_{\H(\mathcal{O})}=&-\mu \|\v(t)\|^2_{\V(\mathcal{O})}-\alpha \|\v(t)\|^2_{\H(\mathcal{O})} - b\big(\v(t)+\y(\theta_t\omega)\textbf{g}, \v(t)+\y(\theta_t\omega)\textbf{g}, \v(t)\big)\nonumber\\&-\beta\left\langle \mathcal{C}(\v(t)+\y(\theta_t\omega)\textbf{g}),\v(t)\right\rangle+\big(\f,\v(t)\big)+(\ell-\alpha)\big(\y(\theta_t\omega)\textbf{g}, \v(t)\big) \nonumber\\&-\big(\mu\y(\theta_t\omega)\A \textbf{g},\v(t)\big),\nonumber\\
		=&-\mu \|\v(t)\|^2_{\V(\mathcal{O})}-\alpha \|\v(t)\|^2_{\H(\mathcal{O})} - b\big(\v(t)+\y(\theta_t\omega)\textbf{g},\y(\theta_t\omega)\textbf{g}, \v(t)\big)\nonumber\\&-\beta\|\v(t)+\y(\theta_t\omega)\textbf{g}\|^{r+1}_{\wi \L^{r+1}(\mathcal{O})}+\beta\left\langle \mathcal{C}(\v(t)+\y(\theta_t\omega)\textbf{g}),\y(\theta_t\omega)\textbf{g}\right\rangle\nonumber\\&+\big(\f,\v(t)\big)+(\ell-\alpha)\big(\y(\theta_t\omega)\textbf{g}, \v(t)\big) -\big(\mu\y(\theta_t\omega)\A\textbf{g},\v(t)\big),
	\end{align}
	for a.e. $t\in[0,T]$, where we have used \eqref{b0}. Applying H\"older's and Young's inequalities, we obtain for $r>1$
	\begin{align}\label{H2}
		&\big|(\ell-\alpha)\big(\y(\theta_t\omega)\textbf{g}, \v\big)\big|+\big|\big(\f,\v\big)\big|+\big|\big(\mu\y(\theta_t\omega)\A\textbf{g},\v\big)\big|
		\nonumber\\&\leq C\|\f\|_{\H(\mathcal{O})}^2+ C\left[\|\textbf{g}\|^2_{\H(\mathcal{O})}+\|\textbf{g}\|^2_{\D(\A)}\right] |\y(\theta_t\omega)|^2 + \frac{\alpha}{2} \|\v\|_{\H(\mathcal{O})}^2,
	\end{align} 
	\begin{align}\label{H3}
	&	\big|b\big(\v+\y(\theta_t\omega)\textbf{g}, \y(\theta_t\omega)\textbf{g}, \v\big)\big|\nonumber\\&\leq \left|\y(\theta_t\omega)\right|\|\v+\y(\theta_t\omega)\textbf{g}\|_{\wi\L^{r+1}(\mathcal{O})}\|\v\|_{\V(\mathcal{O})}\|\textbf{g}\|_{\wi\L^{\frac{2(r+1)}{r-1}}(\mathcal{O})}\nonumber\\&\leq \frac{\beta}{4}\|\v+\y(\theta_t\omega)\textbf{g}\|^{r+1}_{\wi\L^{r+1}(\mathcal{O})}+\frac{\mu}{2}\|\v\|^2_{\V(\mathcal{O})}+C\left|\y(\theta_t\omega)\right|^{\frac{2(r+1)}{r-1}}\|\textbf{g}\|^{\frac{2(r+1)}{r-1}}_{\wi\L^{\frac{2(r+1)}{r-1}}(\mathcal{O})},
	\end{align}
and
	\begin{align}\label{H4}
		\beta\left\langle \mathcal{C}(\v+\y(\theta_t\omega)\textbf{g}),\y(\theta_t\omega)\textbf{g}\right\rangle&\leq \beta|\y(\theta_t\omega)|\|\v+\y(\theta_t\omega)\textbf{g}\|^{r}_{\wi \L^{r+1}(\mathcal{O})}\|\textbf{g}\|_{\wi \L^{r+1}(\mathcal{O})}\nonumber\\&\leq\frac{\beta}{4}\|\v+\y(\theta_t\omega)\textbf{g}\|^{r+1}_{\wi \L^{r+1}(\mathcal{O})}+ C |\z(\theta_t\omega)|^{r+1}\|\textbf{g}\|^{r+1}_{\wi\L^{r+1}(\mathcal{O})} .
	\end{align}
	Combining \eqref{H2}-\eqref{H4} and using it in \eqref{H1}, we find
	\begin{align*}
		&\frac{\d}{\d t} \|\v(t)\|^2_{\H(\mathcal{O})}+\mu \|\v(t)\|^2_{\V(\mathcal{O})}+\alpha\|\v(t)\|^2_{\H(\mathcal{O})}+\beta\|\v(t)+\y(\theta_t\omega)\textbf{g}\|^{r+1}_{\wi \L^{r+1}(\mathcal{O})}\nonumber\\&\leq C\|\f\|_{\H(\mathcal{O})}^2+ C\left[\|\textbf{g}\|^2_{\H(\mathcal{O})}+\|\textbf{g}\|^2_{\D(\A)}\right] |\y(\theta_t\omega)|^2 +C\left|\y(\theta_t\omega)\right|^{\frac{2(r+1)}{r-1}}\|\textbf{g}\|^{\frac{2(r+1)}{r-1}}_{\wi\L^{\frac{2(r+1)}{r-1}}(\mathcal{O})}\nonumber\\&\quad+ C |\z(\theta_t\omega)|^{r+1}\|\textbf{g}\|^{r+1}_{\wi\L^{r+1}(\mathcal{O})} .
	\end{align*}
	for a.e. $t\in[0,T]$. By means of variation of constants formula, we obtain 
	\begin{align*}
		&	e^{\alpha t_1}\|\v(t_1)\|^2_{\H(\mathcal{O})} +\mu\int_{t_2}^{t_1} e^{\alpha \zeta} \|\v(\zeta)\|^2_{\V(\mathcal{O})} \d \zeta +\beta\int_{t_2}^{t_1} e^{\alpha \zeta} \|\v(\zeta)+\y(\theta_{\zeta}\omega)\textbf{g}\|^{r+1}_{\wi\L^{r+1}(\mathcal{O})} \d \zeta\nonumber\\& \leq e^{-\alpha t_2}\|\v(t_2)\|^2_{\H(\mathcal{O})}+C\int_{t_2}^{t_1} e^{\alpha\zeta} \bigg[\|\f\|_{\H(\mathcal{O})}^2+\left(\|\textbf{g}\|^2_{\H(\mathcal{O})}+\|\textbf{g}\|^2_{\D(\A)}\right)|\y(\theta_{\zeta}\omega)|^{2}\nonumber\\&\qquad+\left|\y(\theta_{\zeta}\omega)\right|^{\frac{2(r+1)}{r-1}}\|\textbf{g}\|^{\frac{2(r+1)}{r-1}}_{\wi\L^{\frac{2(r+1)}{r-1}}(\mathcal{O})}+  |\y(\theta_{\zeta}\omega)|^{r+1}\|\textbf{g}\|^{r+1}_{\wi\L^{r+1}(\mathcal{O})}\bigg]\d \zeta.
	\end{align*}
	Therefore, for any $t>0$, $s\geq-t$ (for $s\in[-t,0]$) and $\v_{0}\in D(\theta_{-t}\omega)$, we get
	\begin{align}\label{H5}
		&	\|\v(s,-t;\omega, \v_{0})\|^2_{\H(\mathcal{O})} +\mu\int_{-t}^{s} e^{\alpha \zeta} \|\v(\zeta,-t;\omega,\v_{0})\|^2_{\V(\mathcal{O})}\ \d \zeta \nonumber\\&+\beta\int_{-t}^{s} e^{\alpha\zeta} \|\v(\zeta,-t;\omega,\v_{0})+\y(\theta_{\zeta}\omega)\textbf{g}\|^{r+1}_{\wi\L^{r+1}(\mathcal{O})} \d \zeta \nonumber\\& \leq e^{-\alpha t}\|\v_{0}\|^2_{\H(\mathcal{O})}+C \int_{-t}^{s} e^{\alpha\zeta} \bigg[\|\f\|_{\H(\mathcal{O})}^2+\left(\|\textbf{g}\|^2_{\H(\mathcal{O})}+\|\textbf{g}\|^2_{\D(\A)}\right)|\y(\theta_{\zeta}\omega)|^{2}\nonumber\\&\qquad+\left|\y(\theta_{\zeta}\omega)\right|^{\frac{2(r+1)}{r-1}}\|\textbf{g}\|^{\frac{2(r+1)}{r-1}}_{\wi\L^{\frac{2(r+1)}{r-1}}(\mathcal{O})}+  |\y(\theta_{\zeta}\omega)|^{r+1}\|\textbf{g}\|^{r+1}_{\wi\L^{r+1}(\mathcal{O})}\bigg]\d \zeta\nonumber\\& \leq e^{-\alpha t}\|\v_{0}\|^2_{\H(\mathcal{O})}+C \int_{-\infty}^{0} e^{\alpha\zeta} \bigg[\|\f\|_{\H(\mathcal{O})}^2+\left(\|\textbf{g}\|^2_{\H(\mathcal{O})}+\|\textbf{g}\|^2_{\D(\A)}\right)|\y(\theta_{\zeta}\omega)|^{2}\nonumber\\&\qquad+\left|\y(\theta_{\zeta}\omega)\right|^{\frac{2(r+1)}{r-1}}\|\textbf{g}\|^{\frac{2(r+1)}{r-1}}_{\wi\L^{\frac{2(r+1)}{r-1}}(\mathcal{O})}+  |\y(\theta_{\zeta}\omega)|^{r+1}\|\textbf{g}\|^{r+1}_{\wi\L^{r+1}(\mathcal{O})}\bigg]\d \zeta.
	\end{align}
Since, $\v_{0}\in D(\theta_{-t}\omega)$ and $D\in\mathfrak{D}$, we have
\begin{align}\label{H6}
	e^{-\alpha t}\|\v_{0}\|^2_{\H(\mathcal{O})}\leq e^{-\alpha t}\|D(\theta_{-t}\omega)\|^2_{\H(\mathcal{O})}\to 0 \text{ as } t\to \infty. 
\end{align}
From \eqref{Y3}, we infer that the second term on the right hand side of inequality \eqref{H5} is finite. Moreover, \eqref{H6} implies that there exists $\mathcal{T}_{\infty}=\mathcal{T}_{\infty}(\omega,D)$ such that 
\begin{align*}
	e^{-\alpha t}\|\v_{0}\|^2_{\H(\mathcal{O})}&\leq C \int_{-\infty}^{0} e^{\alpha\zeta} \bigg[\|\f\|_{\H(\mathcal{O})}^2+\left(\|\textbf{g}\|^2_{\H(\mathcal{O})}+\|\textbf{g}\|^2_{\D(\A)}\right)|\y(\theta_{\zeta}\omega)|^{2}\nonumber\\&\qquad+\left|\y(\theta_{\zeta}\omega)\right|^{\frac{2(r+1)}{r-1}}\|\textbf{g}\|^{\frac{2(r+1)}{r-1}}_{\wi\L^{\frac{2(r+1)}{r-1}}(\mathcal{O})}+  |\y(\theta_{\zeta}\omega)|^{r+1}\|\textbf{g}\|^{r+1}_{\wi\L^{r+1}(\mathcal{O})}\bigg]\d \zeta, \ \ \text{ for all } t\geq\mathcal{T}_{\infty},
\end{align*}
which gives \eqref{H0^*} along with \eqref{H5} and \eqref{H6} for $s=0,$ and \eqref{H0} is immediate.
\end{proof}
The following Lemma can be proved in a similar way as Lemma \ref{Absorb}.
	\begin{lemma}\label{Absorb1}
	For $r>1$, let $\f_m\in\H(\mathcal{O}_m)$ and $\textbf{g}_m\in\D(\A_m)$. Then for all $\omega\in\tilde{\Omega}$ and for every $D=\{D(\omega):\omega\in\tilde{\Omega}\}\in\mathfrak{D}$, there exists $\mathcal{T}_m=\mathcal{T}_m(\omega,D)>0$ such that for all $t\geq\mathcal{T}_m$, the solution $\v_m(\cdot)$ of \eqref{cscbf_m} satisfies
	\begin{align}\label{H0^*1}
		&		\|\v_m(0,-t;\omega, \v_{m,0})\|^2_{\H(\mathcal{O}_m)} +\mu\int_{-t}^{0} e^{\alpha \zeta} \|\v_m(\zeta,-t;\omega,\v_{m,0})\|^2_{\V(\mathcal{O}_m)}\ \d \zeta \nonumber\\&+\beta\int_{-t}^{0} e^{\alpha\zeta} \|\v_m(\zeta,-t;\omega,\v_{m,0})+\y(\theta_{\zeta}\omega)\textbf{g}_m\|^{r+1}_{\wi\L^{r+1}(\mathcal{O}_m)} \d \zeta \leq\mathscr{L}(\omega),
	\end{align}
and
	\begin{align}\label{H01}
		\sup_{s\in[-t,0]}\|\v_m(s,-t;\omega, \v_{m,0})\|^2_{\H(\mathcal{O}_m)}\leq\mathscr{L}(\omega),
	\end{align}
	where $\mathscr{L}(\omega)$ is same as in \eqref{H0^*} and $\v_{m,0}\in D(\theta_{-t}\omega)$.
\end{lemma}

It implies from Lemmas \ref{Absorb} and \ref{Absorb1} that, for every $D=\{D(\omega):\omega\in\tilde{\Omega}\}$ and $\mathbb{P}$-a.e. $\omega\in\tilde{\Omega}$, there exists a time $\mathcal{T}(\omega,D)=\max\limits_{m\in\bar{\N}}\mathcal{T}_m(\omega,D)>0$ such that for all $t\geq\mathcal{T}(\omega,D)$,
\begin{align}
	\|\Phi_{m}(t,\theta_{-t}\omega,D(\theta_{-t}\omega))\|^2_{\H(\mathcal{O}_m)}\leq[\mathscr{L}^*(\omega)]^2,
\end{align}
where $[\mathscr{L}^*(\omega)]^2=\mathscr{L}(\omega)+\|\textbf{g}\|^2_{\H(\mathcal{O})}\left|\y(\omega)\right|^2$ and $\mathscr{L}(\omega)$ is the same as in \eqref{H0^*}, that is, there exists a random $\mathfrak{D}$-absorbing set $\mathcal{K}_m(\omega)$ of $\Phi_{m}$ ($m\in\bar{\N}$) given by
\begin{align}\label{absorbing}
	\mathcal{K}_m(\omega):=\{\u\in\H(\mathcal{O}_m):\|\u\|^2_{\H(\mathcal{O}_m)}\leq[\mathscr{L}^*(\omega)]^2\}.
\end{align}

	\begin{lemma}\label{largeradius}
	For $r>1$, assume that $\f\in\H(\mathcal{O})$, $\textbf{g}\in\D(\A)$. Then, for any $\v_{0}\in D(\theta_{-t}\omega),$ where $D=\{D(\omega):\omega\in \tilde{\Omega}\}\in\mathfrak{D}$, and for any $\varepsilon>0$ and $\omega\in \tilde{\Omega}$, there exists $\mathscr{T}_{\infty}=\mathscr{T}_{\infty}(\omega,D,\varepsilon)\geq 0$ and $K_1(\omega,\varepsilon)>0$ such that for all $t\geq \mathscr{T}_{\infty}$ and $k\geq K_1$, the solution of \eqref{cscbf_in}  satisfies
	\begin{align}\label{ep}
		&\int_{\mathcal{O}\backslash\mathcal{O}_k}|\v(0,-t;\omega,\v_{0}) |^2\d x+\mu\int_{-t}^{0}e^{\alpha\zeta} \int_{\mathcal{O}\backslash\mathcal{O}_k}|\nabla\v(\zeta,-t;\omega,\v_{0})|^2 \d x \d\zeta\nonumber\\& +\int_{-t}^{0}e^{\alpha\zeta} \int_{\mathcal{O}\backslash\mathcal{O}_k}|\v(\zeta,-t;\omega,\v_{0})+\textbf{g}(x)\y(\theta_{\zeta}\omega)|^{r+1}\d x \d\zeta\leq \varepsilon,
	\end{align}
and
\begin{align*}
	\sup_{s\in[-t,0]}\int_{\mathcal{O}\backslash\mathcal{O}_k}|\v(s,-t;\omega,\v_{0}) |^2\d x\leq \varepsilon.
\end{align*}
\end{lemma}
\begin{proof}
	Let $\uprho$ be a smooth function such that $0\leq\uprho(s)\leq 1$ for $s\in\R^+$ and 
	\begin{align}\label{Psi}
		\uprho(s)=\begin{cases*}
			0, \text{ for }0\leq s\leq 1,\\
			1, \text{ for } s\geq2.
		\end{cases*}
	\end{align}
	Then, there exists a positive constant $C$ such that $|\uprho'(s)|\leq C$ for all $s\in\R^+$. Taking divergence to the first equation of \eqref{C_SCBF_in}, formally we obtain
	\begin{align*}
-\Delta p&=\nabla\cdot\left[\big((\v+\mathrm{\textbf{g}}\y)\cdot\nabla\big)(\v+\mathrm{g}\y)\right]+\beta\nabla\cdot\left[|\v+\mathrm{\textbf{g}}\y|^{r-1}(\v+\mathrm{\textbf{g}}\y)\right]\\
&=\nabla\cdot\left[\nabla\cdot\big((\v+\textbf{g}\y)\otimes(\v+\textbf{g}\y)\big)\right]+\beta\nabla\cdot\left[|\v+\mathrm{\textbf{g}}\y|^{r-1}(\v+\mathrm{\textbf{g}}\y)\right]\\
&=\sum_{i,j=1}^{2}\frac{\partial^2}{\partial x_i\partial x_j}\big((v_i+\text{g}_i\y)(v_j+\text{g}_j\y)\big)+\beta\nabla\cdot\left[|\v+\mathrm{\textbf{g}}\y|^{r-1}(\v+\mathrm{\textbf{g}}\y)\right],
	\end{align*}
which implies that
	\begin{align}\label{p-value}
	 p=(-\Delta)^{-1}\left[\sum_{i,j=1}^{2}\frac{\partial^2}{\partial x_i\partial x_j}\big((v_i+\text{g}_i\y)(v_j+\text{g}_j\y)\big)+\beta\nabla\cdot\left[|\v+\mathrm{\textbf{g}}\y|^{r-1}(\v+\mathrm{\textbf{g}}\y)\right]\right].
\end{align}
Taking the inner product to the first equation of \eqref{C_SCBF_in} with $\uprho\left(\frac{|x|^2}{k^2}\right)\v$, we have
	\begin{align}\label{ep1}
		&\frac{1}{2} \frac{\d}{\d t} \int_{\mathcal{O}}\uprho\left(\frac{|x|^2}{k^2}\right)|\v|^2\d x\nonumber \\&= \mu \int_{\mathcal{O}}(\Delta\v) \uprho\left(\frac{|x|^2}{k^2}\right) \v \d x-\alpha \int_{\mathcal{O}}\uprho\left(\frac{|x|^2}{k^2}\right)|\v|^2\d x-b\left(\v+\textbf{g}\y,\v+\textbf{g}\y,\uprho\left(\frac{|x|^2}{k^2}\right)\v\right)\nonumber\\&\quad-\beta \int_{\mathcal{O}}\left|\v+\textbf{g}\y\right|^{r-1}(\v+\textbf{g}\y)\uprho\left(\frac{|x|^2}{k^2}\right)\v\d x-\int_{\mathcal{O}}(\nabla p)\uprho\left(\frac{|x|^2}{k^2}\right)\v\d x+ \int_{\mathcal{O}}\f\uprho\left(\frac{|x|^2}{k^2}\right)\v\d x \nonumber\\&\quad+(\ell-\alpha)\y\int_{\mathcal{O}}\textbf{g}\uprho\left(\frac{|x|^2}{k^2}\right)\v\d x+\mu\y\int_{\mathcal{O}}(\Delta\textbf{g})\uprho\left(\frac{|x|^2}{k^2}\right)\v\d x\nonumber \\&= \mu \int_{\mathcal{O}}(\Delta\v) \uprho\left(\frac{|x|^2}{k^2}\right) \v \d x-\alpha \int_{\mathcal{O}}\uprho\left(\frac{|x|^2}{k^2}\right)|\v|^2\d x-b\left(\v+\textbf{g}\y,\v+\textbf{g}\y,\uprho\left(\frac{|x|^2}{k^2}\right)(\v+\textbf{g}\y)\right)\nonumber\\&\quad+b\left(\v+\textbf{g}\y,\v+\textbf{g}\y,\uprho\left(\frac{|x|^2}{k^2}\right)\textbf{g}\y\right)-\beta \int_{\mathcal{O}}\uprho\left(\frac{|x|^2}{k^2}\right)\left|\v+\textbf{g}\y\right|^{r+1}\d x\nonumber\\&\quad+\beta \int_{\mathcal{O}}\left|\v+\textbf{g}\y\right|^{r-1}(\v+\textbf{g}\y)\uprho\left(\frac{|x|^2}{k^2}\right)\textbf{g}\y\d x-\int_{\mathcal{O}}(\nabla p)\uprho\left(\frac{|x|^2}{k^2}\right)\v\d x+ \int_{\mathcal{O}}\f\uprho\left(\frac{|x|^2}{k^2}\right)\v\d x \nonumber\\&\quad+(\ell-\alpha)\y\int_{\mathcal{O}}\textbf{g}\uprho\left(\frac{|x|^2}{k^2}\right)\v\d x+\mu\y\int_{\mathcal{O}}(\Delta\textbf{g})\uprho\left(\frac{|x|^2}{k^2}\right)\v\d x.
	\end{align}
	We estimate each term on right hand side of \eqref{ep1}. Integration by parts and divergence free condition of $\v(\cdot)$ help us to obtain
	\begin{align}\label{ep2}
		&\mu \int_{\mathcal{O}}(\Delta\v) \uprho\left(\frac{|x|^2}{k^2}\right) \v \d x\nonumber\\&= -\mu \int_{\mathcal{O}}|\nabla\v|^2 \uprho\left(\frac{|x|^2}{k^2}\right)  \d x -\mu \int_{\mathcal{O}} \uprho'\left(\frac{|x|^2}{k^2}\right)\frac{2}{k^2}(x\cdot\nabla) \v\cdot\v \d x\nonumber\\&= -\mu \int_{\mathcal{O}}|\nabla\v|^2 \uprho\left(\frac{|x|^2}{k^2}\right)  \d x -\mu \int\limits_{k\leq|x|\leq \sqrt{2}k}\uprho'\left(\frac{|x|^2}{k^2}\right)\frac{2}{k^2}(x\cdot\nabla) \v\cdot\v \d x\nonumber\\&\leq -\mu \int_{\mathcal{O}}|\nabla\v|^2 \uprho\left(\frac{|x|^2}{k^2}\right)  \d x +\frac{2\sqrt{2}\mu}{k} \int\limits_{k\leq|x|\leq \sqrt{2}k}\left|\v\right| \left|\uprho'\left(\frac{|x|^2}{k^2}\right)\right|\left|\nabla \v\right| \d x\nonumber\\&\leq -\mu \int_{\mathcal{O}}|\nabla\v|^2 \uprho\left(\frac{|x|^2}{k^2}\right)  \d x +\frac{C}{k} \int_{\mathcal{O}}\left|\v\right| \left|\nabla \v\right| \d x\nonumber\\&\leq -\mu \int_{\mathcal{O}}|\nabla\v|^2 \uprho\left(\frac{|x|^2}{k^2}\right)  \d x +\frac{C}{k} \left(\|\v\|^2_{\H(\mathcal{O})}+\|\v\|^2_{\V(\mathcal{O})}\right)\nonumber\\&\leq -\mu \int_{\mathcal{O}}|\nabla\v|^2 \uprho\left(\frac{|x|^2}{k^2}\right)  \d x +\frac{C}{k}\|\v\|^2_{\V(\mathcal{O})},
	\end{align}
\begin{align}
	\y^2\ b\left(\v+\textbf{g}\y,\textbf{g},\uprho\left(\frac{|x|^2}{k^2}\right)\textbf{g}\right)&=\y^2\int_{\mathcal{O}} \uprho'\left(\frac{|x|^2}{k^2}\right)\frac{x}{k^2}\cdot(\v+\textbf{g}\y) |\textbf{g}|^2 \d x\nonumber\\&=  \y^2\int\limits_{k\leq|x|\leq \sqrt{2}k} \uprho'\left(\frac{|x|^2}{k^2}\right)\frac{x}{k^2}\cdot(\v+\textbf{g}\y) |\textbf{g}|^2 \d x\nonumber\\&\leq\y^2 \frac{\sqrt{2}}{k} \int\limits_{k\leq|x|\leq \sqrt{2}k} \left|\uprho'\left(\frac{|x|^2}{k^2}\right)\right| |\v+\textbf{g}\y| |\textbf{g}|^2 \d x\nonumber\\&\leq\y^2 \frac{C}{k}\|\v+\textbf{g}\y\|_{\H(\mathcal{O})}\|\textbf{g}\|^2_{\wi\L^4(\mathcal{O})}\nonumber\\&\leq\frac{C}{k}\left[\|\v+\textbf{g}\y\|^2_{\H(\mathcal{O})}+\y^4\|\textbf{g}\|^{4}_{\wi \L^{4}(\mathcal{O})}\right]\nonumber\\&\leq\frac{C}{k}\left[\|\v\|^2_{\V(\mathcal{O})}+\y^2\|\textbf{g}\|^2_{\H(\mathcal{O})}+\y^4\|\textbf{g}\|^{4}_{\wi \L^{4}(\mathcal{O})}\right],
\end{align}
and
	\begin{align}\label{ep3}
		&-b\left(\v+\textbf{g}\y,\v+\textbf{g}\y,\uprho\left(\frac{|x|^2}{k^2}\right)(\v+\textbf{g}\y)\right)\nonumber\\&=\int_{\mathcal{O}} \uprho'\left(\frac{|x|^2}{k^2}\right)\frac{x}{k^2}\cdot(\v+\textbf{g}\y) |\v+\textbf{g}\y|^2 \d x\nonumber\\&=  \int\limits_{k\leq|x|\leq \sqrt{2}k} \uprho'\left(\frac{|x|^2}{k^2}\right)\frac{x}{k^2}\cdot(\v+\textbf{g}\y) |\v+\textbf{g}\y|^2 \d x\nonumber\\&\leq \frac{\sqrt{2}}{k} \int\limits_{k\leq|x|\leq \sqrt{2}k} \left|\uprho'\left(\frac{|x|^2}{k^2}\right)\right| |\v+\textbf{g}\y|^3 \d x\nonumber\\&\leq \frac{C}{k}\|\v+\textbf{g}\y\|^3_{\wi \L^3(\mathcal{O})}\leq\frac{C}{k}\left[\|\v+\textbf{g}\y\|^2_{\H(\mathcal{O})}+\|\v+\textbf{g}\y\|^{4}_{\wi \L^{4}(\mathcal{O})}\right]\nonumber\\&
		\leq\frac{C}{k}\left[\|\v\|^2_{\V(\mathcal{O})}+\y^2\|\textbf{g}\|^2_{\H(\mathcal{O})}+\|\v\|^2_{\H(\mathcal{O})}\|\v\|^2_{\V(\mathcal{O})}+\y^4\|\textbf{g}\|^{4}_{\wi \L^{4}(\mathcal{O})}\right],
	\end{align}
where we have used interpolation and Young's inequalities in the last inequality. Using integration by parts, divergence free condition and \eqref{p-value}, we obtain 
\begin{align}
	&-\int_{\mathcal{O}}(\nabla p)\uprho\left(\frac{|x|^2}{k^2}\right)\v\d x\nonumber\\&=\int_{\mathcal{O}}p\uprho'\left(\frac{|x|^2}{k^2}\right)\frac{2}{k^2}(x\cdot\v) \leq\frac{C}{k} \int\limits_{k\leq|x|\leq \sqrt{2}k}\left|p\right|\left|\v\right|\d x\nonumber\\&\leq \frac{C}{k}\bigg[\|\v+\textbf{g}\y\|^2_{\wi\L^{4}(\mathcal{O})}\|\v\|_{\wi\L^2(\mathcal{O})}+\|\v+\textbf{g}\y\|^{r}_{\wi\L^{r+1}(\mathcal{O})}\|\v\|_{\wi\L^{r+1}(\mathcal{O})}\bigg]\nonumber\\&\leq \frac{C}{k}\bigg[\|\v+\textbf{g}\y\|^4_{\wi\L^{4}(\mathcal{O})}+\|\v\|^2_{\V(\mathcal{O})}+\|\v+\textbf{g}\y\|^{r+1}_{\wi\L^{r+1}(\mathcal{O})}+\|\v\|^{r+1}_{\wi\L^{r+1}(\mathcal{O})}\bigg]\nonumber\\&\leq \frac{C}{k}\bigg[\|\v\|^2_{\H(\mathcal{O})}\|\v\|^2_{\V(\mathcal{O})}+\y^4\|\textbf{g}\|^4_{\wi\L^{4}(\mathcal{O})}+\|\v\|^2_{\V(\mathcal{O})}+\|\v+\textbf{g}\y\|^{r+1}_{\wi\L^{r+1}(\mathcal{O})}+\y^{r+1}\|\g\|^{r+1}_{\wi\L^{r+1}(\mathcal{O})}\bigg].
\end{align}
 Finally, we estimate the remaining terms of \eqref{ep1} by using H\"older's and Young's inequalities as follows, for $r>1$,
\begin{align}
	&\y b\left(\v+\textbf{g}\y,\v,\uprho\left(\frac{|x|^2}{k^2}\right)\textbf{g}\right)\nonumber\\&\leq \y \left[\int_{\mathcal{O}}\uprho\left(\frac{|x|^2}{k^2}\right)\left|\v+\textbf{g}\y\right|^{r+1}\d x\right]^{\frac{1}{r+1}}\left[\int_{\mathcal{O}} \uprho\left(\frac{|x|^2}{k^2}\right)|\nabla\v|^2  \d x\right]^{\frac{1}{2}}\left[\int_{\mathcal{O}}\uprho\left(\frac{|x|^2}{k^2}\right)\left|\textbf{g}\right|^{\frac{2(r+1)}{r-1}}\d x\right]^{\frac{r-1}{2(r+1)}}\nonumber\\&\leq\frac{\beta}{4}\int_{\mathcal{O}}\uprho\left(\frac{|x|^2}{k^2}\right)\left|\v+\textbf{g}\y\right|^{r+1}\d x+\frac{\mu}{2}\int_{\mathcal{O}} \uprho\left(\frac{|x|^2}{k^2}\right)|\nabla\v|^2  \d x\nonumber\\&\quad+C\y^{\frac{2(r+1)}{r-1}} \int_{\mathcal{O}}\uprho\left(\frac{|x|^2}{k^2}\right)\left|\textbf{g}\right|^{\frac{2(r+1)}{r-1}}\d x,\\
	&\beta\y\int_{\mathcal{O}}\left|\v+\textbf{g}\y\right|^{r-1}(\v+\textbf{g}\y)\uprho\left(\frac{|x|^2}{k^2}\right)\textbf{g}\d x  \nonumber\\&\leq\frac{\beta}{4} \int_{\mathcal{O}}\uprho\left(\frac{|x|^2}{k^2}\right)\left|\v+\textbf{g}\y\right|^{r+1}\d x+C\y^{r+1} \int_{\mathcal{O}}\uprho\left(\frac{|x|^2}{k^2}\right)|\textbf{g}(x)|^{r+1}\d x,\\
	&	\int_{\mathcal{O}}\f(x)\uprho\left(\frac{|x|^2}{k^2}\right)\v \d x\nonumber\\&\leq \frac{\alpha}{4} \int_{\mathcal{O}}\uprho\left(\frac{|x|^2}{k^2}\right)|\v|^2\d x +\frac{1}{\alpha} \int_{\mathcal{O}}\uprho\left(\frac{|x|^2}{k^2}\right)|\f(x)|^2\d x,\\
	&(\ell-\alpha)\y\int_{\mathcal{O}}\textbf{g}\uprho\left(\frac{|x|^2}{k^2}\right)\v\d x+\mu\y\int_{\mathcal{O}}(\Delta\textbf{g})\uprho\left(\frac{|x|^2}{k^2}\right)\v\d x\nonumber\\&\leq\frac{\alpha}{4} \int_{\mathcal{O}}\uprho\left(\frac{|x|^2}{k^2}\right)|\v|^2\d x +C\y^2 \int_{\mathcal{O}}\uprho\left(\frac{|x|^2}{k^2}\right)|\textbf{g}(x)|^2\d x+C\y^2 \int_{\mathcal{O}}\uprho\left(\frac{|x|^2}{k^2}\right)|\Delta\textbf{g}(x)|^2\d x.	\label{ep4}
\end{align}
	Combining \eqref{ep1}-\eqref{ep4}, we get
	\begin{align}\label{ep5}
		&	\frac{\d}{\d t} \int_{\mathcal{O}}\uprho\left(\frac{|x|^2}{k^2}\right)|\v|^2\d x+\alpha \int_{\mathcal{O}}\uprho\left(\frac{|x|^2}{k^2}\right)|\v|^2\d x\nonumber\\&+\mu \int_{\mathcal{O}}|\nabla\v|^2 \uprho\left(\frac{|x|^2}{k^2}\right)  \d x +\beta \int_{\mathcal{O}}\uprho\left(\frac{|x|^2}{k^2}\right)\left|\v+\textbf{g}\y\right|^{r+1}\d x \nonumber\\&\leq \frac{C}{k}\bigg[\|\v\|^2_{\V(\mathcal{O})}+\|\v\|^2_{\H(\mathcal{O})}\|\v\|^2_{\V(\mathcal{O})}+\y^2\|\textbf{g}\|^2_{\H(\mathcal{O})}+\y^4\|\textbf{g}\|^{4}_{\wi \L^{4}(\mathcal{O})}+\|\v+\textbf{g}\y\|^{r+1}_{\wi\L^{r+1}(\mathcal{O})}\nonumber\\&\qquad+\left|\y\right|^{r+1}\|\textbf{g}\|^{r+1}_{\wi\L^{r+1}(\mathcal{O})}\bigg]+C\y^{\frac{2(r+1)}{r-1}} \int_{\mathcal{O}}\uprho\left(\frac{|x|^2}{k^2}\right)\left|\textbf{g}\right|^{\frac{2(r+1)}{r-1}}\d x+C\y^{r+1} \int_{\mathcal{O}}\uprho\left(\frac{|x|^2}{k^2}\right)|\textbf{g}(x)|^{r+1}\d x\nonumber\\&\quad+C \int_{\mathcal{O}}\uprho\left(\frac{|x|^2}{k^2}\right)|\f(x)|^2\d x+C\y^2 \int_{\mathcal{O}}\uprho\left(\frac{|x|^2}{k^2}\right)|\textbf{g}(x)|^2\d x+C\y^2 \int_{\mathcal{O}}\uprho\left(\frac{|x|^2}{k^2}\right)|\Delta\textbf{g}(x)|^2\d x\nonumber\\&\leq \frac{C}{k}\bigg[\|\v\|^2_{\V(\mathcal{O})}+\|\v\|^2_{\H(\mathcal{O})}\|\v\|^2_{\V(\mathcal{O})}+\y^2\|\textbf{g}\|^2_{\H(\mathcal{O})}+\y^4\|\textbf{g}\|^{4}_{\wi \L^{4}(\mathcal{O})}+\|\v+\textbf{g}\y\|^{r+1}_{\wi\L^{r+1}(\mathcal{O})}\nonumber\\&\qquad+\left|\y\right|^{r+1}\|\textbf{g}\|^{r+1}_{\wi\L^{r+1}(\mathcal{O})}\bigg]+C\y^{\frac{2(r+1)}{r-1}} \int_{|x|\geq k}\left|\textbf{g}\right|^{\frac{2(r+1)}{r-1}}\d x+C\y^{r+1} \int_{|x|\geq k}|\textbf{g}(x)|^{r+1}\d x\nonumber\\&\quad+C \int_{|x|\geq k}|\f(x)|^2\d x+C\y^2 \int_{|x|\geq k}|\textbf{g}(x)|^2\d x+C\y^2 \int_{|x|\geq k}|\Delta\textbf{g}(x)|^2\d x.
	\end{align}
Applying variation of constant formula, for $s\in[-t,0]$, we find 
\begin{align*}
	&\int_{|x|\geq k}|\v(s,-t;\omega,\v_{0})|^2\d x+\mu\int_{-t}^{s}e^{\alpha\zeta} \int_{|x|\geq k}|\nabla\v(\zeta,-t;\omega,\v_{0})|^2 \d x \d\zeta\nonumber\\ & +\int_{-t}^{s}e^{\alpha\zeta} \int_{|x|\geq k}|\v(\zeta,-t;\omega,\v_{0})+\textbf{g}(x)\y(\theta_{\zeta}\omega)|^{r+1}\d x \d\zeta \nonumber\\&\leq e^{-\alpha t}\|\v_{0}\|^2_{\H(\mathcal{O})}+\frac{C}{k}\bigg[\sup_{\zeta\in[-t,0]}\|\v(\zeta,-t;\omega,\v_{0})\|^2_{\H(\mathcal{O})}\int_{-t}^{0}e^{\alpha\zeta}\|\v(\zeta,-t;\omega,\v_{0})\|^2_{\V(\mathcal{O})}\d\zeta\nonumber\\&\quad+\int_{-t}^{0}e^{\alpha\zeta}\|\v(\zeta,-t;\omega,\v_{0})\|^2_{\V(\mathcal{O})}\d\zeta+\int_{-t}^{0}e^{\alpha\zeta}\|\v(\zeta,-t;\omega,\v_{0})+\y(\theta_{\zeta}\omega)\g\|^{r+1}_{\wi\L^{r+1}(\mathcal{O})}\d\zeta\bigg]\nonumber\\&\quad+\frac{C}{k}\bigg[\|\textbf{g}\|^2_{\H(\mathcal{O})}\int_{-t}^{0}e^{\alpha\zeta}|\y(\theta_{\zeta}\omega)|^{2} \d\zeta+\|\textbf{g}\|^{4}_{\wi\L^{4}(\mathcal{O})}\int_{-t}^{0}e^{\alpha\zeta}|\y(\theta_{\zeta}\omega)|^{4} \d\zeta\nonumber\\&\quad+\|\textbf{g}\|^{r+1}_{\wi\L^{r+1}(\mathcal{O})}\int_{-t}^{0}e^{\alpha\zeta}|\y(\theta_{\zeta}\omega)|^{r+1} \d\zeta\bigg]+C\bigg[ \int_{|x|\geq k}\left|\textbf{g}\right|^{\frac{2(r+1)}{r-1}}\d x\int_{-t}^{0}e^{\alpha\zeta}|\y(\theta_{\zeta}\omega)|^{\frac{2(r+1)}{r-1}}\d\zeta\nonumber\\&\quad+ \int_{|x|\geq k}|\textbf{g}(x)|^{r+1}\d x\int_{-t}^{0}e^{\alpha \zeta}|\y(\theta_{\zeta}\omega)|^{r+1}\d\zeta+ \int_{|x|\geq k}|\f(x)|^2\d x\int_{-t}^{0}e^{\alpha\zeta}\d\zeta\nonumber\\&\quad+ \int_{|x|\geq k}|\textbf{g}(x)|^2\d x\int_{-t}^{0}e^{\alpha \zeta}|\y(\theta_{\zeta}\omega)|^{2}\d\zeta+ \int_{|x|\geq k}|\Delta\textbf{g}(x)|^2\d x\int_{-t}^{0}e^{\alpha \zeta}|\y(\theta_{\zeta}\omega)|^{2}\d\zeta\bigg].
\end{align*}
Since $\v_{0}\in D(\theta_{-t}\omega)$, and $D\in\mathfrak{D}$, from \eqref{H6}, we infer that, for given $\varepsilon>0$, there exists $\mathscr{T}_1=\mathscr{T}_1(\omega,D,\varepsilon)>0$ such that for all $t\geq\mathscr{T}_1$,
\begin{align*}
	e^{-\alpha t}\|\v_{0}\|^2_{\H(\mathcal{O})}\leq\frac{\varepsilon}{3}.
\end{align*}
Due to Lemma \ref{Absorb}, we conclude that, for given $\varepsilon>0$, there exists $k_1=k_1(\omega,\varepsilon)>0$ and $\mathscr{T}_2=\mathscr{T}_2(\omega,D,\varepsilon)>0$ such that for all $k\geq k_1$ and $t\geq\mathscr{T}_2$,
\begin{align*}
	&\frac{C}{k}\bigg[\sup_{\zeta\in[-t,0]}\|\v(0,-t;\omega,\v_{0})\|^2_{\H(\mathcal{O})}\int_{-t}^{0}e^{\alpha\zeta}\|\v(\zeta,-t;\omega,\v_{0})\|^2_{\V(\mathcal{O})}\d\zeta\nonumber\\&\quad+\int_{-t}^{0}e^{\alpha\zeta}\|\v(\zeta,-t;\omega,\v_{0})\|^2_{\V(\mathcal{O})}\d\zeta+\int_{-t}^{0}e^{\alpha\zeta}\|\v(\zeta,-t;\omega,\v_{0})+\y(\theta_{\zeta}\omega)\g\|^{r+1}_{\wi\L^{r+1}(\mathcal{O})}\d\zeta\bigg]\leq \frac{\varepsilon}{3}.
\end{align*}
Since, $\textbf{g}\in\D(\A)$, $\f\in\H(\mathcal{O})$ and by \eqref{Y3}, we get 
\begin{align*}
	\int_{-\infty}^{0}e^{\alpha \zeta}\bigg[1+|\y(\theta_{\zeta}\omega)|^{2}+|\y(\theta_{\zeta}\omega)|^{4}+|\y(\theta_{\zeta}\omega)|^{\frac{2(r+1)}{r-1}}+|\y(\theta_{\zeta}\omega)|^{r+1}\bigg]\d\zeta<\infty,
\end{align*}
and that for given $\varepsilon>0$, there exists $k_2=k_2(\omega,\varepsilon)>0$ such that for all $k\geq k_2$,
\begin{align*}
	 &\frac{C}{k}\bigg[\|\textbf{g}\|^2_{\H(\mathcal{O})}\int_{-t}^{0}e^{\alpha\zeta}|\y(\theta_{\zeta}\omega)|^{2} \d\zeta+\|\textbf{g}\|^{4}_{\wi\L^{4}(\mathcal{O})}\int_{-t}^{0}e^{\alpha\zeta}|\y(\theta_{\zeta}\omega)|^{4} \d\zeta\nonumber\\&\quad+\|\textbf{g}\|^{r+1}_{\wi\L^{r+1}(\mathcal{O})}\int_{-t}^{0}e^{\alpha\zeta}|\y(\theta_{\zeta}\omega)|^{r+1} \d\zeta\bigg]+C\bigg[ \int_{|x|\geq k}\left|\textbf{g}\right|^{\frac{2(r+1)}{r-1}}\d x\int_{-t}^{0}e^{\alpha\zeta}|\y(\theta_{\zeta}\omega)|^{\frac{2(r+1)}{r-1}}\d\zeta\nonumber\\&\quad+ \int_{|x|\geq k}|\textbf{g}(x)|^{r+1}\d x\int_{-t}^{0}e^{\alpha \zeta}|\y(\theta_{\zeta}\omega)|^{r+1}\d\zeta+ \int_{|x|\geq k}|\f(x)|^2\d x\int_{-t}^{0}e^{\alpha\zeta}\d\zeta\nonumber\\&\quad+ \int_{|x|\geq k}|\textbf{g}(x)|^2\d x\int_{-t}^{0}e^{\alpha \zeta}|\y(\theta_{\zeta}\omega)|^{2}\d\zeta+ \int_{|x|\geq k}|\Delta\textbf{g}(x)|^2\d x\int_{-t}^{0}e^{\alpha \zeta}|\y(\theta_{\zeta}\omega)|^{2}\d\zeta\bigg]\leq \frac{\varepsilon}{3}.
\end{align*}
Let us take $\mathscr{T}_{\infty}=\mathscr{T}_{\infty}(\omega,D,\varepsilon)=\max\{\mathscr{T}_1,\mathscr{T}_2\}$ and $K_1(\omega,\varepsilon)=\max\{k_1,k_2\}$, so that the proof can be completed.
\end{proof}
The following Lemma can be proved in a similar way as Lemma \ref{largeradius}.
	\begin{lemma}\label{largeradius1}
		For $r>1$, assume that $\f_m\in\H(\mathcal{O}_m)$, $\textbf{g}_m\in\D(\A_m)$. Then, for any $\v_{m,0}\in D(\theta_{-t}\omega),$ where $D=\{D(\omega):\omega\in \tilde{\Omega}\}\in\mathfrak{D}$, and for any $\varepsilon>0$ and $\omega\in \tilde{\Omega}$, there exists $\mathscr{T}_m=\mathscr{T}_m(\omega,D,\varepsilon)\geq 0$ and $K_2(\omega,\varepsilon)>0$ such that for all $t\geq \mathscr{T}_m$ and $m\geq k\geq K_2$, the solution of \eqref{cscbf_m}  satisfies
	\begin{align}\label{ep*}
		&\int_{\mathcal{O}_m\backslash\mathcal{O}_k}|\v_m(0,-t;\omega,\v_{m,0}) |^2\d x+\mu\int_{-t}^{0}e^{\alpha\zeta} \int_{\mathcal{O}_m\backslash\mathcal{O}_k}|\nabla\v_m(\zeta,-t;\omega,\v_{m,0})|^2 \d x \d\zeta\nonumber\\& +\int_{-t}^{0}e^{\alpha\zeta} \int_{\mathcal{O}_m\backslash\mathcal{O}_k}|\v_m(\zeta,-t;\omega,\v_{m,0})+\textbf{g}(x)\y(\theta_{\zeta}\omega)|^{r+1}\d x \d\zeta\leq \varepsilon,
	\end{align}
and
\begin{align}
	\sup_{s\in[-t,0]}\left[\int_{\mathcal{O}_m\backslash\mathcal{O}_k}|\v_m(s,-t;\omega,\v_{m,0}) |^2\d x\right]\leq\varepsilon.
\end{align}
\end{lemma}
	\begin{lemma}\label{v_Conver}
	Assume that $\f\in \H(\mathcal{O})$. Let $\v_{m}$ and $\v$ be the solutions of \eqref{cscbf_m} and \eqref{cscbf_in} with initial data $\v_{m,0}$ and $\v_{0}$, respectively. If $\v_{m,0}\xrightharpoonup{w}\v_{0}$ in $\H$ as $n\to\infty$, then for every $\omega\in\tilde{\Omega}$,
	\begin{itemize}
		\item [(i)] $\widetilde{\v}_{m}(\xi,s;\omega,\v_{m,0})\xrightharpoonup{w}\v(\xi,s;\omega,\v_{0})$ in $\H$ for all $\xi\geq s$.
		\item [(ii)] $\widetilde{\v}_{m}(\cdot,s;\omega,\v_{m,0})\xrightharpoonup{w}\v(\cdot,s;\omega,\v_{0})$ in $\mathrm{L}^2((s,s+T);\V)$ for every $T>0$.
		\item [(iii)] $\widetilde{\v}_{m}(\cdot,s;\omega,\v_{m,0})\xrightharpoonup{w}\v(\cdot,s;\omega,\v_{0})$ in $\mathrm{L}^{r+1}((s,s+T);\widetilde{\L}^{r+1})$ for every $T>0$.
		\item [(iv)] $\widetilde{\v}_{m}(\cdot,s;\omega,\v_{m,0})\to\v(\cdot,s;\omega,\v_{0})$ in $\mathrm{L}^2((s,s+T);\L^2(\mathcal{O}_l))$ for every $T>0$ and $0<l\leq m$,
		where $\mathcal{O}_l=\{x\in\R^d:|x|<l\}$.
	\end{itemize}
\end{lemma}
\begin{proof}
	Proof is similar to Lemma 4.2, \cite{KM}, and hence we omit it here.
\end{proof}
\begin{lemma}\label{Converence_m}
	If $\u_{m,0}\in\mathcal{A}_m(\omega) ( \text{ for } m\in\N)$, then there exists an element $\u_0\in\mathcal{A}_{\infty}(\omega)$ such that upto a subsequence 
\begin{align}\label{CON3}
	\widetilde{\u}_{m,0}\xrightharpoonup{w}\u_{0} \text{ in }  \H(\mathcal{O}) \text{ as } m\to\infty.
\end{align}
\end{lemma}
\begin{proof}
	Since $\u_{m,0}\in\mathcal{A}_m(\omega)$ and attractor is a subset of absorbing set, we have that there exists a sequence $\u_{j,m,0}\in\mathcal{A}_{\infty}(\theta_{-j}\omega)$ such that $\u_{m,0}=\Phi_{m}(j,\theta_{-j}\omega)\u_{j,m,0}$ and 
	\begin{align*}
		\|\widetilde{\u}_{j,m,0}\|_{\H(\mathcal{O})}=\|\u_{j,m,0}\|_{\H(\mathcal{O}_m)}\leq \mathscr{L}^*(\theta_{-j}\omega),
	\end{align*}
which implies that there exists an element $\u_{j,0}\in\H(\mathcal{O})$ such that upto a subsequence
\begin{align}\label{WC1}
	\widetilde{\u}_{j,m,0}\xrightharpoonup{w}\u_{j,0} \ \text{ in }  \ \H(\mathcal{O}) \ \text{ as } \ m\to\infty.
\end{align}
From Lemma \ref{v_Conver}, we get 
\begin{align}\label{WC2}
	\widetilde{\Phi}_m(j,\theta_{-j}\omega)\widetilde{\u}_{j,m,0}\xrightharpoonup{w}\Phi(j,\theta_{-j}\omega)\u_{j,0} \ \text{ in } \ \H(\mathcal{O}).
\end{align}
It follows from the property of weak lower-semicontinuity of norm and \eqref{WC1} that
\begin{align*}
	\|\u_{j,0}\|_{\H(\mathcal{O})}\leq\liminf_{m\to\infty}\|\u_{j,m,0}\|_{\H(\mathcal{O}_m)}\leq \mathscr{L}^*(\theta_{-j}\omega),
\end{align*} 
which implies that $\u_{j,0}\in \mathcal{K}_{\infty}(\theta_{-j}\omega)$. Since $\Phi$ is asymptotically compact (Lemma 5.7, \cite{KM7}), for a sequence $t_n\to\infty$ as $n\to \infty$, the sequence $\{\Phi(t_n,\theta_{-t_n}\omega)\u_{t_n,0}\}_{n\in\N}$ has a convergent subsequence. Therefore, there exists an element $\u_0\in\mathcal{A}_{\infty}(\omega)$ such that $$\lim_{n \to \infty}\Phi(t_n,\theta_{-t_n}\omega)\u_{t_n,0}=\u_0.$$ From the invariance property of  the random attractors, we have $$\lim_{n \to \infty}\Phi_m(t_n,\theta_{-t_n}\omega)\u_{t_n,m,0}=\u_{m,0}$$ Hence, as $j\to\infty$, \eqref{WC2} gives  
\begin{align*}
	\widetilde{\u}_{m,0}\xrightharpoonup{w}\u_{0} \text{ in }  \H(\mathcal{O}),
\end{align*}
which completes the proof.
\end{proof}

\begin{lemma}\label{StrongC}
	For $r>1$, if $\u_{m,0}\in\mathcal{A}_m(\omega) (m\in\N)$, then there exists an element $\u_0\in\mathcal{A}_{\infty}(\omega)$ such that upto a subsequence 
	\begin{align}\label{CON4}
		\widetilde{\u}_{m,0}\to\u_{0} \text{ strongly in } \H(\mathcal{O}).
	\end{align}
\end{lemma}
\begin{proof}
	Since $\u_{m,0}\in\mathcal{A}_m(\omega)$, by \eqref{CON3}, there exists an element $\u_{0}\in\mathcal{A}_{\infty}(\omega)$  such that 
\begin{align}\label{SC0}
			\widetilde{\u}_{m,0}\xrightharpoonup{w}\u_{0} \text{ in } \H(\mathcal{O}).
\end{align}
Since $\u_{m,0}\in\mathcal{A}_m(\omega)$ and $\u_{0}\in\mathcal{A}_{\infty}(\omega)$, by invariance property of attractors, for all $j>0$ (for our purpose, we take $j$ sufficiently large so that Lemmas \ref{Absorb}-\ref{largeradius1} can be used properly in further calculations), there exists $\u_{j,m,0}\in\mathcal{A}_m(\theta_{-j}\omega)$ and $\u_{j,0}\in\mathcal{A}_{\infty}(\theta_{-j}\omega)$ such that
\begin{align*}
	\u_{m,0}=\Phi_{m}(j,\theta_{-j}\omega)\u_{j,m,0}=\u_m(0,-j;\omega,\u_{j,m,0})
\end{align*}
and 
\begin{align*}
	\u_{0}=\Phi(j,\theta_{-j}\omega)\u_{j,0}=\u_m(0,-j;\omega,\u_{j,0}).
\end{align*}
From \eqref{Phi_m}, we have
\begin{align*}
	\v_m(0,-j;\omega,\v_{j,m,0})=\u_{m}(0,-j;\omega,\u_{j,m,0})-\textbf{g}_m(x)\y(\omega),
\end{align*}
and
\begin{align*}
	\v(0,-j;\omega,\v_{j,0})=\u(0,-j;\omega,\u_{j,0})-\textbf{g}(x)\y(\omega),
\end{align*}
where
\begin{align*}
	\v_{j,m,0}=\u_{j,m,0}-\textbf{g}_m(x)\y(\theta_{-j}\omega)\ \ \text{ and }\ \  	\v_{j,0}=\u_{j,0}-\textbf{g}(x)\y(\theta_{-j}\omega).	
\end{align*}
Moreover, from \eqref{SC0}, we have 
\begin{align}\label{SC00}
	\widetilde{\v}_m(0,-j;\omega,\v_{j,m,0})\xrightharpoonup{w} \v(0,-j;\omega,\v_{j,0}) \ \text{ in }\  \H(\mathcal{O}).
\end{align}
From the fact that attractor is a subset of absorbing set, we have 
\begin{align*}
	\|\u_{j,m,0}\|_{\H(\mathcal{O}_m)}\leq \mathscr{L}^*(\theta_{-j}\omega),
\end{align*} 
which implies that there exists an element $\u_{j,0}\in\H(\mathcal{O})$ (due to  the uniqueness of weak limit we are taking element $\u_{j,0}$) such that 
\begin{align*}
	\widetilde{\u}_{j,m,0}\xrightharpoonup{w}\u_{j,0} \text{ in } \H(\mathcal{O}).	
\end{align*}
Consequently, we get
\begin{align}\label{SC1}
	\widetilde{\v}_{j,m,0}\xrightharpoonup{w}\v_{j,0} \text{ in } \H(\mathcal{O}).	
\end{align}
By \eqref{SC1} and Lemma \ref{v_Conver}, we have 
	\begin{equation}\label{weak_con}
	\left\{
	\begin{aligned}			\widetilde{\v}_m(\cdot,-j;\omega,\v_{j,m,0})&\xrightharpoonup{w}\v(\cdot,-j;\omega,\v_{j,0}) \text{ in } \mathrm{L}^2(-j,0;\H(\mathcal{O})),\\
	\widetilde{\v}_m(\cdot,-j;\omega,\v_{j,m,0})&\xrightharpoonup{w}\v(\cdot,-j;\omega,\v_{j,0}) \text{ in } \mathrm{L}^2(-j,0;\V(\mathcal{O})),\\
	\widetilde{\v}_m(\cdot,-j;\omega,\v_{j,m,0})&\xrightharpoonup{w}\v(\cdot,-j;\omega,\v_{j,0}) \text{ in } \mathrm{L}^{r+1}(-j,0;\wi\L^{r+1}(\mathcal{O})),\\
	\widetilde{\v}_m(\cdot,-j;\omega,\v_{j,m,0})&\to\v(\cdot,-j;\omega,\v_{j,0}) \text{ in } \mathrm{L}^2(-j,0;\L^2(\mathcal{O}_l)),
	\end{aligned}
	\right.
\end{equation} 	
	for every $l>0$, where $\mathcal{O}_l=\{x\in\R^d:|x|<l\}$. From \eqref{cscbf_m} together with \eqref{b0}, we obtain
\begin{align}\label{SC2}
	&\frac{\d}{\d t}\|\v_{m}\|^2_{\H(\mathcal{O}_m)}+2\alpha\|\v_{m}\|^2_{\H(\mathcal{O}_m)}+2\mu\|\v_{m}\|^2_{\V(\mathcal{O}_m)}+2\beta\|\v_{m}+\textbf{g}_m\y\|^{r+1}_{\wi\L^{r+1}(\mathcal{O}_m)} \nonumber\\&=2b_m(\v_{m}+\textbf{g}_m\y,\textbf{g}_m\y,\v_{m})+2\beta\left\langle\mathcal{C}_m(\v_{m}+\textbf{g}_m\y),\textbf{g}_m\y\right\rangle+2\left(\f_m,\v_{m}\right)\nonumber\\&\quad + 2\y\left(\left(\ell-\alpha\right)\textbf{g}_m-\mu\A\textbf{g}_m,\v_{m}\right). 
\end{align}
An application of variation of constant formula yields
\begin{align}\label{SC3}
	&\|\v_{m}(0,-j;\omega,\v_{j,m,0})\|^2_{\H(\mathcal{O}_m)}\nonumber\\&=I_1(m,j)+I_2(m,j)+I_3(m,j)+I_4(m,j)+I_5(m,j)+I_6(m,j)+I_7(m,j)+I_8(m,j),
\end{align} 
where 
\begin{align*}
	I_1(m,j)&=e^{-\alpha j}\|\v_{j,m,0}\|^2_{\H(\mathcal{O}_m)},\\
	I_2(m,j)&=-\alpha\int_{-j}^{0}e^{\alpha\xi}\|\v_{m}(\xi,-j;\omega,\v_{j,m,0})\|^2_{\H(\mathcal{O}_m)}\d\xi,\\
	I_3(m,j)&=-2\mu\int_{-j}^{0}e^{\alpha\xi}\|\v_{m}(\xi,-j;\omega,\v_{j,m,0})\|^2_{\V(\mathcal{O}_m)}\d\xi,\\
	I_4(m,j)&=-2\beta\int_{-j}^{0}e^{\alpha\xi}\|\v_{m}(\xi,-j;\omega,\v_{j,m,0})+\textbf{g}_m\y(\theta_{\xi}\omega)\|^{r+1}_{\wi\L^{r+1}(\mathcal{O}_m)}\d\xi,\\
	I_5(m,j)&=	2 \int_{-j}^{0} e^{\alpha\xi}b_m(\v_{m}(\xi,-j;\omega,\v_{j,m,0})+\textbf{g}_m\y(\theta_{\xi}\omega),\textbf{g}_m\y(\theta_{\xi}\omega),\v_{m}(\xi,-j;\omega,\v_{j,m,0}))\d\xi,\\
	I_6(m,j)&=2\beta \int_{-j}^{0} e^{\alpha\xi}\left\langle\mathcal{C}_m(\v_{m}(\xi,-j;\omega,\v_{j,m,0})+\textbf{g}_m\y(\theta_{\xi}\omega)),\textbf{g}_m\y(\theta_{\xi}\omega)\right\rangle\d\xi,\\
	I_7(m,j)&=2 \int_{-j}^{0} e^{\alpha\xi}\left(\f_m,\v_{m}(\xi,-j;\omega,\v_{j,m,0})\right)\d\xi,
\end{align*}
and
\begin{align*}
	I_8(m,j)&=2 \int_{-j}^{0} e^{\alpha\xi} \y(\theta_{\xi}\omega)\big((\ell-\alpha)\textbf{g}_m-\mu\A_m\textbf{g}_m,\v_{m}(\xi,-j;\omega,\v_{j,m,0})\big)\d\xi.
\end{align*}
Similar to \eqref{SC3}, by \eqref{cscbf_in}, it can also be obtained that 
\begin{align}\label{SC4}
	&\|\v(0,-j;\omega,\v_{j,0})\|^2_{\H(\mathcal{O})} = e^{-\alpha j}\|\v_{j,0}\|^2_{\H(\mathcal{O})} -\alpha\int_{-j}^{0}e^{\alpha \xi}\|\v(\xi,-j;\omega,\v_{j,0})\|^2_{\H(\mathcal{O})}\d\xi\nonumber\\&\quad-2\mu\int_{-j}^{0}e^{\alpha \xi}\|\v(\xi,-j;\omega,\v_{j,0})\|^2_{\V(\mathcal{O})}\d\xi-2\beta\int_{-j}^{0}e^{\alpha \xi}\|\v(\xi,-j;\omega,\v_{j,0})+\textbf{g}\y(\theta_{\xi}\omega)\|^{r+1}_{\wi \L^{r+1}(\mathcal{O})}\d\xi\nonumber\\&\quad+2 \int_{-j}^{0} e^{\alpha\xi}b_{\infty}(\v(\xi,-j;\omega,\v_{j,0})+\textbf{g}\y(\theta_{\xi}\omega),\textbf{g}\y(\theta_{\xi}\omega),\v(\xi,-j;\omega,\v_{j,0}))\d\xi\nonumber\\&\quad+2\beta \int_{-j}^{0} e^{\alpha\xi}\left\langle\mathcal{C}_{\infty}(\v(\xi,-j;\omega,\v_{j,0})+\textbf{g}\y(\theta_{\xi}\omega)),\textbf{g}\y(\theta_{\xi}\omega)\right\rangle\d\xi\nonumber\\&\quad+2 \int_{-j}^{0} e^{\alpha\xi}\left(\f,\v(\xi,-j;\omega,\v_{j,0})\right)\d\xi\nonumber\\&\quad+2 \int_{-j}^{0} e^{\alpha\xi}\y(\theta_{\xi}\omega)\big((\ell-\alpha)\textbf{g}-\mu\A\textbf{g},\v(\xi,-j;\omega,\v_{j,0})\big)\d\xi.
\end{align}
Using the convergences obtained in \eqref{weak_con}, we have
	\begin{equation}\label{SC5}
	\left\{
	\begin{aligned}	
	\limsup_{m\to\infty}I_2(m,j)&\leq -\alpha\int_{-j}^{0}e^{\alpha \xi}\|\v(\xi,-j;\omega,\v_{j,0})\|^2_{\H(\mathcal{O})}\d\xi,\\
	\limsup_{m\to\infty}I_3(m,j)&\leq -2\mu\int_{-j}^{0}e^{\alpha \xi}\|\v(\xi,-j;\omega,\v_{j,0})\|^2_{\V(\mathcal{O})}\d\xi,\\
	\limsup_{m\to\infty}I_4(m,j) &\leq -2\beta\int_{-j}^{0}e^{\alpha \xi}\|\v(\xi,-j;\omega,\v_{j,0})+\textbf{g}\y(\theta_{\xi}\omega)\|^{r+1}_{\wi \L^{r+1}(\mathcal{O})}\d\xi,\\
	\lim_{m \to \infty}I_7(m,j)&=2 \int_{-j}^{0} e^{\alpha\xi}\left(\f,\v(\xi,-j;\omega,\v_{j,0})\right)\d\xi,\\
\lim_{m \to \infty}I_8(m,j)&=2 \int_{-j}^{0} e^{\alpha\xi} \y(\theta_{\xi}\omega)\big((\ell-\alpha)\textbf{g}-\mu\A\textbf{g},\v(\xi,-j;\omega,\v_{j,0})\big)\d\xi.
	\end{aligned}
\right.
\end{equation} 	
For $I_1(m,j)$, from the definition of absorbing set, we have
\begin{align}\label{SC6}
	I_1(m,j)\leq e^{-\alpha j}\mathscr{L}(\theta_{-j}\omega),
\end{align}
where $\mathscr{L}(\omega)$ is the same as in \eqref{H0^*}. Finally, we find the convergence for $I_{5}(m,j)$ and $I_{6}(m,j)$ using the Lemmas \ref{Absorb}-\ref{largeradius1}. For convenience, we write 
\begin{align*}
	\v_m(\xi):=\v_m(\xi,-j;\omega,\v_{j,m,0}) \ \text{ and }\ \v(\xi):=\v(\xi,-j;\omega,\v_{j,0}).
\end{align*}
Now, choose some $k$ sufficiently large so that Lemmas \ref{largeradius} and \ref{largeradius1} can be used in further calculations and consider $m\geq k$. For $I_5(m,j)$, we consider
\begin{align}\label{SC7}
	&\left|I_{5}(m,j)-2 \int_{-j}^{0} e^{\alpha\xi}b_{\infty}(\v(\xi,-j;\omega,\v_{j,0})+\textbf{g}\y(\theta_{\xi}\omega),\textbf{g}\y(\theta_{\xi}\omega),\v(\xi,-j;\omega,\v_{j,0}))\d\xi\right|\nonumber\\&\leq2\int_{-j}^{0}e^{\alpha\xi}\left|\y(\theta_{\xi}\omega)\right|\int_{\mathcal{O}_k}\big|\big((\v_m(x,\xi)+\textbf{g}(x)\y(\theta_{\xi}\omega))\cdot\nabla\big)\textbf{g}(x)\cdot\v_m(x,\xi)\nonumber\\&\qquad\qquad-\big((\v(x,\xi)+\textbf{g}(x)\y(\theta_{\xi}\omega))\cdot\nabla\big)\textbf{g}(x)\cdot\v(x,\xi)\big|\d x\d\xi\nonumber\\&\quad+2\int_{-j}^{0}e^{\alpha\xi}\left|\y(\theta_{\xi}\omega)\right|\int_{\mathcal{O}_m\backslash\mathcal{O}_k}\big|\big((\v_m(x,\xi)+\textbf{g}(x)\y(\theta_{\xi}\omega))\cdot\nabla\big)\textbf{g}(x)\cdot\v_m(x,\xi)\big|\d x\d\xi\nonumber\\&\quad+2\int_{-j}^{0}e^{\alpha\xi}\left|\y(\theta_{\xi}\omega)\right|\int_{\mathcal{O}\backslash\mathcal{O}_k}\big|\big((\v(x,\xi)+\textbf{g}(x)\y(\theta_{\xi}\omega))\cdot\nabla\big)\textbf{g}(x)\cdot\v(x,\xi)\big|\d x\d\xi\nonumber\\&\leq2\int_{-j}^{0}e^{\alpha\xi}\left|\y(\theta_{\xi}\omega)\right|\int_{\mathcal{O}_k}\big|\big((\v_m(x,\xi)-\v(x,\xi))\cdot\nabla\big)\textbf{g}(x)\cdot\v_m(x,\xi)\big|\d x\d\xi\nonumber\\&\quad+2\int_{-j}^{0}e^{\alpha\xi}\left|\y(\theta_{\xi}\omega)\right|\int_{\mathcal{O}_k}\big|\big(\v(x,\xi)\cdot\nabla\big)\textbf{g}(x)\cdot(\v_m(x,\xi)-\v(x,\xi))\big|\d x\d\xi\nonumber\\&\quad+2\int_{-j}^{0}e^{\alpha\xi}\left|\y(\theta_{\xi}\omega)\right|^2\int_{\mathcal{O}_k}\big|\big(\textbf{g}(x)\cdot\nabla\big)\textbf{g}(x)\cdot(\v_m(x,\xi)-\v(x,\xi))\big|\d x\d\xi\nonumber\\&\quad+2\int_{-j}^{0}e^{\alpha\xi}\left|\y(\theta_{\xi}\omega)\right|\int_{\mathcal{O}_m\backslash\mathcal{O}_k}\big|\big(\v_m(x,\xi)\cdot\nabla\big)\textbf{g}(x)\cdot\v_m(x,\xi)\big|\d x\d\xi\nonumber\\&\quad+2\int_{-j}^{0}e^{\alpha\xi}\left|\y(\theta_{\xi}\omega)\right|^2\int_{\mathcal{O}_m\backslash\mathcal{O}_k}\big|\big(\textbf{g}(x)\cdot\nabla\big)\textbf{g}(x)\cdot\v_m(x,\xi)\big|\d x\d\xi\nonumber\\&\quad+2\int_{-j}^{0}e^{\alpha\xi}\left|\y(\theta_{\xi}\omega)\right|\int_{\mathcal{O}\backslash\mathcal{O}_k}\big|\big(\v(x,\xi)\cdot\nabla\big)\textbf{g}(x)\cdot\v(x,\xi)\big|\d x\d\xi\nonumber\\&\quad+2\int_{-j}^{0}e^{\alpha\xi}\left|\y(\theta_{\xi}\omega)\right|^2\int_{\mathcal{O}\backslash\mathcal{O}_k}\big|\big(\textbf{g}(x)\cdot\nabla\big)\textbf{g}(x)\cdot\v(x,\xi)\big|\d x\d\xi=:\sum_{i=1}^{7}I_{5i}(m,j,k).
\end{align}
Using H\"older's inequality and \eqref{lady}, we obtain
\begin{align*}
	&\left|I_{51}(m,j,k)\right|\nonumber\\&\leq C\|\textbf{g}\|_{\V(\mathcal{O})}\int_{-j}^{0}e^{\alpha\xi}\left|\y(\theta_{\xi}\omega)\right|\|\v_m(\xi)-\v(\xi)\|^{\frac{1}{2}}_{\L^2(\mathcal{O}_k)}\|\nabla(\v_m(\xi)-\v(\xi))\|^{\frac{1}{2}}_{\L^2(\mathcal{O}_k)}\nonumber\\&\qquad\times\|\v_m(\xi)\|^{\frac{1}{2}}_{\H(\mathcal{O}_m)}\|\v_m(\xi)\|^{\frac{1}{2}}_{\V(\mathcal{O}_m)}\d\xi\nonumber\\&\leq C\|\textbf{g}\|_{\V(\mathcal{O})}\int_{-j}^{0}e^{\alpha\xi}\left|\y(\theta_{\xi}\omega)\right|\|\v_m(\xi)-\v(\xi)\|^{\frac{1}{2}}_{\L^2(\mathcal{O}_k)}\|\v_m(\xi)\|^{\frac{1}{2}}_{\H(\mathcal{O}_m)}\|\v_m(\xi)\|_{\V(\mathcal{O}_m)}\d\xi\nonumber\\&\ \ +C\|\textbf{g}\|_{\V(\mathcal{O})}\int_{-j}^{0}e^{\alpha\xi}\left|\y(\theta_{\xi}\omega)\right|\|\v_m(\xi)-\v(\xi)\|^{\frac{1}{2}}_{\L^2(\mathcal{O}_k)}\|\v(\xi)\|^{\frac{1}{2}}_{\V(\mathcal{O})}\|\v_m(\xi)\|^{\frac{1}{2}}_{\H(\mathcal{O}_m)}\|\v_m(\xi)\|^{\frac{1}{2}}_{\V(\mathcal{O}_m)}\d\xi\nonumber\\&\leq C\|\textbf{g}\|_{\V(\mathcal{O})}\sup_{\xi\in[-j,0]}\left[\|\v_m(\xi)\|^{\frac{1}{2}}_{\H(\mathcal{O}_m)}\right]\bigg(\int_{-j}^{0}\|\v_m(\xi)-\v(\xi)\|^{2}_{\L^2(\mathcal{O}_k)}\d\xi\bigg)^{\frac{1}{4}}\nonumber\\&\qquad\times\bigg(\int_{-j}^{0}e^{\alpha\xi}\|\v_m(\xi)\|^{2}_{\V(\mathcal{O}_m)}\d\xi\bigg)^{\frac{1}{2}}\bigg(\int_{-j}^{0}e^{2\alpha\xi}\left|\y(\theta_{\xi}\omega)\right|^4\d\xi\bigg)^{\frac{1}{4}}\nonumber\\&\quad+C\|\textbf{g}\|_{\V(\mathcal{O})}\sup_{\xi\in[-j,0]}\left[\|\v_m(\xi)\|^{\frac{1}{2}}_{\H(\mathcal{O}_m)}\right]\bigg(\int_{-j}^{0}\|\v_m(\xi)-\v(\xi)\|^{2}_{\L^2(\mathcal{O}_k)}\d\xi\bigg)^{\frac{1}{4}}\nonumber\\&\qquad\times\bigg(\int_{-j}^{0}e^{\alpha\xi}\|\v(\xi)\|^{2}_{\V(\mathcal{O})}\d\xi\bigg)^{\frac{1}{4}}\bigg(\int_{-j}^{0}e^{\alpha\xi}\|\v_m(\xi)\|^{2}_{\V(\mathcal{O}_m)}\d\xi\bigg)^{\frac{1}{4}}\bigg(\int_{-j}^{0}e^{2\alpha\xi}\left|\y(\theta_{\xi}\omega)\right|^4\d\xi\bigg)^{\frac{1}{4}}.
\end{align*}
For sufficiently large $j>0$, making use of Lemmas \ref{Absorb} and \ref{Absorb1}, and strong convergence in \eqref{weak_con}, we obtain
\begin{align}\label{SC8}
	I_{51}(m,j,k)\to0 \  \text{ as }\  m\to \infty.
\end{align}
Similarly, we deduce 
	\begin{equation}\label{SC9}
	\left\{
	\begin{aligned}	
		I_{52}(m,j,k)&\to0\  \text{ as }\ m\to \infty,\\
		I_{53}(m,j,k)&\to0\  \text{ as }\  m\to \infty.
	\end{aligned}
	\right.
\end{equation} 	
Again, using H\"older's inequality and \eqref{lady}, we get 
\begin{align*}
	&\left|I_{54}(m,j,k)\right|\nonumber\\&\leq C\int_{-j}^{0}e^{\alpha\xi}\left|\y(\theta_{\xi}\omega)\right|\bigg(\int_{\mathcal{O}_m\backslash\mathcal{O}_k}\big|\v_m(x,\xi)\big|^2\d x\bigg)^{\frac{1}{2}}\bigg(\int_{\mathcal{O}_m\backslash\mathcal{O}_k}\big|\nabla\textbf{g}(x)\big|^2\d x\bigg)^{\frac{1}{2}}\nonumber\\&\qquad\times\bigg(\int_{\mathcal{O}_m\backslash\mathcal{O}_k}\big|\nabla\v_m(x,\xi)\big|^2\d x\bigg)^{\frac{1}{2}}\d\xi\nonumber\\& \leq C\bigg(\int_{\mathcal{O}_m\backslash\mathcal{O}_k}\big|\nabla\textbf{g}(x)\big|^2\d x\bigg)^{\frac{1}{2}}\sup_{\xi\in[-j,0]}\bigg[\bigg(\int_{\mathcal{O}_m\backslash\mathcal{O}_k}\big|\v_m(x,\xi)\big|^2\d x\bigg)^{\frac{1}{2}}\bigg]\bigg(\int_{-j}^{0}e^{\alpha\xi}\left|\y(\theta_{\xi}\omega)\right|^2\d\xi\bigg)^{\frac{1}{2}}\nonumber\\&\qquad\times\bigg(\int_{-j}^{0}e^{\alpha\xi}\int_{\mathcal{O}_m\backslash\mathcal{O}_k}\big|\nabla\v_m(x,\xi)\big|^2\d x\d\xi\bigg)^{\frac{1}{2}}.
\end{align*}
Lemma \ref{largeradius1} and $\textbf{g}\in\D(\A)$ imply that, for every $\varepsilon>0$, there exists some $j>0$ and $k>0$ such that
\begin{align}\label{SC10}
	I_{54}(m,j,k)\leq \frac{\varepsilon}{4}.
\end{align}
Similarly, with the help of Lemma \ref{largeradius} and \ref{largeradius1}, we have 
	\begin{equation}\label{SC11}
	\left\{
	\begin{aligned}	
		I_{55}(m,j,k)&\leq \frac{\varepsilon}{4},\\
		I_{56}(m,j,k)&\leq \frac{\varepsilon}{4},\\
		I_{57}(m,j,k)&\leq \frac{\varepsilon}{4}.
	\end{aligned}
	\right.
\end{equation} 	
From \eqref{SC7}-\eqref{SC11}, we find
\begin{align*}
	\lim_{m \to \infty}\left|I_{5}(m,j)-2 \int_{-j}^{0} e^{\alpha\xi}b_{\infty}(\v(\xi,-j;\omega,\v_{j,0})+\textbf{g}\y(\theta_{\xi}\omega),\textbf{g}\y(\theta_{\xi}\omega),\v(\xi,-j;\omega,\v_{j,0}))\d\xi\right|\leq\varepsilon.
\end{align*}
Since $\varepsilon$ is arbitrary, taking $\varepsilon\to 0$, we obtain
\begin{align}\label{SC12}
	\lim_{m \to \infty}I_{5}(m,j)=2 \int_{-j}^{0} e^{\alpha\xi}b_{\infty}(\v(\xi,-j;\omega,\v_{j,0})+\textbf{g}\y(\theta_{\xi}\omega),\textbf{g}\y(\theta_{\xi}\omega),\v(\xi,-j;\omega,\v_{j,0}))\d\xi.
\end{align}
For $I_6(m,j)$, we consider
\begin{align}\label{SC13}
	&\left|I_6(m,j)-2\beta \int_{-j}^{0} e^{\alpha\xi}\left\langle\mathcal{C}_{\infty}(\v(\xi,-j;\omega,\v_{j,0})+\textbf{g}\y(\theta_{\xi}\omega)),\textbf{g}\y(\theta_{\xi}\omega)\right\rangle\d\xi\right|\nonumber\\&\leq2\beta\bigg|\int_{-j}^{0}e^{\alpha\xi}\y(\theta_{\xi}\omega)\int_{\mathcal{O}_k}\big[|\v_m(x,\xi)+\textbf{g}(x)\y(\theta_{\xi}\omega)|^{r-1}\big(\v_m(x,\xi)+\textbf{g}(x)\y(\theta_{\xi}\omega)\big)\nonumber\\&\qquad-|\v(x,\xi)+\textbf{g}(x)\y(\theta_{\xi}\omega)|^{r-1}\big(\v(x,\xi)+\textbf{g}(x)\y(\theta_{\xi}\omega)\big)\big]\textbf{g}(x)\d x\d\xi\bigg| \nonumber\\&\quad+2\beta\int_{-j}^{0}e^{\alpha\xi}\left|\y(\theta_{\xi}\omega)\right|\int_{\mathcal{O}_m\backslash\mathcal{O}_k}|\v_m(x,\xi)+\textbf{g}(x)\y(\theta_{\xi}\omega)|^{r}\left|\textbf{g}(x)\right|\d x\d\xi\nonumber\\&\quad+2\beta\int_{-j}^{0}e^{\alpha\xi}\left|\y(\theta_{\xi}\omega)\right|\int_{\mathcal{O}\backslash\mathcal{O}_k}|\v(x,\xi)+\textbf{g}(x)\y(\theta_{\xi}\omega)|^{r}\left|\textbf{g}(x)\right|\d x\d\xi\nonumber\\&\quad=:I_{61}(m,j,k)+I_{62}(m,j,k)+I_{63}(m,j,k).
\end{align}
Using Taylor's formula (Theorem 7.9.1, \cite{PGC}), \eqref{29}, Lemmas \ref{Absorb}-\ref{Absorb1} and the final convergence in \eqref{weak_con}, we get 
\begin{align}\label{SC14}
	&I_{61}(m,j,k)\nonumber\\&\leq C\int_{-j}^{0}e^{\alpha\xi}\left|\y(\theta_{\xi}\omega)\right|\int_{\mathcal{O}_k}\big[|\v_m(x,\xi)+\textbf{g}(x)\y(\theta_{\xi}\omega)|^{r-1}+|\v(x,\xi)+\textbf{g}(x)\y(\theta_{\xi}\omega)|^{r-1}\big]\nonumber\\&\qquad\times\left|\v_m(x,\xi)-\v(x,\xi)\right|\left|\textbf{g}(x)\right|\d x\d\xi\nonumber\\&\leq C\int_{-j}^{0}e^{\alpha\xi}\left|\y(\theta_{\xi}\omega)\right|\left[\|\v_m(\xi)+\textbf{g}_m\y(\theta_{\xi}\omega)\|^{r-1}_{\wi\L^{r+1}(\mathcal{O}_m)}+\|\v(\xi)+\textbf{g}\y(\theta_{\xi}\omega)\|^{r-1}_{\wi\L^{r+1}(\mathcal{O})}\right]\nonumber\\&\qquad\times\|\v_m(\xi)-\v(\xi)\|_{\L^2(\mathcal{O}_k)} \|\textbf{g}\|_{\wi\L^{\frac{2(r+1)}{3-r}}(\mathcal{O})}\d \xi\nonumber\\&\leq C\|\textbf{g}\|_{\wi\L^{\frac{2(r+1)}{3-r}}(\mathcal{O})}\bigg[\int_{-j}^{0}e^{\frac{4\alpha\xi}{3-r}}\left|\y(\theta_{\xi}\omega)\right|^{\frac{2(r+1)}{3-r}}\bigg]^{\frac{3-r}{2(r+1)}}\bigg[\bigg(\int_{-j}^{0}e^{\alpha\xi}\|\v_m(\xi)+\textbf{g}_m\y(\theta_{\xi}\omega)\|^{r+1}_{\wi\L^{r+1}(\mathcal{O}_m)}\d\xi\bigg)^{\frac{r-1}{r+1}}\nonumber\\&\qquad+\bigg(\int_{-j}^{0}e^{\alpha\xi}\|\v(\xi)+\textbf{g}\y(\theta_{\xi}\omega)\|^{r+1}_{\wi\L^{r+1}(\mathcal{O})}\d\xi\bigg)^{\frac{r-1}{r+1}}\bigg]\bigg(\int_{-j}^{0}\|\v_m(\xi)-\v(\xi)\|^{2}_{\L^2(\mathcal{O}_k)}\d\xi\bigg)^{\frac{1}{2}}
\nonumber\\&
\to 0  \ \text{ as }\   m\to \infty, 
\end{align}
for $ 1<r<3$, and
\begin{align}\label{SC15}
	&I_{61}(m,j,k)\nonumber\\&\leq C\int_{-j}^{0}e^{\alpha\xi}\left|\y(\theta_{\xi}\omega)\right|\int_{\mathcal{O}_k}\big[|\v_m(x,\xi)+\textbf{g}(x)\y(\theta_{\xi}\omega)|^{r-1}+|\v(x,\xi)+\textbf{g}(x)\y(\theta_{\xi}\omega)|^{r-1}\big]\nonumber\\&\qquad\times\left|\v_m(x,\xi)-\v(x,\xi)\right|\left|\textbf{g}(x)\right|\d x\d\xi\nonumber\\&\leq C\int_{-j}^{0}e^{\alpha\xi}\left|\y(\theta_{\xi}\omega)\right|\left[\|\v_m(\xi)+\textbf{g}\y(\theta_{\xi}\omega)\|^{r-1}_{\wi\L^{r+1}(\mathcal{O}_m)}+\|\v(\xi)+\textbf{g}\y(\theta_{\xi}\omega)\|^{r-1}_{\wi\L^{r+1}(\mathcal{O})}\right]\nonumber\\&\qquad\times\|\v_m(\xi)-\v(\xi)\|^{\frac{1}{r-1}}_{\L^2(\mathcal{O}_k)}\|\v_m(\xi)-\v(\xi)\|^{\frac{r-2}{r-1}}_{\L^{r+1}(\mathcal{O}_k)} \|\textbf{g}\|_{\wi\L^{2(r+1)}(\mathcal{O})}\d \xi\nonumber\\&\leq C\|\textbf{g}\|_{\wi\L^{2(r+1)}(\mathcal{O})}\int_{-j}^{0}e^{\alpha\xi}\left|\y(\theta_{\xi}\omega)\right|\left[\|\v_m(\xi)+\textbf{g}_m\y(\theta_{\xi}\omega)\|^{r-1}_{\wi\L^{r+1}(\mathcal{O}_m)}+\|\v(\xi)+\textbf{g}\y(\theta_{\xi}\omega)\|^{r-1}_{\wi\L^{r+1}(\mathcal{O})}\right]\nonumber\\&\qquad\times\left[\|\v_m(\xi)+\textbf{g}_m\y(\theta_{\xi}\omega)\|^{\frac{r-2}{r-1}}_{\wi\L^{r+1}(\mathcal{O}_m)}+\|\v(\xi)+\textbf{g}\y(\theta_{\xi}\omega)\|^{\frac{r-2}{r-1}}_{\wi\L^{r+1}(\mathcal{O})}\right]\|\v_m(\xi)-\v(\xi)\|^{\frac{1}{r-1}}_{\L^2(\mathcal{O}_k)}\d \xi\nonumber\\&\leq C\|\textbf{g}\|_{\wi\L^{2(r+1)}(\mathcal{O})}\bigg[\int\limits_{-j}^{0}e^{\frac{2r\alpha\xi}{r-1}}\left|\y(\theta_{\xi}\omega)\right|^{2(r+1)}\d\xi\bigg]^{\frac{1}{2(r+1)}}\bigg[\bigg(\int\limits_{-j}^{0}e^{\alpha\xi}\|\v_m(\xi)+\textbf{g}_m\y(\theta_{\xi}\omega)\|^{r+1}_{\wi\L^{r+1}(\mathcal{O}_m)}\d\xi\bigg)^{\frac{r-1}{r+1}}\nonumber\\&\ \ +\bigg(\int\limits_{-j}^{0}e^{\alpha\xi}\|\v(\xi)+\textbf{g}\y(\theta_{\xi}\omega)\|^{r+1}_{\wi\L^{r+1}(\mathcal{O})}\d\xi\bigg)^{\frac{r-1}{r+1}}\bigg]\bigg[\bigg(\int\limits_{-j}^{0}e^{\alpha\xi}\|\v_m(\xi)+\textbf{g}_m\y(\theta_{\xi}\omega)\|^{r+1}_{\wi\L^{r+1}(\mathcal{O}_m)}\d\xi\bigg)^{\frac{r-2}{r^2-1}}\nonumber\\&\ \ +\bigg(\int\limits_{-j}^{0}e^{\alpha\xi}\|\v(\xi)+\textbf{g}\y(\theta_{\xi}\omega)\|^{r+1}_{\wi\L^{r+1}(\mathcal{O})}\d\xi\bigg)^{\frac{r-2}{r^2-1}}\bigg]\bigg(\int\limits_{-j}^{0}\|\v_m(\xi)-\v(\xi)\|^{2}_{\L^2(\mathcal{O}_k)}\d\xi\bigg)^{\frac{1}{2(r-1)}}
	\nonumber\\&
	\to 0  \ \text{ as }\   m\to \infty, 
\end{align}
for $r\geq 3$. Applying H\"older's inequality in $I_{62}(m,j,k)$, we have
\begin{align}
	I_{62}(m,j,k)&\leq2\beta \bigg(\int_{-j}^{0}e^{\alpha\xi} \int_{\mathcal{O}_m\backslash\mathcal{O}_k}|\v_m(x,\xi)+\textbf{g}(x)\y(\theta_{\zeta}\omega)|^{r+1}\d x \d\xi\bigg)^{\frac{r}{r+1}}\nonumber\\&\qquad\times \bigg(\int_{\mathcal{O}_m\backslash\mathcal{O}_k}\big|\nabla\textbf{g}(x)\big|^{r+1}\d x\bigg)^{\frac{1}{r+1}} \bigg(\int_{-j}^{0}e^{\alpha\xi}\left|\y(\theta_{\xi}\omega)\right|^{r+1}\d\xi\bigg)^{\frac{1}{r+1}}.
\end{align}
Lemma \ref{largeradius1} and $\textbf{g}\in\V(\mathcal{O}_{\infty})\cap\H^2(\mathcal{O}_{\infty})$ imply that, for every $\varepsilon>0$, there exist some $j>0$ and $k>0$ such that
\begin{align}\label{SC16}
	I_{62}(m,j,k)\leq \frac{\varepsilon}{2}.
\end{align}
Similarly, with the help of Lemma \ref{largeradius1}, we find
\begin{align}\label{SC17}
	I_{63}(m,j,k)\leq \frac{\varepsilon}{2}.
\end{align}
From \eqref{SC13}-\eqref{SC17}, we obtain
\begin{align*}
	\lim_{m \to \infty}\left|I_6(m,j)-2\beta \int_{-j}^{0} e^{\alpha\xi}\left\langle\mathcal{C}_{\infty}(\v(\xi,-j;\omega,\v_{j,0})+\textbf{g}\y(\theta_{\xi}\omega)),\textbf{g}\y(\theta_{\xi}\omega)\right\rangle\d\xi\right|\leq\varepsilon.
\end{align*}
Since $\varepsilon$ is arbitrary, taking $\varepsilon\to0$, we deduce
\begin{align}\label{SC18}
	\lim_{m \to \infty}I_6(m,j)=2\beta \int_{-j}^{0} e^{\alpha\xi}\left\langle\mathcal{C}_{\infty}(\v(\xi,-j;\omega,\v_{j,0})+\textbf{g}\y(\theta_{\xi}\omega)),\textbf{g}\y(\theta_{\xi}\omega)\right\rangle\d\xi.
\end{align}
Combining \eqref{SC3}-\eqref{SC6}, \eqref{SC12} and \eqref{SC18}, we arrive at
\begin{align}
	\limsup_{m\to\infty}\|\v_{m}(0,-j;\omega,\v_{j,m,0})\|^2_{\H(\mathcal{O}_m)}\leq e^{-\alpha j}\mathscr{L}(\theta_{-j}\omega)+ \|\v(0,-j;\omega,\v_{j,0})\|^2_{\H(\mathcal{O})}. 
\end{align}
Since $j$ can be large enough, passing $j\to\infty$, we get 
\begin{align}\label{SC19}
	\limsup_{m\to\infty}\|\v_{m}(0,-j;\omega,\v_{j,m,0})\|^2_{\H(\mathcal{O}_m)}\leq \|\v(0,-j;\omega,\v_{j,0})\|^2_{\H(\mathcal{O})}. 
\end{align}
From \eqref{SC00}, we have 
\begin{align}\label{SC20}
	\|\v(0,-j;\omega,\v_{j,0})\|^2_{\H(\mathcal{O})}\leq \liminf_{m\to\infty}\|\v_{m}(0,-j;\omega,\v_{j,m,0})\|^2_{\H(\mathcal{O}_m)}. 
\end{align}
Now, \eqref{SC19} and \eqref{SC20} give that
\begin{align}\label{SC21}
	\widetilde{\v}_m(0,-j;\omega,\v_{j,m,0})\to \v(0,-j;\omega,\v_{j,0}) \text{ in } \H(\mathcal{O}),
\end{align}
which implies \eqref{CON4}, and it completes the proof.
\end{proof}

Now, we are ready to prove the main result of this section. The following Theorem proves the upper semicontinuity of random attractors with respect to domain.
	\begin{theorem}\label{Main-T}
		For $r>1$, the sequence $\{\mathcal{A}_m(\omega)\}_{m\in{\N}}$ of random attractors associated with the system \eqref{SCBF_m} satisfies
		\begin{align}\label{US-C}
			\lim_{m \to \infty}\emph{dist}_{\H(\mathcal{O})}\big(\widetilde{\mathcal{A}}_m(\omega), \mathcal{A}_{\infty}(\omega)\big)=0, \ \text{ for all }\ \omega\in\Omega,
		\end{align}
	where $\emph{dist}_{\H(\mathcal{O})}\big(\widetilde{\mathcal{A}}_m(\omega), \mathcal{A}_{\infty}(\omega)\big)$ = $\sup\limits_{\u\in\mathcal{A}_m}\emph{dist}_{\H(\mathcal{O})}\big(\widetilde{\u}, \mathcal{A}_{\infty}(\omega)\big)$ is the Hausdorff semidistance of the space $\H(\mathcal{O})$.
	\end{theorem}
	\begin{proof}
		We prove \eqref{US-C}  by contradiction. If \eqref{US-C} is not true, then there exists $\delta>0$ and a sequence $\u_{m}\in \mathcal{A}_{m}(\omega)$ ($m\in\N$) such that 
		\begin{align}\label{Dis1}
			\text{dist}_{\H(\mathcal{O})}\big(\widetilde{\u}_{m}, \mathcal{A}_{\infty}(\omega)\big)\geq \delta,  \ \ m\in\N.
		\end{align}  
	Using Lemma \ref{StrongC}, we can find a subsequence $\{\u_{m_k}\}$ of $\{\u_m\}$ such that $$\text{dist}_{\H(\mathcal{O})}\big(\widetilde{\u}_{m_k}, \mathcal{A}_{\infty}(\omega)\big)=0,$$ which is a contradiction to \eqref{Dis1} and it completes the proof.
	\end{proof}
\begin{remark}\label{r=1}
	In this section, we prove upper semicontinuity of random attractors with respect to domain for $r>1$ only.  For $r=1$, we need an additional assumption on $\textbf{g}(\cdot)$ (see Assumption 1.1, \cite{KM7}). Moreover, upper semicontinuity with respect to domain for 2D stochastic NSE (that is, $\alpha=0$ and $\beta=0$) as well as 2D SCBF equations (for $r=1$) can also be proved in similar way with a few changes in calculations.
\end{remark}
	\medskip\noindent
	{\bf Acknowledgments:}    The first author would like to thank the Council of Scientific $\&$ Industrial Research (CSIR), India for financial assistance (File No. 09/143(0938)/2019-EMR-I).  M. T. Mohan would  like to thank the Department of Science and Technology (DST), Govt of India for Innovation in Science Pursuit for Inspired Research (INSPIRE) Faculty Award (IFA17-MA110).

\end{document}